\documentclass[10pt, twoside]{article}
\usepackage[textwidth=13cm,textheight=21cm]{geometry}
\usepackage[all,cmtip]{xy}
\usepackage{comment}
\usepackage[pdftex]{pict2e}
\usepackage{amsmath,amssymb,amscd,amsthm}
\usepackage{a4wide,pstricks}
\usepackage{pdfsync}
\usepackage{fancyhdr}
\usepackage{makebox, bigstrut}
\usepackage[T1]{fontenc}
\usepackage[utf8]{inputenc} 
\usepackage[english]{babel} 
\usepackage{setspace}
\usepackage{bbm}
\usepackage{tasks}
\usepackage{dsfont} 
\usepackage{amsthm}
\usepackage{amsmath}
\usepackage{amssymb}
\usepackage{mathrsfs}
\usepackage{mathtools}
\usepackage{enumitem}
\usepackage{hyperref}
\usepackage{amssymb}
\usepackage{calligra}
\usepackage{ytableau}
\usepackage{marginnote}
\usepackage{cleveref}
\usepackage{babel}
\usepackage{bm}
\usepackage{tikz-cd}
\usepackage{faktor}
\usepackage{stmaryrd}
\renewcommand{\phi}{\varphi}

\let\emptyset\varnothing

\newcommand{\C}{\mathbb{C}}

\newcommand{\N}{\mathbb{N}}

\newcommand{\Q}{{\mathbb Q}}

\newcommand{\Z}{\mathbb{Z}}
\newcommand{\F}{\mathbb{F}}

\xyoption{all}
\input{xypic}

\newcommand{\G}{\mathbb{G}_m}

\theoremstyle{plain}
\numberwithin{equation}{subsection}
\crefformat{section}{\S#2#1#3} % see manual of cleveref, section 8.2.1
\crefformat{subsection}{\S#2#1#3}
\crefformat{subsubsection}{\S#2#1#3}
\let\oldmarginpar\marginpar
\renewcommand\marginpar[1]{\-\oldmarginpar[\raggedleft\footnotesize #1]
{\raggedright\footnotesize #1}}

\makeindex

\newtheorem{teorema}{Theorem}[subsection]
\newtheorem{prop}[teorema]{Proposition}

\newtheorem{conjecture}[teorema]{Conjecture}
\newtheorem{lemma}[teorema]{Lemma}

\theoremstyle{remark}
\newtheorem{oss}[teorema]{Remark}
\newtheorem{esempio}[teorema]{Example}

\theoremstyle{definition}
\newtheorem{definizione}[teorema]{Definition}

\newcounter{margin}
{\end{itshape}  \bigskip}

\DeclareMathOperator{\Hom}{Hom}

\DeclareMathOperator{\Mat}{Mat}

\DeclareMathOperator{\Res}{Res}
\DeclareMathOperator{\Ker}{Ker}

\DeclareMathOperator{\Gl}{GL}

\DeclareMathOperator{\Imm}{Im}
\DeclareMathOperator{\reg}{reg}

\DeclareMathOperator{\spec}{Spec}

\DeclareMathOperator{\Infl}{Infl}

\DeclareMathOperator{\Rep}{Rep}
\DeclareMathOperator{\Lie}{Lie}

\DeclareMathOperator{\Plexp}{Exp}
\DeclareMathOperator{\Plelog}{Log}

\DeclareMathOperator{\Ind}{Ind}

\DeclareMathOperator{\Coeff}{Coeff}

\DeclareMathOperator{\rank}{rank}

\DeclareMathOperator{\PGl}{PGL}
\DeclareMathOperator{\Cl}{Cl}
\DeclareMathOperator{\codim}{codim}
\begin{document}

\title{Cohomology of non-generic character stacks}

\author{ Tommaso Scognamiglio
\\ {\it Université Paris Cité/IMJ-PRG}
\\{\tt tommaso.scognamiglio@imj-prg.fr}
}

\pagestyle{myheadings}

\maketitle

\begin{abstract}
We study (compactly supported) cohomology of character stacks of punctured Riemann surface with prescribed semisimple local monodromies at punctures. In the case of generic local monodromies, the cohomology of these character stacks has been studied in \cite{HA}, \cite{mellit}. 

In this paper we extend the results and conjectures of \cite{HA} to the non-generic case. In particular we compute the E-series and give a conjectural formula for the mixed Poincar\'e series. We prove our conjecture  in the case of the projective line with $4$ punctures.

In the non-punctured case, formulas of a similar kind have already appeared in the recent work of Davison, Hennecart, Schlegel Mejia \cite{davison-hennecart}. Their approach involves-non abelian Hodge correspondence for stacks and BPS sheaves and does not extend immediately to the more general context of this paper.

Indeed, on the one side, BPS sheaves and their cohomology have not been studied in the punctured case.

On the other side, the authors \cite{davison-hennecart} use non-abelian Hodge correspondence, which is not algebraic in nature. This makes difficult to obtain information about the weight filtration on character stacks and in particular about their E-series.
\end{abstract}

\tableofcontents

\newpage

\section{Introduction}

Consider a Riemann surface $\Sigma$ of genus $g \geq 0$, a subset $D=\{p_1,\dots,p_k\} \subseteq \Sigma$ of $k$ points and $\mathcal{C}=(\mathcal{C}_1,\dots,\mathcal{C}_k)$ a $k$-tuple of semisimple conjugacy classes of $\Gl_n(\C)$. In what follows, denote by $\Gl_n$ the group $\Gl_n(\C)$.

The associated character stack $\mathcal{M}_{\mathcal{C}}$ is defined as the  quotient stack 
\begin{equation}
\label{intro0}
\mathcal{M}_{\mathcal
{C}}\coloneqq \left[\biggl\{\rho \in \Hom(\pi_1(\Sigma \setminus D),\Gl_n) \ | \ \rho(x_i) \in \mathcal{C}_i  \text{ for } i=1,\dots,k \biggr \}\middle /\Gl_n(\C)\right]
\end{equation}
where each $x_i$ is a small loop around the point $p_i$. These stacks classify local systems on $\Sigma\setminus D$ such that the 
monodromy around the point $p_i$ lies in $\mathcal{C}_i$, for $i=1,\dots,k$ and are naturally related to certain moduli spaces of (strongly) parabolic Higgs bundles on $\Sigma$ via the non-abelian Hodge correspondence, see for example the work of Simpson \cite{Simpson}.

\vspace{8 pt}

The stack $\mathcal{M}_{\mathcal{C}}$ has therefore the following explicit form in  terms of matrix equations:
\begin{equation}
\label{intro1}
\mathcal{M}_{\mathcal{C}}= \left[\biggl\{(A_1,B_1,\dots,A_g,B_g,X_1,\dots,X_k) \in \Gl^{2g}_n \times \prod_{j=1}^k \mathcal{C}_j | \ \prod_{i=1}^g [A_i,B_i]\prod_{j=1}^k X_j=1\biggr\}\middle/ \Gl_n\right].
\end{equation}

\vspace{10 pt}

 In what follows, for a complex stack of finite type $\mathcal{X}$, we will denote by $H^*_c(\mathcal{X})\coloneqq H^*_c(\mathcal{X},\C)$ its compactly supported cohomology with $\C$-coefficients. Recall that each vector space $H^i_c(\mathcal{X})$ is endowed with the \textit{weight filtration} $W^i_{\bullet}H^i_c(\mathcal{X})$. For more details about the definition of these cohomology groups and their weight filtration see \cref{cohocharstacks}.
 
We define the mixed Poincaré series $H_c(\mathcal{X},q,t)$ $$H_c(\mathcal{X},q,t)\coloneqq\sum_{m,i}\dim(W^i_m/W^i_{m-1})q^{\frac{m}{2}}t^i .$$

The E-series $E(\mathcal{X},q)$ is the specialization of $H_c(\mathcal{X},q,t)$ obtained by plugging $t=-1$ and the Poincaré series $P_c(\mathcal{X},t)$ is the specialization of $H_c(\mathcal{X},q,t)$ obtained by plugging $q=1$.

\subsection{Review of generic character stacks}
\label{review-intro}

The cohomology of character stacks and its mixed Hodge structure have been extensively studied from different perspectives.  So far, most of the results have been obtained in the case where the 
 $k$-tuple $\mathcal{C}$ is \textit{generic} (see \cref{genericktuplesepoly} for a precise definition of genericity).

 \vspace{10 pt}
 
 In the generic case, the stack $\mathcal{M}_{\mathcal{C}}$ is smooth and it is a $\mathbb{G}_m$-gerbe over the associated GIT quotient, which we denote by $M_{\mathcal{C}}$. Therefore the cohomology of $\mathcal{M}_{\mathcal{C}}$ can be easily deduced from that of the character variety $M_{\mathcal{C}}$. 

We give here a quick review of the results obtained about the cohomology of generic character stacks and varieties, see \cref{section8.2} for more details.

\vspace{8 pt}
The first results concerning this subject were obtained in the case where $k=1$ and $\mathcal{C}$ is a central conjugacy class. For $n \in \N$ and $d \in \Z$, let $\mathcal{M}_{n,d}$ be the stack $\mathcal{M}_{\mathcal{C}}$ for $k=1$ and $\mathcal{C}=\{e^{\frac{2 \pi i d}{n}}I_n\}$ i.e. $$\mathcal{M}_{n,d}=\left[\Biggl\{(A_1,B_1,\dots,A_g,B_g) \in \prod_{i=1}^{2g} \Gl_n \ | \ \prod_{i=1}^g[A_i,B_i]=e^{-\frac{2 \pi i d}{n}}I_n\Biggr\}\middle/\Gl_n \right].$$

The orbit $\mathcal{C}=\{e^{\frac{2 \pi i d}{n}}\}$ is generic if and only if $(n,d)=1$.

\vspace{8 pt}

Hitchin \cite{Hitchin87} computed the Poincaré polynomial $P_c(M_{n,d},t)$
 in the generic case for $n=2$, using non abelian Hodge correspondence and Morse theory on the moduli space of Higgs bundles. Gothen \cite{gothen}  extended his result for $n=3$.

Their approach was later extended to compute the Poincaré polynomial $P_c(M_{\mathcal{C}},t)$ in the case where $n=2$, any $k$ and any generic $k$-tuple $\mathcal{C}$ by Boden, Yogokawa \cite{boden-yoko} and where $n=3$, any $k$ and any generic $k$-tuple $\mathcal{C}$ by García-Prada, Gothen and Muñoz \cite{Garcia}.

However, Morse-theoretic techniques do not give information  about the weight filtration and were hard to generalize to any $n$.

\vspace{8 pt}

Hausel and Rodriguez-Villegas \cite{HRV} were the first to obtain a general result about cohomology of character stacks for any $n$. The authors computed the E-series $E(\mathcal{M}_{n,d},q)$ of the stacks $\mathcal{M}_{n,d}$ for any coprime $n,d$, by counting points over finite fields and proposed a conjectural formula  for the mixed Poincaré series $H_c(\mathcal{M}_{n,d},q,t)$. 

Schiffmann \cite{schiffmann} found an expression for the Poincaré series $P_c(\mathcal{M}_{n,d},t)$ in the generic case and Mellit \cite{mellit-nopunctures} later checked that Schiffmann's formula agrees with the specialization of Hausel and Rodriguez-Villegas' conjecture at $q=1$.

\vspace{8 pt}

Hausel, Letellier and Rodriguez-Villegas afterwards generalized the results of \cite{HRV} and computed \cite[Theorem 1.2.3]{HA}  the E-series $E(\mathcal{M}_{\mathcal{C}},q)$ of the stacks $\mathcal{M}_{\mathcal{C}}$ for any generic $k$-tuple $\mathcal{C}$. We quickly explain more precisely their result since it is the starting point of this paper. For more details, see \cref{genericktuplesepoly}.

\vspace{8 pt}

 Let $\mathcal{P}$ be the set of partitions. In \cite{HA}, the authors introduced, for each multipartition $\bm \mu \in \mathcal{P}^k$, a rational function $\mathbb{H}_{\bm \mu}(z,w) \in \Q(z,w)$, defined in terms of Macdonald polynomials (for a precise definition see \cref{genericktuplesepoly}). The authors  \cite[Theorem 1.2.3]{HA} showed that there is an equality
\begin{equation}
\label{Epolynomialgenericintro}
E(\mathcal{M}_{\mathcal{C}},q)=\dfrac{q^{\frac{d_{\bm \mu}}{2}}}{q-1}\mathbb{H}_{\bm \mu}\left(\sqrt{q},\dfrac{1}{\sqrt{q}} \right)
\end{equation}
where $2d_{\bm \mu}=\dim (\mathcal{M}_{\mathcal{C}})+1$ and $\bm \mu=( \mu^1,\dots, \mu^k)$ is the multipartition given by the multiplicities of the eigenvalues of $\mathcal{C}_1,\dots,\mathcal{C}_k$ respectively.

In the same paper, the authors \cite[Conjecture 1.2.1]{HA} proposed the following conjectural formula for the  mixed Poincaré series $H_c(\mathcal{M}_{\mathcal{C}},q,t)$, which generalizes Hausel and Rodriguez-Villegas' conjecture stated in \cite{HRV} and naturally deformes 
Identity (\ref{Epolynomialgenericintro}):
\begin{equation}
\label{mhpgenericintro}
    H_c(\mathcal{M}_{\mathcal{C}},q,t)=\dfrac{(qt^2 )^{\frac{d_{\mu}}{2}}}{qt^2-1}\mathbb{H}_{\bm \mu}\left(-t\sqrt{q},\dfrac{1}{\sqrt{q}}, \right).
\end{equation}

\vspace{8 pt}

Mellit \cite[Theorem 7.12]{mellit} later computed the Poincaré series $P_c(\mathcal{M}_{\mathcal{C}},t)$ using the non-abelian Hodge correspondence. His formula matches with the specialization at $q=1$ of the conjectural formula (\ref{mhpgenericintro}) for the mixed Poincaré series.

\vspace{8 pt}

\subsection{Overview of the paper}
\label{overview}
One of the aims of this paper is the study of the  cohomological invariants of the stacks $\mathcal{M}_{\mathcal{C}}$ in the case where $\mathcal{C}$ is not necessarily generic. The  cohomology of character stacks $\mathcal{M}_{\mathcal{C}}$ for  non-generic $k$-tuples  $\mathcal{C}$ has not been much studied  in the literature until recently.  

\vspace{8 pt}

The most explicit and general results in the non-generic case were obtained mainly  for the stacks $\mathcal{M}_{n,d}$.

Hausel and Rodriguez-Villegas \cite[Theorem 3.8.1]{HRV}  expressed the E-series for the stacks $\mathcal{M}_{n,0}$ in terms of the E-series for the generic character stacks $\mathcal{M}_{n,1}$ by the following formula:

 \begin{equation}
\label{HRV}
\Plexp\left(\sum_{n \in \N} \dfrac{E(\mathcal{M}_{n,1},q)}{q^{n^2(g-1)}} T^n\right)=\sum_{n \in \N} \dfrac{E(\mathcal{M}_{n,0},q)}{q^{n^2(g-1)}} T^n
\end{equation}

where $\Plexp$ is the plethystic exponential in the ring of formal power series $\Q(q)[[T]]$ (see \cref{lambdarings} for details about plethystic operations). The authors' result is obtained by counting points over finite fields.  

\vspace{8 pt}

Fix now $r \in \Q$. Recently, Davison, Hennecart and Schlegel Mejia \cite[Theorem 14.3, Corollary 14.7]{davison-hennecart} proved the following formula expressing the compactly supported Poincaré series of $\mathcal{M}_{n,d}$ for any $n,d$, in terms of the Poincaré series for the generic character stacks $\mathcal{M}_{n,1}$:

\begin{equation}
\label{DHM}
\sum_{\substack{(n,d) \in \N_{>0} \times \Z \\ d=rn}} \dfrac{P_c(\mathcal{M}_{n,d},-t)}{t^{n^2(2g-2)}}z^nw^d=\Plexp\left(\sum_{\substack{(n,d) \in \N_{>0} \times \Z \\ d=rn}}\dfrac{P_c(\mathcal{M}_{n,1},-t)}{t^{n^2(2g-2)}}z^n w^d\right)
\end{equation}

and formulated a similar conjecture for the mixed Poincaré series of $H_c(\mathcal{M}_{n,d},q,t)$ for any $n,d$ (see the discussion after \cite[Theorem 14.10]{davison-hennecart}). 

They obtained this formula by relating the cohomology of a character stack with the cohomology of the so-called BPS sheaves. The latter are certain perverse sheaves defined on character varieties and are well-understood for the stacks $\mathcal{M}_{n,d}$. More precisely the non-abelian Hodge correspondence for stacks, proved in \cite{davison-hennecart}, and the recent work of Koseki and Kinjo \cite{kinjo2021cohomological} about BPS sheaves for the moduli stack of Higgs bundles, give a way to compute the cohomology of BPS sheaves for a stack $\mathcal{M}_{n,d}$.

However notice that, since the authors use non-abelian Hodge correspondence which does not preserve the weight filtration on cohomology, their method does not allow to prove an analogous formula for the E-series or the mixed Poincaré series of $\mathcal{M}_{n,d}$.  

\vspace{8 pt}

Finally, the cohomology of BPS sheaves for character stacks $\mathcal{M}_{\mathcal{C}}$ is not understood for an arbitrary $\mathcal{C}$ and so a generalisation of  Formula (\ref{DHM}) for an arbitrary $\mathcal{C}$ is still unproved.

One of the main results  of our paper is a generalization of (\ref{DHM}) to arbitrary $\mathcal{C}$ for the E-series $E(\mathcal{M}_{\mathcal{C}},q)$ instead of the Poincar\'e series $P_c(\mathcal{M}_\mathcal{C},t)$. As a result we get an explicit formula for $E(\mathcal{M}_{\mathcal{C}},q)$  for  any $k$-tuple
 $\mathcal{C}$, see Theorem \ref{mainteointro} below and \cref{mainresultparagraph}.

\vspace{8 pt}

We also give a conjectural formula (see Conjecture \ref{conjmhpintro}) for $H_c(\mathcal{M}_{\mathcal{C}},q,t)$, which we verify  in the case of $\Sigma=\mathbb{P}^1_{\C}$, $|D|=4$ and a certain family of non-generic quadruples, for more details see \cref{chaptermhpnongene}. 

Conjecture \ref{conjmhpintro} for the stacks $\mathcal{M}_{n,d}$ has already appeared in \cite{davison-hennecart}, see the discussion there after Theorem $14.10$.

\vspace{8 pt}

Finally let us notice that our approach is very different from that of \cite{davison-hennecart} as we do not use non-abelian Hodge theory nor BPS sheaves.

\subsection{Main results}
\label{mainresult} 
An important tool to formulate and prove the main results of this paper is the construction of character stacks as multiplicative quiver stacks, as first introduced by 
Crawley-Boevey and Shaw \cite{cb-monod},\cite{cb-shaw}, which we quickly recall here (see \cref{quiverrep} for more details). 

Notice that this construction is not needed for studying generic character stacks and does not appear for example in the articles \cite{HA}, \cite{mellit}. However, it is a key point in our paper, as it allows to distinguish between different levels of non-genericity for non-generic character stacks.

\vspace{8 pt}

Let $s_1,\dots,s_k \in \N$ be such that, for each $i=1,\dots,k$, the conjugacy class $\mathcal{C}_i$ has $s_i+1$ distinct eigenvalues $\gamma_{i,0},\dots,\gamma_{i,s_i}$ with multiplicities $m_{i,0},\dots,m_{i,s_i}$ respectively. Let $Q=(I,\Omega)$ be the following star-shaped quiver with $g$ loops on the central vertex 
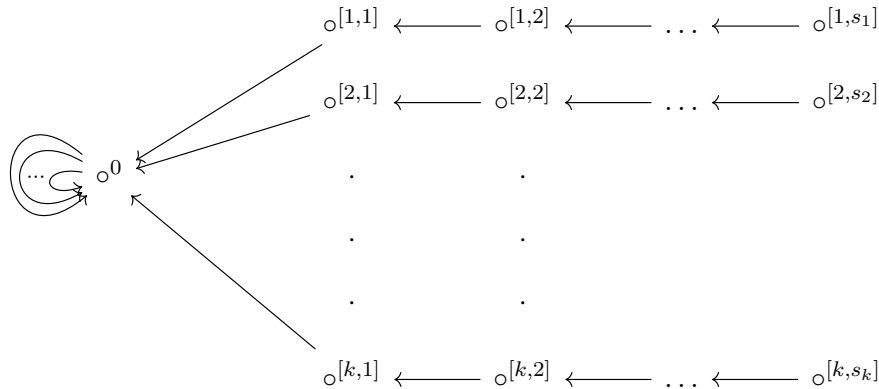
\begin{center}
    \begin{tikzcd}[row sep=1em,column sep=3em]
    & &\circ^{[1,1]} \arrow[ddll,""] &\circ^{[1,2]} \arrow[l,""]  &\dots \arrow[l,""] &\circ^{[1,s_1]} \arrow[l,""]\\
    & &\circ^{[2,1]} \arrow[dll,""] &\circ^{[2,2]} \arrow[l,""] &\dots \arrow[l,""] &\circ^{[2,s_2]} \arrow[l,""]\\
    \circ^0 \arrow[out=170,in=200,loop,swap] \arrow[out=150,in=210,loop,"\cdots"] \arrow[out=140,in=220,loop,swap]  & &\cdot &\cdot\\
    & &\cdot &\cdot\\
    & &\cdot &\cdot\\
    & &\circ^{[k,1]} \arrow[uuull,""] &\circ^{[k,2]} \arrow[l,""]  &\dots \arrow[l,""] &\circ^{[k,s_k]} \arrow[l,""]
    \end{tikzcd}
\end{center}
\vspace{8 pt}

Recall that for any $\beta \in \N^I$, there is a representation variety $R(\overline{Q},\beta)^{\circ,\ast}$ and a multiplicative moment map $$\Phi^{\ast}_{\beta}:R(\overline{Q},\beta)^{\circ,\ast} \to \Gl_{\beta}\coloneqq \prod_{i \in I}\Gl_{\beta_i} .$$ For any $s \in (\C^*)^I$, we denote by $s$ the central element $s\coloneqq (s_iI_{\beta_i})_{i \in I} \in \Gl_{\beta}$. The multiplicative quiver stack with parameters $\beta,s$ is the quotient stack  $$\mathcal{M}_{s,\beta}^{\ast}\coloneqq [(\Phi^{\ast}_{\beta})^{-1}(s)/\Gl_{\beta}] .$$

\vspace{8 pt}

Consider now the dimension vector $\alpha_{\mathcal{C}} \in \N^I$ defined as $$(\alpha_{\mathcal{C}})_{[i,j]}=\sum_{h=j}^{s_i}m_{i,h} $$ for every $j=0,\dots,s_i$, where we are identifiying $[i,0]=0$ for each $i=1,\dots,k$. Notice that $(\alpha_{\mathcal{C}})_0=n$. Let moreover $\gamma_{\mathcal{C}} \in (\C^*)^I$ be defined as follows $$\displaystyle (\gamma_{\mathcal{C}})_{[i,j]}=\begin{cases} \displaystyle\prod_{i=1}^k\gamma^{-1}_{i,0} \text{ if } j=0\\
\gamma^{-1}_{i,j}\gamma_{i,j-1} \text{ otherwise }

\end{cases}.$$

\vspace{8 pt}

In Theorem \ref{theomultquiverstack}, we show that, for the elements $\alpha_{\mathcal{C}},\gamma_{\mathcal{C}}$, there is an isomorphism of stacks
\begin{equation}
\label{eqisomintro}
\mathcal{M}_{\mathcal{C}} \cong \mathcal{M}^{\ast}_{\gamma_{\mathcal{C}},\alpha_{\mathcal{C}}} . 
\end{equation}
\vspace{8 pt}

Let now $(\N^I)^* \subseteq \N^I$ be the subset of vectors with non-increasing coordinates along the legs and denote by $\mathcal{H}^*_{\gamma_{\mathcal{C}},\alpha_{\mathcal{C}}} \subseteq (\N^I)^*$ the subset defined as $$\mathcal{H}_{\gamma_{\mathcal{C}},\alpha_{\mathcal{C}}}^*=\{\delta \in (\N^I)^* \ | \ \gamma_{\mathcal{C}}^{\delta}=1\ \text{ and } \delta \leq \alpha_{\mathcal{C}} \} $$

where $\gamma_{\mathcal{C}}^{\delta}=\displaystyle \prod_{i \in I}(\gamma_{\mathcal{C}})_i^{\delta_i}$.

\vspace{8 pt}

\begin{esempio}
It can be checked that  if  $\mathcal{H}^*_{\gamma_{\mathcal{C}},\alpha_{\mathcal{C}}}=\{\alpha_{\mathcal{C}}\}$ then the $k$-tuple $\mathcal{C}$ is \textit{generic}.

\end{esempio}

\vspace{8 pt}

 The subsets $\mathcal{H}^*_{\gamma_{\mathcal{C}},\alpha_{\mathcal{C}}}$ allow to define a natural stratification on the set of $k$-tuples $\mathcal{C}$ and so of character stacks $\mathcal{M}_{\mathcal{C}}$.

The introduction of this stratification is one of the key ingredients to study the cohomology of $\mathcal{M}_{\mathcal{C}}$ in the non-generic case.

Notice that although not explicitly defined, the subsets $\mathcal{H}^*_{\gamma_{\mathcal{C}},\alpha_{\mathcal{C}}}$ appear implicitly in  \cite{davison-hennecart}.

\subsubsection{E-series of character stacks}

For any $\beta \in (\N^I)^*$ and for any $j=1,\dots,k$, the integers $(\beta_{[j,0]}-\beta_{[j,1]},\dots,\beta_{[j,s_j-1]}-\beta_{[j,s_j]},\beta_{[j,s_j]})$ up to reordering form a partition $\mu_{\beta}^j \in \mathcal{P}$. Denote by $\bm \mu_{\beta} \in \mathcal{P}^k$ the multipartition $\bm \mu_{\beta}=( \mu_{\beta}^1,\dots, \mu_{\beta}^k)$ and by $\mathbb{H}_{\beta}(z,w)$ the function $\mathbb{H}_{\bm \mu_{\beta}}(z,w)$. 

\vspace{8 pt}

\begin{oss}
For a $k$-tuple $\mathcal{C}$, the multipartition $\bm \mu_{\alpha_{\mathcal{C}}} \in \mathcal{P}^k$ is the multipartition given by the multiplicities of the orbits $\mathcal{C}_1,\dots,\mathcal{C}_k$ respectively. 

 Moreover, it can be checked that $\dim(\mathcal{M}_{\mathcal{C}})=-2(\alpha_{\mathcal{C}},\alpha_{\mathcal{C}})+1$, where $(,)$ is the Euler form of $Q$. The result \cite[Theorem 1.2.3]{HA} of Hausel, Letellier, Rodriguez-Villegas for a generic $k$-tuple $\mathcal{C}$ can then be rewritten as follows:

\begin{equation}
\label{HLRVresultgenericintro}
\dfrac{E(\mathcal{M}_{\mathcal{C}},q)}{q^{-(\alpha_{\mathcal{C}},\alpha_{\mathcal{C}})}}=\dfrac{q\mathbb{H}_{\alpha_{\mathcal{C}}}\left(\sqrt{q},\frac{1}{\sqrt{q}}\right)}{q-1}.
\end{equation}
\end{oss}

\vspace{12 pt}

The main result about character stacks of this paper (see Theorem \ref{Epolynomialtheorem}) is the following Theorem:
\begin{teorema}
    \label{mainteointro}
For any $k$-tuple of semisimple orbits $\mathcal{C}$, it holds:
\begin{equation}
\label{maineqintro}
\Coeff_{\alpha_{\mathcal{C}}}\left(\Plexp\left(\sum_{\substack{\beta \in \mathcal{H}_{\gamma_{\mathcal{C}},\alpha_{\mathcal{C}}}^*}}\dfrac{q\mathbb{H}_{\beta}\left(\sqrt{q},\frac{1}{\sqrt{q}}\right)}{q-1}y^{\beta} \right)\right)=\dfrac{E(\mathcal{M}_{\mathcal{C}},q)}{q^{-(\alpha_{\mathcal{C}},\alpha_{\mathcal{C}})}}.
\end{equation}

\end{teorema}

\vspace{10 pt}

\begin{oss}
One of the interesting aspects of Theorem \ref{mainteointro} is that it expresses the E-series $E(\mathcal{M}_{\mathcal{C}},q)$ for any $k$-tuple $\mathcal{C}$ in terms of the functions $\mathbb{H}_{\beta}\left(\sqrt{q},\frac{1}{\sqrt{q}}\right)$, i.e. in terms of the E-series for generic $k$-tuples. 

Similar type of results, relating non-generic to generic, have already appeared elsewhere, see for example the discussion in \cref{overview} and also Letellier's paper \cite{letellierchar}. 

\end{oss}

In Chapter \cref{cohocharstacks}, we compute the E-series of the complex character stacks $\mathcal{M}_{\mathcal{C}}$ through the approach introduced in \cite{HRV},\cite{HA},\cite{NonOr}, i.e. by reduction to finite fields and point counting. Namely, recall that if  there exists a rational function $Q(t) \in \Q(t)$ such that, for any $\F_q$-stack $\mathcal{M}_{\mathcal{C},\F_q}$ obtained from $\mathcal{M}_{\mathcal{C}}$ by base change and any $m$, it holds $$\#\mathcal{M}_{\mathcal{C},\F_q}(\F_{q^m})=Q(q^m) ,$$ we have an equality
 $$E(\mathcal{M}_{\mathcal{C}},q)=Q(q), $$ for more details see \cref{finitefiledsandcoho}. 

 However, the way we count rational points of character stacks in this paper is quite different from that of \cite{HA}. The description of the rational functions $Q(t)$ for non-generic character stacks is given through the results of Chapter \cref{plethysticidentitiesconvolution} and \cref{chaptercharstack}, where we show how to compute the rational points of a multiplicative quiver stack for a star-shaped quiver over $\F_q$. 
 
 In the articles \cite{AH},\cite{HA}, star-shaped quivers had already been introduced and used to relate character varieties and \textit{quiver varieties} for star-shaped quivers. However, the authors did not use multiplicative quiver stacks  and in particular they didn't need the isomorphism (\ref{eqisomintro}) to compute $\F_q$-points of generic character stacks. 

\vspace{6 pt}

The results of \cref{plethysticidentitiesconvolution} and \cref{chaptercharstack} about the rational functions $Q(t)$ will be obtained as a consequence of the main technical result of this paper which is Theorem \ref{mainteo}. The latter theorem is very general and works for certain families of rational functions called \emph{Log compatible}.

Therefore, to prove Theorem \ref{mainteointro} we will have to prove that the rational functions involved in it  satisfy this  Log compatibility property. 

Theorem \ref{mainteo} can be used to  prove similar formulas in a different context : quiver representations and multiplicities in tensor product of irreducible characters of ${\rm GL}_n(\mathbb{F}_q)$, which are the main results of \cite{scognamiglio1}.  However, in the latter article we could avoid using the technical  Theorem \ref{mainteo}, using a more categorical approach instead.

\vspace{6 pt}

The proof of Theorem \ref{mainteo} will be one of the main technical results of the first part of this paper. We will have to use combinatorial objects different from the ones used for the generic case in \cite{AH},\cite{HA}. They we will be introduced in Chapters \cref{subtoriandmultitypes},\cref{plethysticidentities}.

\subsubsection{Conjecture for mixed Poincaré series of character stacks}

 Hausel, Letellier, Rodriguez Villegas conjectural formula (\ref{mhpgenericintro}) for the mixed Poincaré series of character stacks for generic $k$-tuples and Theorem \ref{mainteointro} suggest the following conjecture for the mixed Poincaré series of character stacks. For more details see \cref{chaptermhpnongene}.
 
\begin{conjecture}
\label{conjmhpintro}
 For any $k$-tuple of semisimple orbits $\mathcal{C}$, it holds:
\begin{equation}
\label{mhsconjnongene11}
\Coeff_{\alpha_{\mathcal{C}}}\left(\Plexp\left(\sum_{\substack{\beta \in \mathcal{H}^*_{\gamma_{\mathcal{C}},\alpha_{\mathcal{C}}}}} \dfrac{(qt^2)\mathbb{H}_{\beta}\left(t\sqrt{q},\frac{1}{\sqrt{q}}\right)}{qt^2-1}y^{\beta}\right)\right)=\dfrac{H_c(\mathcal{M}_{\mathcal{C}},q,-t)}{(qt^2)^{-(\alpha_{\mathcal{C}},\alpha_{\mathcal{C}})}}.
\end{equation}
\end{conjecture}

\vspace{10 pt}

\begin{oss}
The presence of the sign $-$ in $H_c(\mathcal{M}_{\mathcal{C}},q,-t)$ in eq.(\ref{conjmhpintro}) is due to the combinatorial properties of the plethystic exponential $\Plexp$, see for example the discussion in \cite[Paragraph 4.3]{davisonpreprojective}. 
\end{oss}

\vspace{4 pt}

In c\ref{chaptermhpnongene}, we verify that Conjecture \ref{conjmhpintro} holds in the case of $\Sigma=\mathbb{P}^1_{\C}$, $|D|=4$ and the following family of non-generic quadruples. 

Pick $\lambda_1,\lambda_2,\lambda_3,\lambda_4 \in \C^*\setminus\{1,-1\}$ and denote by $\mathcal{C}_j$ the conjugacy class of the diagonal matrix $$\begin{pmatrix}
\lambda_j &0\\
0 &\lambda_j^{-1}
\end{pmatrix} .$$

Assume moreover that $\lambda_1,\lambda_2,\lambda_3,\lambda_4$ have the following property. Given $\epsilon_1,\dots,\epsilon_4 \in \{1,-1\}$ such that $\lambda_1^{\epsilon_1}\cdots\lambda_4^{\epsilon_4}=1$, then either $\epsilon_1=\cdots=\epsilon_4=1$ or $\epsilon_1=\cdots=\epsilon_4=-1$.

\vspace{8 pt}

\subsection{Final remark}

In the generic case \cite{HA}, the conjectural formula for the mixed Poincar\'e series is written in the language of symmetric functions. More precisely, let $\Lambda_k$ the ring of functions with values in the  rational functions $\Q(z,w)$ and separately symmetric in $k$ sets of infinite variables $\mathbf{x_1},\dots,\mathbf{x_k}$. Then the function $\mathbb{H}_{\bm \mu}(z,w)$ mentioned in \cref{review-intro} is defined as

 $$\mathbb{H}_{\bm \mu}(z,w) \coloneqq(z^2-1)(1-w^2) \langle \Coeff_{T^n}(\Plelog(\Omega(z,w))),h_{\bm \mu} \rangle $$
where $\Omega(z,w) \in \Lambda_k[[T]]$ (see \cite[Paragraph 2.3.6]{HA} for a definition), $h_{\bm \mu}=h_{\mu^1}(\mathbf{x_1}) \cdots h_{\mu^k}(\mathbf{x_k})$ is the complete symmetric functions and $\langle, \rangle$ is the natural extension to $\Lambda_k$ of the bilinear product on symmetric functions making Schur functions orthonormal.
\bigskip

One of the advantages of this approach is that the cohomology of every generic character stack is encoded in the single object $\Omega(z,w)$. Moreover, this allows to immediately  generalize the results and the conjectures to the cohomology of generic character stacks with non necessary semisimple local monodromies by replacing the complete symmetric functions in the pairing by other symmetric functions (for instance it would be Schur symmetric functions for conjugacy classes which are the product of a unipotent conjugacy class with a central element, see for example \cite{letellierchar}). 

\vspace{8 pt}

It would be interesting to "symmetrize" our main formulas and conjectures.

\paragraph{Acknowledgements.}

The author is very grateful to Ben Davison, Emmanuel Letellier, Fernando Rodriguez-Villegas and Olivier Schiffmann for many useful discussions about the topics dealt in this paper.

The author would also like to thank the anonymous referee for many useful comments about a first draft of this paper. This article is a part of the author's PhD thesis.

\section{Preliminaries and notations}
\label{chaptermultitypes}
In this Section, $I$ will be a fixed finite set. In the cases relevant to the main result of this paper on non-generic character stacks, $I$ will be the set of vertices of a star-shaped quiver.

In \cref{multitypes} and \cref{lambdarings}, we will introduce and recall the properties of some combinatorial objects defined in terms of $I$, such as the multitypes and an associated ring, which will be a key technical point of this paper and were not needed in the works about the generic case \cite{AH},\cite{HA}.

\vspace{6 pt}
We endow the set $\N^I$ with the partial order $\leq$ such that, for $\alpha,\beta \in \N^I$, we have $\beta \geq \alpha$ if and only if $\beta_i \geq \alpha_i$ for each $i \in I$.

\subsection{Partitions and multipartitions}

Let $\mathcal{P}$ be the set of all partitions and $\mathcal{P}^* \subseteq \mathcal{P}$ the subset of nonzero partitions. A partition $\lambda$ will be denoted by $\lambda=(\lambda_1,\lambda_2 \dots ,\lambda_h)$ with $\lambda_1 \geq \lambda_2 \geq \dots \geq \lambda_h$ or  by $\lambda=(1^{m_1},2^{m_2},\dots )$ where $m_k$ is the number of occurrances of the number $k$ in the partition $\lambda$. We will denote by $\lambda'$ the partition conjugate to $\lambda$.

The \textit{size} of $\lambda$ is $\displaystyle |\lambda|=\sum_{i} \lambda_i$ and its length $l(\lambda)$ is the largest $i$ such that $\lambda_i \neq 0$. For each $n \in \N$, we denote by $\mathcal{P}_n$ the subset of partitions of size $n$. For two partitions $\lambda,\mu$, we denote by $\langle \lambda,\mu \rangle$ the quantity $$\langle \lambda,\mu \rangle=\sum_{i}\lambda_i'\mu_i' $$ and by $$n(\lambda)=\sum_i (i-1)\lambda_i .$$ We have that $$\langle \lambda, \lambda \rangle=2 n(\lambda)+|\lambda| .$$

Recall that the set $\mathcal{P}$ admits different possible orderings. In the following, we will denote by  $\lambda \leq \mu$ the ordering induced by the \textit{lexicographic} order.

\vspace{8 pt}

The conjugacy classes of the symmetric group $S_n$ are indexed by the partitions  $\mathcal{P}_n$. For each $\lambda \in \mathcal{P}_n$, denote by $z_{\lambda}$ the cardinality of the centralizer of an element of the conjugacy class associated to $\lambda$. If $\lambda=(1^{m_1},2^{m_2},\dots,)$, we have  $$z_{\lambda}=\prod_j j! j^{m_j} .$$

Moreover, recall that the set of irreducible characters of  $S_n$ is in bijection with $\mathcal{P}_n$. In our bijection we associate to the partition $(n)$ the trivial character of $S_n$. We denote the irreducible character of $S_n$ associated to $\lambda$ by $\chi^{\lambda}$.

For any three  partitions $\lambda \in \mathcal{P}_n, \mu \in \mathcal{P}_m, \nu \in \mathcal{P}_{n+m}$, we denote by $c_{\lambda,\mu}^{\nu}$ the integer $$c_{\lambda,\mu}^{\nu}\coloneqq\langle \chi^{\nu},\Ind_{S_n \times S_m}^{S_{n+m}}(\chi^{\lambda} \boxtimes \chi^{\mu}) \rangle .$$

\vspace{12 pt}

Consider also the set of multipartitions $\mathcal{P}^I$. The elements of $\mathcal{P}^I$ will be usually be denoted in bold letters $\bm \lambda \in \mathcal{P}^I$. To avoid confusion with the notation used for partitions, we will use the notation $\bm \lambda=( \lambda^i )_{i \in I}$. For $\bm \lambda \in \mathcal{P}^I$, the \textit{size}  $|\bm \lambda| \in \N^I$ of $\bm \lambda$ is defined as $$|\bm \lambda|_i\coloneqq | \lambda^i| .$$

We also put $n(\bm \lambda)\coloneqq \sum_{i \in I}n(\lambda^i)$.

\vspace{2 pt}

For an element $\alpha \in \N^I$, we will denote by $(1^{\alpha}) \in \mathcal{P}^I$ the multipartition $((1^{\alpha_i}))_{i \in I}$ and by $(\alpha) \in \mathcal{P}^I$ the multipartition $((\alpha_i))_{i \in I}$.

The order $\leq$ on $\mathcal{P}$ induces an ordering on $\mathcal{P}^I$ through the lexicographical ordering, which we still denote by $\leq$.

\subsection{ Multitypes}
\label{multitypes}

A multitype is a function $\omega: \N \times \mathcal{P}^I \to \N$ such that its support (i.e. the elements $(d,\bm\mu)$ such that $\omega(d,\bm\mu) \neq 0$) is finite and $\omega(0,\bm\lambda)=\omega(d,0)=0$ for any $\bm\lambda \in \mathcal{P}^I$ and $d \in \N$.

On the set $\N \times \mathcal{P}^I$ put the total order defined as follows. If $d > d'$ then $(d,\bm\lambda) > (d',\bm\mu)$, if $|\bm\lambda| > |\bm\mu|$ then $(d,\bm\lambda) > (d',\bm\mu)$ and if $|\bm\lambda|=|\bm\mu|$ then $(d,\bm\lambda) > (d',\bm\mu)$ if $\bm\lambda > \bm\mu$.

We can alternately think of a multitype $\omega$ as a non-decreasing sequence $\omega=(d_1,\bm\lambda_1)\dots (d_r,\bm\lambda_r)$, where the value $\omega(d,\bm \lambda)$ corresponds to the number of times the element $(d,\bm\lambda)$ appear in the sequence $(d_1,\bm\lambda_1)\dots (d_r,\bm\lambda_r)$.

 We will denote by $\mathbb{T}_I$  the set of multitypes. If $|I|=1$, we call multitypes simply types.

\vspace{8 pt}
 
 For $\omega \in \mathbb{T}_I$ where $\omega=(d_1,\bm\lambda_1)\dots (d_r,\bm\lambda_r)$ and $d \in \N_{>0}$, we denote by $\psi_d( \omega)$ the multitype $$\psi_d( \omega) \coloneqq (dd_1,\bm\lambda_1)\dots (dd_r,\bm\lambda_r) .$$

The \textit{size} $|\omega| $ of a multitype $\omega$ is the following element of $\N^I$ $$\displaystyle |\omega|\coloneqq \sum_{(d,\bm \mu) \in \N \times \mathcal{P}^I} d \omega(d,\bm \mu) |\bm \mu| $$ and  $w(\omega)$ is the quantity $$w(\omega) \coloneqq \prod_{(d,\bm\mu)\in \N \times \mathcal{P}^I}d^{\omega(d,\bm \mu)}\omega(d,\bm \mu)! .$$ For $\alpha \in \N^I$, we denote by $\mathbb{T}_{\alpha} \subseteq \mathbb{T}_I$ the subset of multitypes of size  $\alpha$.

\vspace{8 pt}

   The sum of multitypes endows the set $\mathbb{T}_I$ with an associative operation $\ast: \mathbb{T}_I \times \mathbb{T}_I \to \mathbb{T}_I$. More precisely, for $\omega_1,\omega_2 \in \mathbb{T}_I$, we define $\omega_1 \ast \omega_2 $ as the multitype such that $$\omega_1 \ast \omega_2(d,\bm \lambda) \coloneqq \omega_1(d,\bm \lambda)+\omega_2(d,\bm \lambda) .$$ The reason for this choice of notation will be clear in \cref{paragraphring}, where we will use the sum of multitypes to define an associative ring.

\vspace{8 pt}

We view $\N \times \N^I$ as a subset of $\N \times \mathcal{P}^I$ by associating to $(d,\alpha)$ the element $(d,(1^{\alpha}))$. We call a multitype $\omega$ \textit{semisimple} if its support is contained in $\N \times \N^I$. Given a semisimple $\omega$ we will see it as a function $ \N \times \N^I \to \N$ which we still denote by $\omega$ where $$\omega(d,\alpha)\coloneqq\omega(d,(1^{\alpha})) .$$ Whenever the context is clear, we will frequently switch between the two notations for semisimple multitypes.

\vspace{8 pt}

For each $\alpha \in \N^I$, we denote by $\omega_{\alpha}$ the semisimple multitype such that $\omega_{\alpha}(1,\alpha)=1$ and $\omega_{\alpha}(d,\beta)=0$  for every other element  $(d,\beta) \in \N \times \N^I$.

 We denote by $\mathbb{T}_I^{ss}\subseteq \mathbb{T}_I$ the subset of semisimple multitypes. Notice that for any semisimple multitype $\omega \in \mathbb{T}_I^{ss}$, there exist $d_1,\dots,d_r \in \N_{>0}$ and $\alpha_1,\dots,\alpha_r \in \N^I$ such that $$\omega=\psi_{d_1}(\omega_{\alpha_1}) \ast \cdots \ast \psi_{d_r}(\omega_{\alpha_r}) .$$

For a multitype $\omega=(d_1,\bm\lambda_1)\dots (d_r,\bm\lambda_r) \in \mathbb{T}_I$, we define its semisimplification $\omega^{ss} \in \mathbb{T}_I^{ss}$ as the following semisimple multitype $$\omega^{ss} \coloneqq (d_1,(1^{|\bm\lambda_1|}))\dots (d_r,(1^{|\bm \lambda_r|})) ,$$ i.e. $\omega^{ss}=\psi_{d_1}(\omega_{|\bm \lambda_1|}) \ast \cdots \ast \psi_{d_r}(\omega_{|\bm \lambda_r|})$.

\vspace{8 pt}

To a semisimple multitype $\omega=(d_1,\alpha_1)\dots (d_r,\alpha_r)$, we associate the following polynomial $P_{\omega}(t) \in \Z[t]$ $$P_{\omega}(t)\coloneqq\prod_{j=1}^r (t^{d_j}-1) .$$

Notice that for  any $\alpha \in \N^I$, we have $$P_{\omega_{\alpha}}(t)=t-1.$$

\vspace{6 pt}

For a semisimple multitype $\omega=(d_1,\beta_1)\cdots(d_r,\beta_r)\in\mathbb{T}^{ss}_{\alpha}$  put
$$
C_{\omega}^o:=\begin{cases}
  \mu(d)d^{r-1}(-1)^{r-1}(r-1)! &\text{ if } d_1=d_2=\cdots=d_r=d\\ 
0&\text{ otherwise}.\end{cases}
$$
 where $\mu$ denotes the ordinary
M\"obius function.

\vspace{6 pt}

Lastly, we introduce the notion of \textit{levels} for semisimple multitypes.

\begin{definizione}
\label{multitypeleveldefinition}
For a subset $V \subseteq \N^I$ and a semisimple multitype $\omega$ with $$\omega=\psi_{d_1}(\omega_{\alpha_1})\ast \cdots \ast \psi_{d_r}(\omega_{\alpha_r}) ,$$ we say that $\omega$ is of level $V$ if $\alpha_j \in V$ for each $j=1,\dots,r$.

\end{definizione}
\vspace{8 pt}

\begin{esempio}
For any $\alpha \in \N^I $, the multitype $\omega_{\alpha}$ is of level $\{\alpha\}$. Conversely, the only semisimple multitype $\omega \in \mathbb{T}^{ss}_{\alpha}$ of level $\{\alpha\}$ is $\omega_{\alpha}$.
\end{esempio}

\subsection{Lambda rings and plethystic operations}
\label{lambdarings}
In this paragraph we recall the definition and some properties of $\lambda$-rings. We follow \cite[Appendix A]{mozgovoy}.

\begin{definizione}

A $\lambda$-ring $R$ is a commutative $\Q$-algebra with homomorphisms $\psi_d:R \to R$ for any $d \geq 1$ such that $\psi_{d'}(\psi_d(r))=\psi_{dd'}(r)$ for every $d,d'\in \N_{>0}$ and $r \in R$.

\end{definizione}

The morphisms $\psi_d$ are called Adams operations. For any partition $\mu=(\mu_1,\dots,\mu_h)$, we denote by $\psi_{\mu}:R \to R$ the homomorphism defined as $\psi_{\mu}(r)=\psi_{\mu_1}(r)\cdots \psi_{\mu_{h}}(r)$.

For every integer $n \in \N$, denote by $\sigma_n(f)$ the element \begin{equation}
    \label{plethysm}
    \sigma_n(f)=\sum_{\lambda \in \mathcal{P}_n}\dfrac{\psi_{\lambda}(f)}{z_{\lambda}}
\end{equation}

\vspace{10 pt}

For a $\lambda$-ring $R$, consider now the ring $R[[y_i]]_{i \in I}$. For $\alpha \in \N^I$ put $y^{\alpha}\coloneqq \displaystyle\prod_{i \in I}y_i^{\alpha_i}$. 
\vspace{2 pt}
We endow the ring $R[[y_i]]_{i \in I}$ with the $\lambda$-ring structure defined by the Adams operations defined as $$\psi_d(ry^{\alpha})\coloneqq\psi_{d}(r)y^{d \alpha}$$ for $r \in R$ and $\alpha \in \N^I$.

Denote by $R[[y_i]]_{i \in I}^+$ the ideal generated by the $y_i$'s. The \textit{plethystic exponential} is the following map $\Plexp :R[[y_i]]_{i \in I}^+ \to 1+R[[y_i]]_{i \in I}^+$: \begin{equation}
\label{plethysm1}\Plexp(f)= \exp\left(\sum_{n \geq 1}\dfrac{\psi_n(f)}{n}\right) .\end{equation}

Notice that for $f,g \in R[[y_i]]_{i \in I}^+,$ we have $\Plexp(f+g)=\Plexp(f)\Plexp(g)$.

\vspace{8 pt}
\begin{esempio}
Consider $R=\Q$ and $|I|=1$. In the ring $\Q[[T]]$ we have: \begin{equation}
    \label{esempio}
    \Plexp(T)=\exp\left(\sum_{n \geq 1} \frac{T^n}{n}\right)=\exp\left(\log\left(\frac{1}{1-T}\right)\right)=\frac{1}{1-T}
\end{equation}
\end{esempio}

\vspace{8 pt}
We have the following Lemma (see the discussion before \cite[Corollary 21]{mozgovoy}):

\begin{lemma}
\label{plethysm3}
For any $f \in R[[y_i]]^+_{i \in I}$, we have
\begin{equation}
\label{plethysm33}
\Plexp(f)=\sum_{n \geq 1}\sigma_n(f).
\end{equation}
\end{lemma}

\vspace{8 pt}

The plethystic exponential admits an inverse operation  $\Plelog:1+R[[y_i]]_{i \in I}^+ \to R[[y_i]]_{i \in I}^+$ known as \textit{plethystic logarithm}. The plethystic logarithm can either be defined by the property $\Plelog(\Plexp(f))=f $ or by the following explicit rule. For $\alpha \in \N^I$ we put $$\overline{\alpha}\coloneqq \gcd(\alpha_i)_{i \in I}$$ and we define $U_{\alpha} \in R$ by: \begin{equation}
    \log(f)=\sum_{\alpha \in \N^I}\frac{U_{\alpha}}{\overline{\alpha}}y^{\alpha}.
\end{equation}
Then put
 \begin{equation}
    \label{log}
    \Plelog(f)\coloneqq\sum_{\alpha \in \N^I}V_{\alpha}y^{\alpha}
\end{equation}
where \begin{equation}
    V_{\alpha}\coloneqq\dfrac{1}{\overline{\alpha}}\sum_{d| \overline{\alpha}}\mu(d)\psi_d(U_{\frac{\alpha}{d}})
\end{equation}

Indeed, consider $h=\sum_{\alpha \in \N^I}h_{\alpha}y^{\alpha} \in R[[y_i]]_{i \in I}^+$ and $f=\Plexp(h)$. We have then \begin{equation}\log(\Plexp(h))=\sum_{d \geq 1}\dfrac{\psi_d(h)}{d}=\sum_{d \geq 1}\sum_{\beta \in \N^I}\dfrac{\psi_d(h_{\beta})}{d}y^{d \beta}=\sum_{\alpha \in \N^I}y^{\alpha}\sum_{d | \overline{\alpha}}\dfrac{\psi_d(h_{\frac{\alpha}{d}})}{d}
\end{equation}
i.e.  \begin{equation}
U_{\alpha}=\overline{\alpha}\sum_{d |\overline{\alpha}}\dfrac{\psi_d(h_{\frac{\alpha}{d}})}{d}.
\end{equation}

We have then \begin{equation}
V_{\alpha}=\dfrac{1}{\overline{\alpha}}\sum_{d | \overline{\alpha}}\mu(d)\psi_d(U_{\frac{\alpha}{d}})=\dfrac{1}{\overline{\alpha}}\sum_{d | \overline{\alpha}}\mu(d) \psi_d\left(\dfrac{\overline{\alpha}}{d}\sum_{d'|\frac{\overline{\alpha}}{d}}\dfrac{\psi_{d'}(h_{\frac{\alpha}{dd'}})}{d'}\right)=
\end{equation}
\begin{equation}
\label{zioperalog}
=\sum_{d|\overline{\alpha}}\mu(d)\sum_{d'|\frac{\overline{\alpha}}{d}}\dfrac{\psi_{dd'}(h_{\frac{\alpha}{d d'}})}{dd'}=\sum_{m | \overline{\alpha}}\psi_m(h_{\frac{\alpha}{m}})\sum_{\substack{d, d' \text{ s.t. }\\ dd'=m}}\mu\left(d\right)=h_{\alpha},
\end{equation}

where the last equality of eq.(\ref{zioperalog}) comes from the fact that $$\sum_{\substack{d, d' \text{ s.t. }\\ dd'=m}}\mu\left(d\right)=\sum_{d | m}\mu(d)=0 ,$$ if $m \neq 1$.

\subsubsection{Multitypes and plethysm}
\label{paragraphring}

Let $\mathcal{K}^{ss}_I$ be the $\Q$-vector space having as a base the semisimple multitypes $\mathbb{T}_I^{ss}$. The size of the multitypes endows $\mathcal{
K}^{ss}_I$ with the structure of an $\N^I$-graded vector space and the operation $\ast$ endows $\mathcal{K}^{ss}_I$ with the structure of an $\N^I$-graded $\Q$-algebra. 

The operation $\psi_d$ on multitypes endows the $\Q$-algebra $\mathcal{K}^{ss}_I$ with the structure of a $\lambda$-ring with Adams operations defined as  $$\psi_d(q_1\omega_1 +\cdots +q_r \omega_r)\coloneqq q_1 \psi_d(\omega_1)+\cdots +q_r\psi_d(\omega_r) $$ for any element $q_1\omega_1+\cdots+q_r \omega_r \in \mathcal{K}^{ss}_I$ wth $q_1,\dots,q_r \in \Q$ and $\omega_1,\dots,\omega_r \in \mathbb{T}_I^{ss}$. This $\lambda$-ring is going to be a key tool to the proof of our main result \ref{mainteo}.

\vspace{2 pt}

Given a semisimple multitype $\omega=(d_1,(1^{\alpha_1}))\dots (d_r,(1^{\alpha_r}))$, in the ring $\mathcal{K}^{ss}_I$ we have an equality $$\omega=\psi_{d_1}(\omega_{\alpha_1})\ast \cdots \ast \psi_{d_r}(\omega_{\alpha_r}) .$$ We  therefore deduce that $\mathcal{K}^{ss}_I$ is isomorphic to the ring of polynomials in the variables $\psi_d(\omega_{\alpha})$ for $(d,\alpha) \in \N_{>0} \times \N^I$.

\vspace{8 pt}

 Consider now the ring $\hat{\mathcal{K}}^{ss}_I \coloneqq \mathcal{K}^{ss}_I[[y_i]]_{i \in I}$. For  semisimple multitypes of level $V$, we have the following lemma:

\begin{lemma}
\label{Vgenlevi}
For any $V \subseteq \N^I$, we have the following identity in the ring $\hat{\mathcal{K}}^{ss}_I$:

\begin{equation}
    \label{V-generic}
    \Plexp\left(\sum_{\alpha \in V} \omega_{\alpha} y^{\alpha}\right)=\sum_{\substack{\omega \in \mathbb{T}^{ss}_I \\ \text{ of level } V}} \dfrac{\omega}{w(\omega)}y^{|\omega|}
\end{equation}

\end{lemma}

\begin{proof}

By  eq.(\ref{plethysm}), there is an equality \begin{equation}\label{plethysm2}
\Plexp\left(\sum_{\alpha \in V} \omega_{\alpha} y^{\alpha}\right)=\prod_{\alpha \in V}\left(\sum_{n \in \N}\sigma_n(\omega_{\alpha})y^{n \alpha}\right)=\prod_{\alpha \in V}\left(\sum_{\lambda \in \mathcal{P}}\dfrac{\psi_{\lambda}(\omega_{\alpha})}{z_{\lambda}}y^{|\lambda|\alpha}\right). \end{equation}

For each  semisimple multitype $\omega$ of level $V$, there exist unique $\beta_1 \neq \beta_2 \neq \dots \neq \beta_h \in V$ and integers $d_{1,1},\dots d_{1,l_1},d_{2,1}, \dots ,d_{h,l_h}$ such that $\omega=(d_{1,1},(1^{\beta_1}))(d_{1,2},(1^{\beta_1}))\cdots (d_{h,l_h},(1^{\beta_h}))$ i.e. $$\omega=\psi_{d_{1,1}}(\omega_{\beta_1})\ast \cdots \ast \psi_{d_{h,l_h}}(\omega_{\beta_h}) .$$

Up to reordering, we can assume that for each $j=1,\dots,h$, the integers $(d_{j,1},\dots,d_{j,l_j})$ form a partition $\lambda_j$. We have thus an identity $$\displaystyle \omega=\prod_{j=1}^h \psi_{\lambda_j}(\omega_{\beta_j}) .$$ Notice moreover that $z_{\lambda_1}\cdots z_{\lambda_h}=w(\omega)$. We therefore deduce  that the RHS of eq.(\ref{plethysm2}) is equal to the RHS of eq.(\ref{V-generic}).

\end{proof}

\section{Finite reductive groups, subtori and multitypes}
\label{subtoriandmultitypes}
In this chapter we will review the notions about the geometry of reductive groups over $\F_q$ that we will need in the rest of the article. The main reference is the book \cite{DM}. We will start by fixing some notations about varieties over finite fields and Frobenius morphisms.

\subsection{Varieties over finite fields}

Let $q=p^r$ where $p$ is a prime number, let $\F_q$ be the field with $q$ elements and $\overline{\F}_q$ its algebraic closure. In the following, a variety over $\F_q$  will be a pair $(X,F)$ where $X$ is a  reduced scheme  of finite type over $\overline{\F}_q$ and $F$ is an $\F_q$- Frobenius morphism $F:X \to X$ (for more details see for example Milne's book \cite{milne} or \cite[Chapter 4]{DM}).

Whenever the $\F_q$-structure of $X$ is clear, we will often drop the Frobenius morphism in the notation and we will simply use the terminology "the $\F_q$-variety $X$".  

A morphism of $\F_q$-varieties $f:(X,F) \to (Y,F)$ is a morphism $f:X \to Y$ which commutes with the corresponding Frobenius maps.

Given an affine variety $(X,F)$ over $\mathbb{F}_q$  with Frobenius $F:X \to X$, consider  the  variety $X^d $ equipped with the twisted Frobenius $$F_d:X^d \to X^d$$ defined as $$F_d(x_1,\dots x_d)\coloneqq(F(x_d),F(x_1),\dots,F(x_{d-1})) .$$ In the following, we will denote  the $\F_q$-variety $(X^d,F_d)$ simply by $(X)_d$. Notice that there is a bijection $(X)_d(\F_q)=X(\F_{q^d})$.

\vspace{12 pt}

For $n \in \N$, we denote by $\Gl_n$ the general linear group over $\overline{\F}_q$. The group $\Gl_n$ is endowed with the canonical Frobenius morphism $F((a_{i,j}))=(a_{i,j}^q)$ for a matrix $(a_{i,j}) \in \Gl_n$.

For $\alpha \in \N^I$, we denote by $\Gl_{\alpha}$ the group $$\Gl_{\alpha}\coloneqq\displaystyle \prod_{i \in I} \Gl_{\alpha_i}$$ with the usual product Frobenius structure and by $\Gl_{\alpha}(\F_q)$ the finite group $\Gl_{\alpha}(\F_q)=\Gl_{\alpha}^F$.

\vspace{10 pt}

\begin{oss}
\label{exampleembeddingtwistedtori}
For each $n,d \geq 1$, we will define an embedding  $(\Gl_n)_d \subseteq \Gl_{nd}$ defined over $\F_q$ in the following way. Let $\Delta:\Gl_n^d \to \Gl_{nd}$ be the block diagonal embedding. 

Notice that $\Delta$ is not defined over $\F_q$ when $\Gl_n^d$ is equipped with the Frobenius structure $F_d$. Consider then the permutation $\sigma \in S_{nd}$ given by $\sigma=(1 \ (n+1) \cdots (n(d-1)+1)) \cdots (n \ 2n \cdots dn)$ and the associated permutation matrix $J_{\sigma} \in \Gl_{nd}$. 

Fix an element $g_{\sigma} \in \Gl_{nd}$ such that $g_{\sigma}^{-1}F(g_{\sigma})=J_{\sigma}$ (such an element exists because of the surjectivity of the Lang map, see for example \cite[Theorem 4.29]{DM}). The embedding $$g_{\sigma}\Delta g_{\sigma}^{-1}:(\Gl_n^d,F_d) \to (\Gl_{nd},F)$$ is defined 
over $\F_q$. 

Similarly, we get an $\F_q$-embedding of $(\Gl_{\alpha})_d$ inside $\Gl_{\alpha d}$ for any $d \geq 1$ and any $\alpha \in \N^I$. 

\end{oss}

\vspace{6 pt}

\begin{esempio}
\label{twistedtorusGl_2}
Let $n=2$. Fix $x \in \F_{q^2}^* \setminus \F^*_q$ and let $T_{\epsilon}$ be the torus $$T_{\epsilon} \coloneqq \Biggl\{\dfrac{1}{x^q-x}\begin{pmatrix}
ax^q-bx &-a+b\\
(a-b)xx^q &-ax+bx^q
\end{pmatrix} \ | \ a,b \in \overline{\F_q}^* \Biggr\} .$$  

The torus $T_{\epsilon}$ is $F$-stable and is $\Gl_2(\F_q)$-conjugated to the torus $(\mathbb{G}_m)_2$ embedded inside $\Gl_2$ as in Remark \ref{exampleembeddingtwistedtori} above. 
\end{esempio}

\subsection{Reductive groups over finite fields}

We start by recalling the following definition.

\begin{definizione}
\begin{itemize}
\item A  group $T$ over $\F_q$ is called a torus if there is an isomorphism $T \times_{\spec(\F_q)} \spec(\overline{\F}_q) \cong \mathbb{G}_m^d$ for a certain integer $d$. For a torus $T$, we denote by $\rank(T)$ its dimension. 
\item A torus $T$ over $\F_q$ is called split if we have an $\F_q$-isomorphism $T \cong_{\F_q}\mathbb{G}_m^{\rank(T)}$.
\end{itemize}
\end{definizione}

\vspace{4 pt}

In this paragraph and in the rest of the article, $G$ is going to be a reductive group defined over $\F_q$ with a fixed Frobenius morphism $F:G \to G$. In the cases that interest us in this article, $G$ will always be a product of factors of type $(\Gl_n)_d$'s.

We denote by $\rank(G)$ the dimension of a maximal torus of $G$. Recall that we always have an $F$-stable maximal subtorus $T \subseteq G$.

We denote by $\epsilon_G$ the rank of a  maximal  split $F$-stable subtorus of $G$. In general $\epsilon_G \neq \rank(G)$. If $\rank(G)=\epsilon_G$, we say that $G$ is split.

\vspace{8 pt}

\begin{esempio}
Consider the group $(\Gl_n)_d$. Let $T \subseteq G$ be the maximal torus $T_n \times \cdots \times T_n \subseteq (\Gl_n)_d$, where we denote by $T_n \subseteq \Gl_n$ the torus of diagonal matrices. Notice that $T$ is $F$-stable and $\dim(T)=\rank(G)=nd$. However, it is possible to verify that $\epsilon_G=n$, i.e. $(\Gl_n)_d$ is split if and only if $d=1$.
\end{esempio}

For a maximal torus $T$, we denote by $$X_*(T) \coloneqq \Hom(T,\mathbb{G}_m) $$ and $$Y_*(T) \coloneqq \Hom(\mathbb{G}_m,T) $$ the group of \textit{characters} and \textit{cocharacters} of $T$ respectively. Recall that these are free abelian groups of rank equal to $\rank(G)$ and that there is a pairing $$ \langle, \rangle:  Y_*(T) \times X_*(T) \to \Z ,$$ where, for $\beta \in Y_*(T), \alpha \in X_*(T)$, we have $$\alpha \circ  \beta(z)=z^{\langle \beta,\alpha \rangle} $$ for any $z \in \mathbb{G}_m$.

Denote by $W_G(T)$ the  Weyl group of $T$, i.e. $W_G(T)=N_G(T)/T$. Notice that $W_G(T)$ acts on $X_*(T)$ as follows $$ w \cdot \alpha(t)=\alpha(wtw^{-1}) $$ for each $w \in W_G(T)$, $t \in T$ and $\alpha \in X_*(T)$. For each $w \in W_G(T)$, we denote by $w: X_*(T) \to X_*(T)$ the corresponding endomorphism. 

Recall that inside $X_*(T)$ there is the \textit{root system} $\Phi(T) \subseteq X_*(T)$ given by the characters appearing in the weight space decomposition of the adjoint action of $T$ on $\mathfrak{g}=\Lie(G)$. Recall that for any $\epsilon \in \Phi(T)$, there is an injective homomorphism 
$u_{\epsilon}:\mathbb{G}_a \to G$ such that for any $x \in \overline{\F}_q$ and any $t \in T$, we have $$tu_{\epsilon}(x)t^{-1}=u_{\epsilon}(\epsilon(t)x) .$$ We denote by $U_{\epsilon} \subseteq G$ the subgroup $U_{\epsilon}\coloneqq \Imm(u_{\epsilon})$. 

\vspace{6 pt}

Moreover,  inside $Y_*(T)$ there is the dual root system $\Phi^{\vee}(T)$, provided with a canonical bijection $$\Phi(T) \leftrightarrow \Phi^{\vee}(T) $$ $$\epsilon \leftrightarrow \epsilon^{\vee}$$ such that $\langle \epsilon^{\vee},\epsilon \rangle=2$ for every $\epsilon \in \Phi(T)$.

If $T$ is $F$-stable, the Frobenius acts on the groups $X_*(T),Y_*(T)$ as follows $$F:X_*(T) \to X_*(T) $$ $$ \alpha \to \alpha \circ F$$
and $$F:Y_*(T) \to Y_*(T) $$ $$\beta \to F \circ \beta .$$ 

\subsubsection{Twisted Frobenius of maximal tori}
\label{twistedFrobeniuspara}
Fix now an $F$-stable maximal torus $T \subseteq G$. As $T$ is $F$-stable, the Frobenius acts on the Weyl group too. Given two elements $h_1,h_2 \in W_G(T)$, we say that they are $F$-conjugated if there exists $w \in W_G(T)$ such that $h_1=wh_2F(w)^{-1}$.

The set of $F$-conjugacy classes of $W_G(T)$, usually denoted by $H^1(F,W_G(T))$, parametrize the $G^F$-conjugacy classes of $F$-stable maximal tori in the following way. 

Given an $F$-stable maximal torus $T'$ there exists  $g \in G$ such that $gTg^{-1}=T'$. As $F(T')=T'$ we see that $\dot{w}=g^{-1}F(g)$ belongs to $N_G(T)$ and so determines an associated element $w \in W_G(T)$. The element $w$ is well defined up to conjugacy, i.e. the conjugacy class of $w$ does not depend on $g$.

Conversely, for every $w \in W_G(T)$, consider an element $g \in G$ such that $g^{-1}F(g)=w$. Such an element exists thanks to the surjectivity of Lang's map, see for example \cite[Theorem 4.29]{DM}. To $w$ we associate the torus $gTg^{-1}=T'$. The $G^F$-conjugacy class of the torus $gTg^{-1}$ does not depend on the choice of $g$ neither on $w$, but rather on the $F$-conjugacy class of $w$.

\vspace{4 pt}

We can reformulate this correspondence in terms of the twisted $\mathbb{F}_q$-structures of the torus $T$. While the conjugation by $g$ provides an isomorphism $ T'\cong T$ over $\overline{\F}_q$, this isomorphism is not in general an $\F_q$-morphism $(T',F) \to (T,F)$.

However, when $T$ is equipped with the $\mathbb{F}_q$ structure coming from the twisted Frobenius $\dot{w}F:T \to T$, the conjugation by $g$ is an $\mathbb{F}_q$ -isomorphism $$(T',F) \cong_{\mathbb{F}_q} (T,\dot{w}F) .$$

In the following, we assume to have fixed, for each $w \in W_G(T)$, a corresponding $F$-stable maximal torus $T_w \subseteq G$.

\vspace{6 pt}

\begin{esempio}
\label{toriandLevi}
Consider the case of $G=\Gl_n$ and $T=T_n$ the torus of diagonal matrices. In this case, we have $W_G(T_n)=S_n$ and the $F$-action on $S_n$ is trivial. In particular, the $F$-conjugacy classes of $S_n$ are the conjugacy classes of $S_n$ and are therefore indexed by the partitions $\mathcal{P}_n$ of size $n$.

For any $\lambda=(\lambda_1,\dots,\lambda_h) \in \mathcal{P}_n$, any associated torus $F$-stable maximal torus $T'$ is $\Gl_n(\F_q)$-conjugated to $$(\mathbb{G}_m)_{\lambda_1} \times \cdots \times (\mathbb{G}_m)_{\lambda_h} ,$$ i.e. $$(T',F) \cong  (\mathbb{G}_m)_{\lambda_1} \times \cdots \times (\mathbb{G}_m)_{\lambda_h} .$$

\end{esempio}

\vspace{8 pt}

\begin{esempio}
\label{twisted}
 Consider the group $G=(\Gl_n)_d$ and the $F$-stable maximal torus $T$ introduced above. The Weyl group $W_G(T)$ is isomorphic to $S_n^d$ and the corresponding Frobenius action $F:S_n^d \to S_n^d$ is given by $$F(\sigma_1,\dots,\sigma_d)=(\sigma_d,\sigma_1,\dots,\sigma_{d-1}) .$$

The $F$-conjugacy classes of $S_n^d$ are in bijection with the conjugacy classes of $S_n$ in the following way. Consider $\tau=(\tau_1,\dots,\tau_d)$,  $\sigma=(\sigma_1,\dots,\sigma_d) \in S_n^d$. The element $\tau \sigma F(\tau)^{-1}$ is equal to $(\tau_1\sigma_1\tau_d^{-1},\tau_2 \sigma_2 \tau_{1}^{-1},\dots,\tau_{d}\sigma_d \tau_{d-1}^{-1})$. We have $$\prod_{i=0}^{d-1} (\tau \sigma F(\tau)^{-1})_{d-i}=\tau_{d} (\sigma_d \sigma_{d-1} \cdots \sigma_1) \tau_{d}^{-1}=\tau_d \left(\prod_{i=0}^{d-1} \sigma_{d-i}\right)\tau_d^{-1} .$$

We  deduce therefore that $\sigma,\sigma' \in S_n^d$ are $F$-conjugated if and only if $\displaystyle \prod_{i=0}^{d-1} \sigma_{d-i},\prod_{i=0}^{d-1} \sigma'_{d-i}$ are conjugated in $S_n$.

\end{esempio}

\vspace{10 pt}

Consider a pair of $F$-stable maximal tori $T,T'$ with $gT'g^{-1}=T$ and $\dot{w}=g^{-1}F(g) \in N_G(T')$ and $w \in W_G(T')$ as above. There is an isomorphism of abelian groups $$\Psi_g: X_*(T') \to X_*(T) $$ $$\alpha \to \alpha(g^{-1}(-)g) $$ such that $\Psi_g(\Phi(T'))=\Phi(T)$. In general, $\Psi_g$ does not commute with the respective Frobenius morphisms on $X_*(T),X_*(T')$ and indeed we have

\begin{equation}
\label{isomorphismcharactergroups}
\Psi_g^{-1} F \Psi_g=w \circ F: X_*(T') \to X_*(T').
\end{equation}

\subsubsection{The case of finite general linear groups}
\label{twistedcharactergroup}

For $m \in \N$, let $\Gl_m$ be the general linear group over $\F_q$, with the canonical $\F_q$-structure $F:\Gl_m \to \Gl_m$. Consider the  maximal torus of diagonal matrices $T_m \subseteq \Gl_m$. As mentioned before, in this case $W_G(T_m)=S_m$ and the $F$-action on $S_m$ is trivial.

\vspace{6 pt}

Let $\epsilon_i \in X_*(T_m)$ be the homomorphism
$$\epsilon_i\left(\begin{pmatrix}
    z_1 &0 &0 &0 &0 &\dots &0\\
    0 &z_2 &0 &0 &0 &\dots &0 \\
    0  &0 &z_3 &0 &0 &\dots &0\\
    \vdots &\vdots &\vdots &\ddots &\dots &\dots &0  \\
    0 &0 &0 &0 &0 &0 &z_n    \end{pmatrix} \right)=z_i .$$

The subset $\{\epsilon_1,\dots,\epsilon_m\} \subset X_*(T_m)$ is a basis of the free abelian group $X_*(T_m)$, which we denote by $\mathcal{B}(T_m)$. Notice that, for each $i=1,\dots,m$, we have that $F(\epsilon_i)=q\epsilon_i$. 

Moreover, for such a basis, we have that $$\Phi(T_m)=\{\pm\epsilon_i\mp\epsilon_j \ | \ i \neq j \in \{1,\dots,m\}\} .$$  For $h,j \in \{1,\dots,m\}$, put $\epsilon_{h,j}\coloneqq \epsilon_h -\epsilon_j$. We denote by $\Phi^+(T_m)$ the set of positive roots with respect to the Borel subgroup of upper triangular matrices, i.e. $$\Phi^+(T_m)=\{\epsilon_{i,j} \ | \ i < j\} .$$

\vspace{8 pt}

For any other $F$-stable maximal torus $T \subseteq \Gl_m$, fix $g$ such that $gT_mg^{-1}=T$ and the corresponding permutation $w \in W_G(T_m)=S_m$, as at the end of paragraph \cref{twistedFrobeniuspara} above. Put $\mathcal{B}(T)\coloneqq\Psi_g(\mathcal{B}(T_m))$. 

Whenever the torus $T$ is fixed and the context is clear we will denote by $\epsilon_i$ also the element $\Psi_g(\epsilon_i) \in \mathcal{B}(T)$ and by $\epsilon_{i,j}$ the element $\Psi_g(\epsilon_{i,j}) \in \Phi(T)$. We denote by $\Phi^+(T)=\Psi_g(\Phi^+(T_m))$.

Notice that, by eq.(\ref{isomorphismcharactergroups}), in the character group $X_*(T)$ we have $$F(\epsilon_i)=q\epsilon_{w(i)} .$$

\vspace{10 pt}

Consider now $\alpha \in \N^I$, put $m=|\alpha|\coloneqq \sum_{i \in I}\alpha_i$ and consider $\Gl_{\alpha}$ as a subgroup of $\Gl_{m}$ through the block diagonal embedding.  Fix an $F$-stable  maximal torus $T \subseteq \Gl_{\alpha} \subseteq \Gl_{m}$. The torus $T$ is a maximal torus for $\Gl_{\alpha}$ and $\Gl_m$. Consider then the basis $\mathcal{B}(T)$ introduced before.

\vspace{8 pt}

\begin{oss}
For $i \in I$, let $\pi_i:\Gl_{\alpha} \to \Gl_{\alpha_i}$ the canonical projection. For a maximal torus $T \subseteq \Gl_{\alpha}$, denote by $T_i \coloneqq \pi_i(T)$. We have inclusions $$ \displaystyle T \subseteq \prod_{i \in I} T_i \subseteq \prod_{i \in I} \Gl_{\alpha_i} .$$

As $T$ is a maximal torus, we have thus an equality $T=\displaystyle\prod_{i \in I} T_i$. For dimensional reasons, we deduce that, for each $i \in I$, $T_i$ is a maximal torus of $\Gl_{\alpha_i}$. From the identity $T=\displaystyle\prod_{i \in I}T_i$, we therefore deduce that there is an isomorphism $X_*(T)=\displaystyle \bigoplus_{i \in I} X_*(T_i)$.
\end{oss}

\vspace{6 pt}

We can choose $g \in \Gl_{\alpha}$ such that $gT_mg^{-1}=T$. We thus deduce  that, putting $\mathcal{B}_i(T)=\mathcal{B}(T) \cap X_*(T_i)$, we obtain a partition $$\displaystyle \mathcal{B}(T)=\bigsqcup_{i \in I}\mathcal{B}_i(T)$$ such that each $\mathcal{B}_i(T)$ is a basis of $X_*(T_i)$ and $\mathcal{B}_i(T)$ is $w$-stable for every $i \in I$.

\subsection{Levi subgroups and parabolic subgroups of general linear groups}

Recall that a parabolic subgroup $P\subseteq G$ is a subgroup containing a Borel subgroup and denote by $U_P \subseteq P$ its unipotent radical. A Levi factor of $P$ is a reductive subgroup $L \subseteq P$ such that $P=LU_P$. A Levi factor of a parabolic subgroup is called a \textit{Levi subgroup} of $G$.

\vspace{6 pt}

Recall that, for any Levi subgroup $L \subseteq G$, there exists a maximal torus $T \subseteq G$ such that $T \subseteq L$. Moreover, the Levi subgroup $L$  can be described in terms of the root systems $\Phi(T)$ as follows. 

Consider the subset $\Phi_L(T) \subseteq \Phi(T)$ defined as $$\Phi_L(T)\coloneqq \{\epsilon \in \Phi(T) \ | \ \Ker(\epsilon) \supseteq Z^{\circ}_L \} ,$$

where $Z_L^{\circ}$ is the connected componenent containing the identity of the center $Z_L \subseteq L$.

\vspace{6 pt}

We have the following Lemma, see \cite[Lemma 8.4.2]{Springer}.

\begin{lemma}
\label{levisubgroups}
The subset $\Phi_L(T)$ is a root subsystem of $\Phi(T)$ and we have :
\begin{enumerate}
\item $C_{G}(Z^{\circ}_L)=L$.

\item $\displaystyle Z_L^{\circ}= \bigcap_{\epsilon \in \Phi_L(T)}(\Ker(\epsilon)^{\circ})$

\item $\displaystyle L=T \displaystyle \prod_{\epsilon \in \Phi_L(T)}U_{\epsilon}$.
\end{enumerate}
\end{lemma}

\vspace{6 pt}

\begin{esempio}
\label{keyexampleLevi}
For $G=\Gl_n$, Levi subgroups and parabolic subgroups can be explicitly described as follows. For any $n_0,\dots,n_s \in \N$ such that $n_0+\cdots +n_s=n$, the subgroup $L_{n_0,\dots,n_s}$ defined as $$L_{n_0,\dots,n_s}=\begin{pmatrix}
    \Gl_{n_{s}} &0 &0 &0 &0 &\dots &0\\
    0 &\Gl_{n_{s-1}} &0 &0 &0 &\dots &0 \\
    0  &0 &\Gl_{n_{s-2}} &0 &0 &\dots &0\\
    \vdots &\vdots &\vdots &\ddots &\dots &\dots &0  \\
    0 &0 &0 &0 &0 &0 &\Gl_{n_0}
    \end{pmatrix} $$
is a Levi subgroup of $G$. We will denote the group $L_{n_0,\dots,n_s}$ simply by $$ \Gl_{n_0} \times \cdots \times \Gl_{n_s} \subseteq \Gl_n .$$

We have that $$Z_L=\begin{pmatrix}
    \lambda_s I_{n_{s}} &0 &0 &0 &0 &\dots &0\\
    0 & \lambda_{s-1} I_{n_{s-1}} &0 &0 &0 &\dots &0 \\
    0  &0 &\lambda_{s-2} I_{n_{s-2}} &0 &0 &\dots &0\\
    \vdots &\vdots &\vdots &\ddots &\dots &\dots &0  \\
    0 &0 &0 &0 &0 &0 &\lambda_0 I_{n_0}
    \end{pmatrix},  $$
for $\lambda_0,\dots,\lambda_s \in \overline{\F}^*_q$. Notice in particular that $Z_L$ is connected.

\vspace{6 pt}

A parabolic subgroup $P$ containing  $ \Gl_{n_0} \times \cdots \times \Gl_{n_s}$ is given by the upper block triangular matrices $$P=\begin{pmatrix}
    \Gl_{n_{s}} &* &* &* &* &\dots &*\\
    0 &\Gl_{n_{s-1}} &* &* &* &\dots &* \\
    0  &0 &\Gl_{n_{s-2}} &* &* &\dots &*\\
    \vdots &\vdots &\vdots &\ddots &\dots &\dots &*\\
    0 &0 &0 &0 &0 &0 &\Gl_{n_0}
    \end{pmatrix} .$$

\vspace{6 pt}

It is not difficult to verify that, for any Levi subgroup $L \subseteq \Gl_n$, there exist $n_0,\dots,n_s$ such that $n_0+\cdots+n_s=n$ and $L$ is conjugated to $$\Gl_{n_0} \times \cdots \times \Gl_{n_s}  .$$ 

Assume now that $L$ is $F$-stable. In a way similar to what we said about $F$-stable maximal tori in Example \ref{toriandLevi}, we can show that there exist $d_0,\dots,d_r \in \N$ and $m_0,\dots,m_r$ such that $L$ is conjugated by an element of $\Gl_n(\F_q)$ to the group $$(\Gl_{m_0})_{d_0} \times \cdots \times (\Gl_{m_r})_{d_r} ,$$ i.e. there is an $\F_q$-isomorphism $$(L,F) \cong  (\Gl_{m_0})_{d_0} \times \cdots \times (\Gl_{m_r})_{d_r} .$$

In this case we have an isomorphism $$(Z_L,F) \cong (\mathbb{G}_m)_{d_0} \times \cdots \times (\mathbb{G}_m)_{d_r} .$$

\end{esempio}

\subsection{Admissible subtori of general linear groups}

For $\alpha \in \N^I$, put $\displaystyle|\alpha|\coloneqq\sum_{i \in I} \alpha_i$ and consider  $\Gl_{\alpha} $ as a subgroup of $\Gl_{|\alpha|}$ via the block diagonal embedding. Recall that $I$ can be thought of as the set of vertices of a star-shaped quiver. We introduce here the definition of the admissible subtori of $\Gl_{\alpha}$.

\vspace{6 pt}

Admissible subtori will play a significant role in this paper. For instance, they appear in the classification of the irreducible characters of the finite group $\Gl_{\alpha}(\F_q)$ (see \cref{typesirreduciblechar}).

They are also a key part of the proof of Theorem \ref{mainteo}, which will be the main technical result needed to study the cohomology of non-generic character stacks.

\vspace{6 pt}

\begin{definizione}
A subtorus $S$ of $\Gl_{\alpha}$ is said \textit{admissible} if there exists a Levi subgroup $L_S \subseteq \Gl_{|\alpha|}$ such that $Z_{L_S}=S$.

\end{definizione}

\vspace{6 pt}

\begin{esempio}
For any $\alpha \in \N^I$, there is an admissible subtorus $Z_{\alpha} \subseteq \Gl_{\alpha}$, given by $Z_{\alpha}\coloneqq Z_{\Gl_{|\alpha|}} \subseteq  \Gl_{\alpha}$, i.e. the elements of $Z_{\alpha}$ are of the form $(\lambda I_{\alpha_i})_{i \in I}$, for $\lambda \in \overline{\F}_q^{*}$.

\end{esempio}

\vspace{6 pt}

\begin{oss}
If $|I|=1$, admissible subtori are the centers of the Levi subgroups of $\Gl_n$. This latter type of tori has already appeared in \cite[Section 4.2]{HA}, where the authors used them to count points of generic character stacks for Riemann surfaces over finite fields.

As mentioned in the introduction, to generalize their results to the non-generic case we will need to study, more generally, multiplicative quiver stacks over finite fields, whose study will require a careful understanding of the case $|I|>1$. 
\end{oss}

\vspace{6 pt}

We have the following Lemma (see \cite[Proposition 3.4.6]{DM})

\begin{lemma}
\label{sonouncoglione}
For an admissible $S$ and a Levi subgroup $L_S$ such that $Z_{L_S}=S$, we have $C_{\Gl_{|\alpha|}}(S)=L_S$. In particular, the group $L_S$ is unique.
\end{lemma}

\vspace{6 pt}

\begin{oss}
From Lemma \ref{sonouncoglione} above, we have that  $S$ is $F$-stable if and only if $L_S$ is $F$-stable.
\end{oss}

\vspace{6 pt}

\begin{esempio}
Put $|I|=1$ and let $S \subseteq \Gl_2$ be the torus $$S=\Biggl\{\begin{pmatrix} \lambda &0 \\ 0 &\lambda^2  \end{pmatrix}, \ \ \lambda \in \overline{\F}_q^* \Biggr\} .$$

Notice that $C_{\Gl_2}(S)=T_2$, where $T_2 \subseteq \Gl_2$ is the torus of diagonal matrices. However, $Z_{T_2}=T_2 \neq S$. We deduce thus that the torus $S$ is not admissible.
\end{esempio}

\vspace{6 pt}

Consider an admissible subtorus $S \subseteq \Gl_{\alpha}$ and the associated Levi subgroup $L_S \subseteq \Gl_{|\alpha|}$. The group $C_{\Gl_{\alpha}}(S)$ is a Levi subgroup of $\Gl_{\alpha}$ (see \cite[Proposition 3.4.7]{DM}), which we will denote by $\widetilde{L_S}$.

The group $\widetilde{L_S}$  is equal to $L_S \cap \Gl_{\alpha}$ as $C_{\Gl_{|\alpha|}}(S) \cap \Gl_{\alpha}=C_{\Gl_{\alpha}}(S)$. In particular, there exists a maximal torus $T \subseteq \Gl_{\alpha}$ such that $S \subseteq  T \subseteq \widetilde{L_S}$. 

Conversely, consider an $F$-stable Levi subgroup $L \subseteq \Gl_{|\alpha|}$ such that there exists a maximal torus $T \subseteq L \cap \Gl_{\alpha}$. As $Z_L \subseteq T$, the center $Z_L$ is an admissible subtorus of $\Gl_{\alpha}$.

\vspace{12 pt}

\begin{esempio}
Notice that even if two admissible tori $S,S'$ are different, we can have $\widetilde{L_S}=\widetilde{L_{S'}}$. Consider for example $S=Z_{\alpha}$ and $S'$ defined as $$S'=\{(\lambda_i I_{\alpha_i})_{i \in I} \ | \ (\lambda_i)_{i \in I} \in (\overline{\F}_q^*)^I\} .$$

In general, we have $S \neq S'$. However, for any $\alpha \in \N^I$, for both tori we have $$\widetilde{L_S}=\widetilde{L_{S'}}=\Gl_{\alpha} .$$

In particular, it is not true in general that $Z_{\widetilde{L_S}}=S$.
\end{esempio}

\vspace{8 pt}

For each $\alpha \in \N^I$, denote by $\mathcal{Z}_{\alpha}$ the subset of $F$-stable admissible subtori of $\Gl_{\alpha}$ and denote by $\mathcal{Z}$ the set  defined as $$\mathcal{Z} \coloneqq \bigsqcup_{\alpha \in \N^I} \mathcal{Z}_{\alpha} .$$

\vspace{8 pt}

\begin{esempio}
\label{exampleadmissible2}
Consider $I=\{1,2,3,4\}$ and $\alpha=(2,1,1,1) \in \N^I$. We have that $|\alpha|=5$. Consider for example the admissible subtori $S_1, S_2, S_3 \subseteq \Gl_{\alpha}$   given by $$S_1=\Biggl\{\left(\begin{pmatrix} \lambda &0 \\ 0 &\mu  \end{pmatrix}, \lambda ,\lambda,\lambda \right) \ | \ \lambda,\mu \in \overline{\F}_q^* \Biggr\} $$  $$ S_2=\Biggl\{\left(\begin{pmatrix} \lambda &0 \\ 0 &\mu  \end{pmatrix}, \lambda ,\mu,\lambda \right) \ | \ \lambda,\mu  \in \overline{\F}_q^* \Biggr\} $$ and $$ S_3=\Biggl\{\left(\begin{pmatrix} \lambda &0 \\ 0 &\mu  \end{pmatrix}, \gamma ,\delta,\eta \right) \ | \ \lambda,\mu, \gamma ,\delta,\eta  \in \overline{\F}_q^* \Biggr\} .$$

In this case, $L_{S_1}$ is $\Gl_5(\F_q)$-conjugated to $\Gl_4 \times \Gl_1$ and $$\widetilde{L_{S_1}}=T_2 \times \Gl_1 \times \Gl_1 \times \Gl_1 \subseteq \Gl_{2 } \times \Gl_1 \times \Gl_1 \times \Gl_1,$$ where $T_2 \subseteq \Gl_2$ is the torus of diagonal matrices. Moreover, we have that $L_{S_2}$ is $\Gl_5(\F_q)$-conjugated to $\Gl_2 \times \Gl_3 $ and $\widetilde{L_{S_1}}=\widetilde{L_{S_2}}$. 

Lastly, the Levi subgroup $L_{S_3}$ is the maximal torus of diagonal matrices $T_5 \subseteq \Gl_5$, and $\widetilde{L_{S_3}}=\widetilde{L_{S_1}}$ too.

\end{esempio}

\vspace{6 pt}

For a multitype $\omega=(d_1,\bm \lambda_1)\dots (d_r,\bm \lambda_r)$ of size $\alpha$, we denote by $S_{\omega} \in \mathcal{Z}_{\alpha}$ the torus defined as $$ (Z_{|\bm\lambda_1|})_{d_1} \times \cdots \times (Z_{|\bm\lambda_r|})_{d_r} \subseteq \Gl_{\alpha}$$ where $(Z_{|\bm\lambda_1|})_{d_1} \times \cdots \times (Z_{|\bm\lambda_r|})_{d_r}$ is considered a subtorus of $\Gl_{\alpha}$ via the componentwise block diagonal embedding. Denote by $\beta_j=|\bm \lambda_j| \in \N^I$, for each $j=1,\dots,r$. For the Levi subgroup $L_{\omega}\subseteq \Gl_{|\alpha|}$ defined as
$$L_{\omega}=(\Gl_{|\beta_1|})_{d_1} \times \cdots \times (\Gl_{|\beta_r|})_{d_r} $$
embedded block diagonally, we have $Z_{L_{\omega}}=S_{\omega}$, i.e. $S_{\omega}$ is admissible.

We will denote by $\widetilde{L_{\omega}}$ the Levi subgroup of $\Gl_{\alpha}$ defined as $\widetilde{L_{\omega}}\coloneqq L_{\omega} \cap \Gl_{\alpha}$. Notice that the groups $L_{\omega},S_{\omega},\widetilde{L_{\omega}}$ depend only on the semisimplification $\omega^{ss}$ of $\omega$.

\vspace{6 pt}

\vspace{6 pt}

\begin{oss}
\label{remarklevistarbucks}
Let $\omega \in \mathbb{T}_{\alpha}$ and $d_1,\dots,d_r \in \N$ and $\beta_1,\dots,\beta_r \in \N^I$ with $$\omega^{ss}=\psi_{d_1}(\omega_{\beta_1}) \ast \cdots \ast \psi_{d_r}(\omega_{\beta_r}) .$$

For each $i \in I$, consider the Levi subgroup $$(\Gl_{(\beta_1)_i})_{d_1} \times \cdots  \times (\Gl_{(\beta_r)_i})_{d_r} \subseteq \Gl_{\alpha_i} $$ embedded block diagonally. The Levi subgroup $\widetilde{L_{\omega}}$ is given by $$\widetilde{L_{\omega}}=\prod_{i \in I} (\Gl_{(\beta_1)_i})_{d_1} \times \cdots \times (\Gl_{(\beta_r)_i})_{d_r} .$$
\end{oss}

\vspace{6 pt}

From the description of the Levi subgroups of $\Gl_{|\alpha|}$ given in Example \ref{keyexampleLevi}, we deduce the following Lemma.
\begin{lemma}
For any $F$-stable admissible $E \in \mathcal{Z}_{\alpha}$, there exists a unique semisimple multitype, which we denote by $[E]$, such that $E$ is $\Gl_{\alpha}(\F_q)$-conjugated to $S_{[E]}.$
\end{lemma}

\begin{esempio}
For any $\alpha \in \N^I$, we have that $$[Z_{\alpha}]=\omega_{\alpha}. $$
\end{esempio}

\vspace{4 pt}

Let $\sim$ be the equivalence relation on $\mathcal{Z}_{\alpha}$, induced by the conjugation by $\Gl_{\alpha}(\F_q)$ and let $\overline{\mathcal{Z}}\coloneqq\mathcal{Z}/\sim$ be the quotient set. The map $\mathcal{Z}_{\alpha} \to \mathbb{T}_{\alpha}^{ss}$ given by $ S \rightarrow [S]$, induces thus a bijection $$\overline{\mathcal{Z}} \cong \mathbb{T}_I^{ss} .$$

Lastly, we give the following definition of levels for the admissible subtori.,

\begin{definizione}
Given $S \in \mathcal{Z}$ and $V \subseteq \N^I$, we say that $S$ is of level $V$ if $[S]$ is of level $V$ (see Definition \ref{multitypeleveldefinition}).
\end{definizione}

\begin{esempio}
 
Consider the tori $S_1,S_2,S_3$ introduced in Example \ref{exampleadmissible2}. The torus $S_1$ is the product $Z_{(1,1,1,1)} \times Z_{(1,0,0,0)}$ embedded componentwise block diagonally into $\Gl_{\alpha}$. The multitype $[S_1]$ is therefore the semisimple multitype $$[S_1]=(1,(1^{(1,1,1,1)}))(1,(1^{(1,0,0,0)}))=\omega_{(1,1,1,1)} \ast \omega_{(1,0,0,0)}.$$

Similarly, we have $$[S_2]=\omega_{(1,0,1,0)} \ast \omega_{(1,1,0,1)} $$ and $$[S_3]=\omega_{(1,0,0,0)} \ast \omega_{(1,0,0,0)} \ast \omega_{(0,1,0,0)} \ast \omega_{(0,0,1,0)} \ast \omega_{(0,0,0,1)} .$$

Notice that for $V=\{(1,1,1,1),(1,0,0,0)\}$, we have that $S_1$ if of level $V$, while $S_2,S_3$ are not.

\end{esempio}

\subsection{Regular elements and Möbius function for admissible tori}

In this paragraph we give to $\mathcal{Z}$ the structure of a locally finite poset, with the ordering induced by inclusion and we introduce the associated Möbius function $$\mu: \mathcal{Z} \times \mathcal{Z} \to \Z $$ and we recall more generally some properties of the Möbius function of a locally finite poset. The Möbius function $\mu$ is going to be one the main technical ingredient in the proof of our Theorem \ref{mainteo}.

\begin{oss}
The Möbius function $\mu$ had already been studied in \cite[Section 4.2]{HA}, in the case of $|I|=1$, where the authors used it to compute cohomology of generic character stacks. In this case, the only values that are needed are the values $\mu(Z_n,S)$, which have already been computed in \cite{hanlon} (see Proposition \ref{reviewhanlon} for more details).

However, to extend their result to the non-generic case in this article we needed a better understanding of the values $\mu(S,S')$ for any admissible subtori $S,S'$ and any $I$.

The next paragraphs develop the necessary tools to obtain the description of the Möbius function $\mu$ that we will need in the proof of Theorem \ref{mainteo}.
\end{oss}

\subsubsection{Poset of $F$-stable admissible subtori}

For any two  elements $S,S'$ of $\mathcal{Z}_{\alpha}$, we say that $S \leq S'$ if  $S \subseteq S'$. Notice that $Z_{\alpha} \leq S$ for any admissible  $S \subseteq \Gl_{\alpha}$. For any $S \in \mathcal{Z}$, we denote by $S^{reg}$ the subset of \textit{regular elements} of $S$ defined as \begin{equation}
\label{regular}
S^{reg}\coloneqq \{s \in S \ | \ s \notin S' \text{ for any } S'\lneq S, \ S'\in \mathcal{Z} \}.
\end{equation}

 We have the following disjoint union\begin{equation}\label{disjointunionregular1}\displaystyle S=\bigsqcup_{S'\leq S} (S')^{reg}\end{equation}
and so, taking $F$-fixed points, \begin{equation}\label{disjointunionregular}\displaystyle S^F =\bigsqcup_{S'\leq S} ((S')^{reg})^F.\end{equation}
 
 In particular, we have an equality $\displaystyle |S^F|=\sum_{S' \leq S}|((S')^{reg})^F|$. Notice that, if $[S]=\psi_{d_1}(\omega_{\beta_1}) \ast \cdots \ast\psi_{d_r}(\omega_{\beta_r})$, we have $$S^F=\prod_{j=1}^r (\mathbb{G}_m)_{d_j}(\F_q)=\prod_{j=1}^r\F_{q^{d_j}}^*  $$ and therefore we have
  \begin{equation}
\label{cardinalityadmissible}
|S^{F}|=P_{[S]}(q).
\end{equation}

\subsubsection{Möbius functions of locally finite posets}

For a finite poset $(X,\leq)$ denote by $$\mu_X: X \times X \to \Z$$ its associated Möbius functions.  Recall that $\mu_X$ is defined by the following two properties:

\begin{itemize}
\item $\mu(x,x)=1$ for each $x \in X$.
\item $\mu(x,y)=0$ if $x \not \leq y$.
\item For each $x \lneqq x'$, we have \begin{equation}
\label{mobius111}
\sum_{x \leq x'' \lneqq x'}\mu(x,x'')=-\mu(x,x')
\end{equation}
\end{itemize}

The Möbius function has the following property.

\begin{prop}
Given $f_1,f_2:X \to \C$ such that $$\displaystyle f_1(x)=\sum_{x'\leq x}f_2(x') ,$$ we have an equality 
\begin{equation}
\label{mobius}
f_2(x)=\sum_{x'\leq x}f_1(x')\mu_X(x',x).
\end{equation}
\end{prop}

Lastly, we recall the following standard Lemma about Möbius functions.

\begin{lemma}
\label{lemmacorrectionmobius}
Let $(X,\leq),(Y,\leq)$ be two locally finite posets and equip $X \times Y$ with the ordering  defined as $(x,y) \leq (x',y')$ if and only if $x \leq x'$ and $y \leq y'$.

For the locally finite poset $(X \times Y,\leq )$, we have
\begin{equation}
\label{mobiusproduct}
\mu_{X \times Y}((x,y),(x',y'))=\mu_X((x,x'))\mu_Y((y,y')).
\end{equation}
\end{lemma}

\begin{proof}
If $(x,y) \not \leq (x',y')$, both sides of eq.(\ref{mobiusproduct}) are $0$, since in this case we have $x \not \leq x'$ or $y \not \leq y'$. We can therefore assume that $(x,y) \leq (x',y')$.

By induction we can assume that \begin{equation}\mu_{X \times Y}((x,y),(x'',y''))=\mu_X((x,x''))\mu_Y((y,y''))\end{equation} for all $(x,y)< (x'',y'')< (x',y')$.
From eq.(\ref{mobius111}) we have therefore 

\begin{equation}
\mu_{X \times Y}((x,y)(x',y'))=-\sum_{(x,y)\leq (x'',y'')< (x',y')}\mu_{X}(x,x'')\mu_Y(y,y'')=
\end{equation}

\begin{equation}
\label{mobius222}
-\left(\sum_{x\leq x'' < x'}\mu_X(x,x'')\sum_{y \leq y'' \leq y'}\mu_Y(y,y'')+\mu_X(x,x')\sum_{y \leq y'' < y'}\mu_Y(y,y'')\right).    
\end{equation}
By eq.(\ref{mobius111}), $\displaystyle \sum_{y \leq y'' \leq y'}\mu_Y(y,y'')=0$ and $\sum_{y \leq y'' < y'}\mu_Y(y,y'')=-\mu_Y(y,y')$ and therefore 

\begin{equation}
\mu_{X \times Y}((x,y)(x',y'))=-\mu_X(x,x')\sum_{y \leq y'' < y'}\mu_Y(y,y'')=\mu_X(x,x')\mu_Y(y,y').
\end{equation}

\end{proof}

For $x \in X$, denote by $[x,\infty]_X \subseteq X$ the poset $$[x,\infty]=\{x' \in X \ | \ x' \geq x \} .$$

Notice that, for each $x' \in [x,\infty]_X$, from eq.(\ref{mobius111}), we deduce that we have:

\begin{equation}
\label{intervalmobius}
 \mu_{X}(x,x')=\mu_{[x,\infty]_X}(x,x')   
\end{equation}

\subsubsection{Möbius function for admissible subtori}
\label{mobiusetti}

The ordering $\leq$ endows the set $\mathcal{Z}$ with the structure of a locally finite poset. We denote by $$\mu(-,-): \mathcal{Z} \times \mathcal{Z} \to \Z$$ the associated Möbius function.

\vspace{6 pt}

\begin{esempio}
Let $f_1,f_2:\mathcal{Z} \to \C$ be the functions defined as $$f_1(S)=|S^F| $$ and $$f_2(S)=|(S^{reg})^F| .$$ By eq.(\ref{mobius}) and eq.(\ref{disjointunionregular}), we have the following identity:
\begin{equation}
\label{cardinalityregular}
|(S^{reg})^F|=\sum_{S'\leq S}|(S')^F|\mu(S',S)=\sum_{S' \leq S} P_{[S']}(q)\mu(S',S).
\end{equation}

\end{esempio}

\subsection{Levi subgroups and graphs}
\label{graph}
In this  paragraph, we will associate a finite graph $\Gamma_S$ to each admissible subtorus $S$. This construction will be useful to understand the Mobius function $\mu: \mathcal{Z} \times \mathcal{Z} \to \Z$ and to develop the combinatorial arguments of \cref{inclusion1}, both of which will be key elements in our proof of Theorem \ref{mainteo} about Log compatible functions.

\subsubsection{Notations on graphs}

Let $\Gamma$ be a finite graph where $M$ is its set of vertices and $m=|M|$. We say that $\Gamma$ is of type $K_m$ if it is the complete graph associated to $M$, i.e. each pair of distinct vertices is connected exactly by one edge.

We say that $\Gamma$ is \textit{admissible} if each of its connected components is of type $K_d$ for some $d$. 

\begin{oss}
\label{admissiblecondition}
Notice that the property of being admissible for a graph $\Gamma$ can be stated in the following equivalent way. 

For any two $m,m'\in M$, there is at most one edge of $\Gamma$ joining $m$ to $m'$ and, if $m_1,m_2,m_3 \in M$ are such that there is an edge of $\Gamma$ between $m_1$ and $m_2$ and an edge of $\Gamma$ between $m_2$ and $m_3$, there is an edge of $\Gamma$ between $m_1$ and $m_3$.  

\end{oss}

\subsubsection{Admissible graphs and maximal tori}
\label{madremia}
Let now $\alpha \in \N^I$ and $m=|\alpha|$ and fix an $F$-stable  maximal torus $T \subseteq \Gl_{\alpha} \subseteq \Gl_{m}$. 

\vspace{6 pt}

 Denote simply by $\mathcal{B},\mathcal{B}_i,\Phi,\Phi^+$ the sets $\mathcal{B}(T),\mathcal{B}_i(T),\Phi(T),\Phi^+(T)$ and by $\sigma \in S_m$ the permutation such that $F(\epsilon_j)=q\epsilon_{\sigma(j)}$ for each $\epsilon_j \in \mathcal{B}$, introduced in \cref{twistedcharactergroup}. Recall that we have a decomposition $$\mathcal{B}=\bigsqcup_{i \in I} \mathcal{B}_i .$$  

\vspace{8 pt}
 For any two admissible graphs $\Gamma,\Gamma'$ with set of vertices $\mathcal{B}$ and sets of edges $\Omega_{\Gamma},\Omega_{\Gamma'}$ respectively, we say that $\Gamma \leq \Gamma'$ if  $\Omega_{\Gamma} \supseteq \Omega_{\Gamma'}$.

\vspace{6 pt}
 
 We denote by $A(\mathcal{B},\sigma)$  the poset of admissible and $\sigma$-stable graph with set of vertices $\mathcal{B}$.  

Here $\sigma$-stable means that  $\Gamma$ has an edge between $\epsilon_i$ and $\epsilon_j$ if and only if it has an edge between $\epsilon_{\sigma(i)}$ and $\epsilon_{\sigma(j)}$.

 We denote by $\mu_{\mathcal{B},\sigma}(-,-)$ the associated Möbius function. Moreover, we will denote the complete graph with vertices $\mathcal{B}$ by $\Gamma_{\alpha} \in A(\mathcal{B},\sigma)$.

\vspace{6 pt}

\begin{oss}
\label{remarktardivo}
From Remark \ref{admissiblecondition}, we see that the poset $A(\mathcal{B},\sigma)$ is the the poset of $\sigma$-stable partitions of the set $\mathcal{B}$ with ordering  given by the reversed inclusion, i.e. the fixed point set lattice considered in \cite{hanlon}.

In the latter article, the authors computed certain values of the Möbius function $\mu_{\mathcal{B},\sigma}$ and in particular the values $\mu_{\mathcal{B},\sigma}(\Gamma_{\alpha},\Gamma')$ for each $\Gamma'$. We will review this result in Proposition \ref{reviewhanlon}. 

We prefer to introduce this graph theoretic description, as in our opinion this can ease the notations and give a more direct and visual understanding of the results of this Paragraph about the relationship between admissible graphs and admissible tori.
\end{oss}

\vspace{6 pt}

Fix now an admissible $\sigma$-stable 
 graph $\Gamma$ with set of vertices $\mathcal{B}$.
  Notice that, as $\Gamma$ is $\sigma$-stable,  $\sigma$ acts by permutation on the set of connected components of $\Gamma$. Assume that this action has $r$ orbits  of length $d_1,\dots,d_r$ respectively, which we denote by $O_1,\dots,O_r$.
  
  For each $j=1,\dots,r$, denote by $\mathcal{B}^{\Gamma}_j \subseteq \mathcal{B}$ the set of vertices contained in the orbit $O_j$. Notice that each $\mathcal{B}^{\Gamma}_j$ is $\sigma$-stable and there is an equality $$ \displaystyle \mathcal{B}=\bigsqcup_{j=1}^r \mathcal{B}^{\Gamma}_j .$$

 For each $j=1,\dots,r$,  choose  a partition of $\mathcal{B}^{\Gamma}_{j}$ into $d_j$ subsets $$\mathcal{B}^{\Gamma}_{j}=\mathcal{B}^{\Gamma}_{j,1}\bigsqcup \dots \bigsqcup \mathcal{B}^{\Gamma}_{j,d_j}$$ such that:
 \begin{itemize}
 \item Each $ \mathcal{B}^{\Gamma}_{j,h}$ is given by the vertices of a connected component belonging to the orbit $O_j$  
 \item We have $\sigma(\mathcal{B}^{\Gamma}_{j,h})=\mathcal{B}^{\Gamma}_{j,h+1}$ for each $h=1,\dots,d_j$ (here we consider the indices modulo $d_j$).
 \end{itemize}

For each $j=1,\dots,r$, let $\beta_j \in \N^I$ be the element  defined as $$(\beta_j)_i\coloneqq|\mathcal{B}^{\Gamma}_{j,1} \cap \mathcal{B}_i|$$ for $i \in I$. We denote by $\omega_{\Gamma}\in \mathbb{T}^{ss}_{\alpha}$ the semisimple multityped defined as

$$\omega_{\Gamma}\coloneqq \psi_{d_1}(\omega_{\beta_1}) \ast \cdots \ast \psi_{d_r}(\omega_{\beta_r}) .$$

The results of \cite[Theorem 4]{hanlon} imply the following Proposition.

\begin{prop}
\label{reviewhanlon}
For each $\Gamma \in A(\mathcal{B},\sigma)$, we have

\begin{equation}
\mu_{A(\mathcal{B},\sigma)}(\Gamma_{\alpha},\Gamma)=C^o_{\omega_{\Gamma}}.
\end{equation}
\end{prop}

\vspace{6 pt}
 
 Denote now by $\Gamma_{j,h}$ the restriction of $\Gamma$ to the set $\mathcal{B}^{\Gamma}_{j,h}$. Notice that $\Gamma_{j,h}$ is the complete graph with vertices $\mathcal{B}^{\Gamma}_{j,h}$ and so $\Gamma$ is totally determined by the subsets  $\{\mathcal{B}^{\Gamma}_{j,1}\}_{j=1,\dots,r}$.  Notice, in addition, that for each $j=1,\dots,r$, we have $\sigma^{d_j}(\mathcal{B}^{\Gamma}_{j,1})=\mathcal{B}^{\Gamma}_{j,1}$.

We have the following Lemma.

\begin{lemma}
\label{mobiustorigraphs}
There is an equivalence of posets \begin{equation}
\label{productburnout}
[\Gamma,\infty]_{A(\mathcal{B},\sigma)} \cong \prod_{j=1}^r [\Gamma_{j,1},\infty]_{A(\mathcal{B}^{\Gamma}_{j,1},\sigma^{d_j})}
\end{equation}
and, for each $\Gamma '\geq \Gamma$, denoting by $\Gamma'_{j,h}$ the restriction of $\Gamma'$ to $\mathcal{B}^{\Gamma}_{j,h}$, we have 
\begin{equation}
\label{mobiustorigraphs1}
\mu_{\mathcal{B},\sigma}(\Gamma,\Gamma')=\prod_{j=1}^r\mu_{\mathcal{B}^{\Gamma}_{j,1},\sigma^{d_j}}(\Gamma_{j,1},\Gamma'_{j,1}) .\end{equation}

\end{lemma}

\begin{proof}
Notice indeed that, given admissible graphs $\Gamma'_{j,1}$ with vertices $\mathcal{B}^{\Gamma}_{j,1}$ which are $\sigma^{d_j}$-stable, for each $j=1,\dots,r$ , there exist a unique $\sigma$-stable and admissible graph $\Gamma'$ with vertices $\mathcal{B}$ 
containing as subgraphs $\Gamma'_{1,1},\dots,\Gamma'_{r,1}$ and such that $\Gamma' \geq \Gamma$.

Eq.(\ref{mobiustorigraphs1}) is thus a consequence of eq.(\ref{productburnout}) and Lemma \ref{lemmacorrectionmobius}.

\end{proof}

\subsubsection{Admissible subtori and admissible graphs}
\label{graphdefinitions}
Fix now an  admissible torus $S \subseteq T$. Denote by $J_S \subseteq \Phi$ the subset $$ J_S\coloneqq \{\epsilon \in \Phi \ | \ S \subseteq \Ker(\epsilon) \} .$$

From Lemma \ref{levisubgroups} we deduce that we have $$ S=\bigcap_{\epsilon \in J_S} \Ker(\epsilon) $$ and $$L_S=T \prod_{\epsilon \in J_S}U_{\epsilon} .$$ Moreover, the subgroup $S$ is $F$-stable if and only if $J_S$ is $\sigma$-stable.
 
\vspace{ 8 pt}
 We now associate the following graph $\Gamma_S$ to the admissible torus $S$.  
 
 \begin{itemize}
\item  The set of vertices of  $\Gamma_S$ is $\mathcal{B}$  \item  $\Gamma_S$ has an edge between the vertices $\epsilon_i$ and $\epsilon_j$ if and only if $\epsilon_{i,j} \in J_S \cap \Phi^+$.

 \end{itemize}
 We denote by $\Omega_{\Gamma_S}$ be the set of edges of $\Gamma_S$. The group $S$ is $F$-stable if and only if $\Gamma_S$ is $\sigma$-stable.

\begin{esempio}
Let $I=\{ \cdot\}$ and $T$ be the torus of diagonal matrices $T \subseteq \Gl_m$. In this case, $\sigma$ is trivial.

The graph $\Gamma_T$ is thus the graph with no edges and $m$ vertices, while the graph $\Gamma_{Z_{\Gl_m}}$ associated to $Z_{\Gl_m}$ is the complete graph with $m$ vertices $K_m$. 
\end{esempio}

\vspace{4 pt}

\begin{esempio}
For any $I$ and any $\alpha \in \N^I$, we have that $\Gamma_{Z_{\alpha}}=\Gamma_{\alpha}$.

\end{esempio}

 We can now state the following Lemma, relating admissible graphs and subtori.
\begin{lemma}
\label{lemmaadmissible}
For any admissible torus $S$, the graph $\Gamma_S$ is admissible. Conversely, for any $\sigma$-stable admissible graph $\Gamma$ with set of vertices $\mathcal{B}$, there is a unique $F$-stable admissible torus $S \subseteq T$ such that $\Gamma_S=\Gamma$.
 \end{lemma}

 \begin{proof}
 
For an admissible subtorus $S$, we have that if $\epsilon_{i,j},\epsilon_{j,h} \in J_S$ then $\epsilon_{i,h} \in J_S$. From Remark \ref{admissiblecondition} we deduce that $\Gamma_S$ is admissible.

\vspace{6 pt}
 
 Consider now an admissible $\Gamma$ and the subset $$J_{\Gamma} \coloneqq \{\epsilon_{j,h} \in \Phi \ | \ \text{there is an edge of } \Gamma \text{ which has vertices } \epsilon_j, \epsilon_h  \} .$$ 

From \cite[Corollary 3.3.4]{DM}, the subset $J_{\Gamma}$ is a root subsystem and, from \cite[Lemma 8.4.2]{Springer}, we have that the torus \begin{equation}S_{\Gamma} \coloneqq \bigcap_{\epsilon \in J_{\Gamma}} \Ker(\epsilon)\end{equation} is admissible and $F$-stable with \begin{equation}
\label{Levigraph}
L_{S_{\Gamma}}=T\prod_{\epsilon \in J_{\Gamma}}U_{\epsilon} .\end{equation} It is not difficult to check that the graph associated to $S_{\Gamma}$ is $\Gamma$.

\end{proof}

\vspace{10 pt}

Let $S,S'$ be two admissible subtori such that $S \supseteq T, S'\supseteq T$. From Lemma \ref{lemmaadmissible}, we deduce the following Proposition

\begin{prop}
 \label{inclusionprop}
 Given $S, S' \subseteq T$, we have that $S \leq S'$ if and only if $\Gamma_S \leq \Gamma_{S'}$. 
\end{prop}

From Proposition \ref{inclusionprop}, we deduce the following Lemma.

\begin{lemma}
\label{mobiusprahs}
For any $S,S' \in \mathcal{Z}$ such that $S,S' \subseteq T$, we have an equality
\begin{equation}
\mu_{A(\mathcal{B},\sigma)}(\Gamma_S,\Gamma_{S'})=\mu(S,S')\end{equation} 

\end{lemma}

\vspace{6 pt}

 Consider now an admissible graph $\Gamma \in A(\mathcal{B},\sigma)$ and  the admissible $F$-stable torus $S_{\Gamma}$ associated to $\Gamma$. We have the following proposition.

\begin{prop}
\label{propositionmultitypeadmissibletori}
With the notations introduced in \cref{madremia},  we have $$[S_{\Gamma}]=\omega_{\Gamma} .$$
\end{prop}

\vspace{8 pt}

Fix now $\Gamma',\Gamma \in A(\mathcal{B},\sigma)$ such that $\Gamma \leq \Gamma'$, and denote by $S=S_{\Gamma}$ and $S'=S_{\Gamma'}$. We use the notations introduced before Lemma \ref{mobiustorigraphs}. We have $S \subseteq S'$. Assume that the torus $S$ is $\Gl_{\alpha}(\F_q)$ conjugated to $\displaystyle \prod_{j=1}^r (Z_{\beta_j})_{d_j}$. 

The permutation $\sigma^{d_j}:\mathcal{B}^{\Gamma}_{j,1} \to \mathcal{B}^{\Gamma}_{j,1}$ determines an associated $F$-stable subtorus $T_j \subseteq \Gl_{\beta_j}$, as explained in \cref{twistedFrobeniuspara},\cref{twistedcharactergroup}. The admissible graphs $\Gamma'_{j,1}$ correspond to admissible tori $S'_j \subseteq T_j \subseteq \Gl_{\beta_j}$, for each $j=1,\dots,r$. 

From Lemma \ref{mobiusprahs} and Lemma \ref{mobiustorigraphs}, we deduce the following Proposition
\begin{prop}
\label{inclusion}
We have an equality
\begin{equation}
\displaystyle \mu(S,S')=\prod_{j=1}^r \mu(Z_{\beta_j},S'_j)=\prod_{j=1}^r C^o_{[S'_j]},
\end{equation}
where the last equality is a consequence of Proposition \ref{reviewhanlon}.
\end{prop}
\vspace{6 pt}

\begin{esempio}
\label{exampletoriputlate}
Consider the set $I=\{1,2,3,4\}$, the dimension vector $\alpha=(2,1,1,1)$ and the admissible tori $S_1,S_2,S_3 \in \mathcal{Z}_{\alpha}$ of Example \ref{exampleadmissible2}. Notice that $S_1,S_2,S_3$ are all contained in the maximal torus $T=T_2 \times \mathbb{G}_m \times \mathbb{G}_m \times \mathbb{G}_m$, where $T_2 \subseteq \Gl_2$ is the maximal torus of diagonal matrices. More precisely, we have that $T=S_3$.

With the notations just introduced, we have $\mathcal{B}=\{ \epsilon_1,\epsilon_2, \epsilon_3,\epsilon_4, \epsilon_5\}$, $\sigma$ is the identity and $$\mathcal{B}_1=\{\epsilon_1,\epsilon_2\} \ \ \  \mathcal{B}_2=\{\epsilon_3\} $$ and $$\mathcal{B}_3=\{\epsilon_4\} \ \ \ \mathcal{B}_4=\{\epsilon_5\} .$$

The graph $\Gamma_{S_1}$ associated to the torus $S_1$ is 
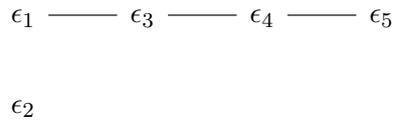
\begin{center}
\begin{tikzcd}
\epsilon_1 \arrow[dash]{r} \arrow[bend left=30, dash]{rr} 
\arrow[bend left=40, dash]{rrr}
&\epsilon_3 \arrow[dash]{r} \arrow[bend right=30, dash]{rr}  &\epsilon_4 \arrow[dash]{r} &\epsilon_5\\
\epsilon_2
\end{tikzcd}
\end{center}

The graph $\Gamma_{S_2}$ associated to the torus $S_2$ is 

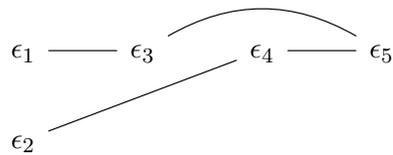
\begin{center}
\begin{tikzcd}
\epsilon_1 \arrow[dash]{r} \arrow[bend left=50, dash]{rrr} &\epsilon_3 \arrow[bend left=30, dash]{rr}  &\epsilon_4  &\epsilon_5\\
\epsilon_2 \arrow[dash]{urr}
\end{tikzcd}
\end{center}

The graph associated to $S_3$ is 
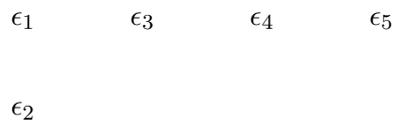
\begin{center}
\begin{tikzcd}
\epsilon_1  &\epsilon_3  &\epsilon_4  &\epsilon_5\\
\epsilon_2
\end{tikzcd}
\end{center}

Notice that $\Gamma_{S_1},\Gamma_{S_2} \leq \Gamma_{S_3}$ and we have corresponding inclusions $S_1,S_2 \subseteq S_3$

\end{esempio}

\vspace{6 pt}

\begin{esempio}
Consider $I=\{1,2\}$ and the dimension vector $\alpha=(2,2) \in \N^I$, i.e. $\Gl_{\alpha}=\Gl_2 \times \Gl_2$ and let $T \subseteq \Gl_{\alpha}$ be the torus $$T=T_{\epsilon} \times T_{\epsilon} \subseteq \Gl_2 \times \Gl_2 ,$$
where $T_{\epsilon} \subseteq \Gl_2$ is the torus of Example \ref{twistedtorusGl_2}. In this case, we have $\mathcal{B}=\{\epsilon_1,\epsilon_2,\epsilon_3,\epsilon_4\}$ with $$\mathcal{B}_1=\{\epsilon_1,\epsilon_2\}  \ \ \ \mathcal{B}_2=\{\epsilon_3,\epsilon_4\} ,$$ and $\sigma$ is the permutation $\sigma=(1 2) (3 4) \in S_4$.

Let $S \subseteq \Gl_{\beta}$ be the admissible subtorus $$S=\Biggl\{\left(\begin{pmatrix}
\lambda &0\\ 0 &\lambda    
\end{pmatrix},\begin{pmatrix}
\mu &0\\ 0 &\mu    
\end{pmatrix}\right) \ | \ \lambda, \mu \in \overline{\F_q}^*\Biggr\} .$$

We have that $S \subseteq T$ and the graph $\Gamma_S$ is given by

\begin{center}
\begin{tikzcd}
\epsilon_1 \arrow[dash]{d} &\epsilon_3 \arrow[dash]{d}\\
\epsilon_2 &\epsilon_4
\end{tikzcd}
\end{center}

Notice that $\Gamma_S$ has two connected components which are both stabilized by $\sigma$. With the notations introduced before, we have therefore two orbits $O_1,O_2$ with $\mathcal{B}^{\Gamma_S}_1=\mathcal{B}_1$ and $\mathcal{B}^{\Gamma_S}_2=\mathcal{B}_2$ and $d_1=d_2=1$. Denote by $\Gamma_{S,1},\Gamma_{S,2}$ the restriction of $\Gamma_S$ to $\mathcal{B}_1,\mathcal{B}_2$ respectively.

Notice moreover that the associated elements $\beta_1,\beta_2$ are given by $$ \beta_1=(2,0) \ \ \ \beta_2=(0,2) $$ and, from Proposition \ref{propositionmultitypeadmissibletori} we have that $$[S]=\omega_{\beta_1} \ast \omega_{\beta_2} .$$

The torus $S$ is indeed $Z_{\beta_1} \times Z_{\beta_2}$ embedded block diagonally in $\Gl_{\alpha}$. 
\vspace{6 pt}

Moreover, from Proposition \ref{inclusion}, we deduce that
\begin{equation}
\label{exampleziopera}\mu(S,T)=\mu_{\mathcal{B}_1,\sigma}(\Gamma_{S,1},\Gamma_{T_{\epsilon}})\mu_{\mathcal{B}_2,\sigma}(\Gamma_{S,2},\Gamma_{T_{\epsilon}})=\mu(Z_{2},T_{\epsilon})^2 ,\end{equation} where $Z_2=Z_{\Gl_2}$.

Eq.(\ref{exampleziopera}) can be checked directly from the definition of the Möbius function $\mu$. We have $$\{Z_{\beta_1} \times Z_{\beta_2}, Z_{\beta_1} \times T_{\epsilon},T_{\epsilon} \times Z_{\beta_2}\}=\{S'' \in \mathcal{Z}_{\alpha} \ | \ Z_{\beta_1} \times Z_{\beta_2} \leq S'' \lneqq T_{\epsilon} \times T_{\epsilon} \} .$$

From. eq(\ref{mobius111}), we deduce that we have $$\mu(Z_{\beta_1} \times Z_{\beta_2},Z_{\beta_1} \times T_{\epsilon})=\mu(Z_{\beta_1} \times Z_{\beta_2},T_{\epsilon} \times Z_{\beta_2})=\mu(Z_2,T_{\epsilon})=-1 $$ and thus from eq.(\ref{mobius111}) that we have

$$\mu(S,T_{\epsilon} \times T_{\epsilon})=-(-1 -1+1)=1=\mu(Z_2,T_{\epsilon})^2 .$$
\vspace{6 pt}

Consider now the admissible and $\sigma$-stable graph $\Gamma'$ given by 
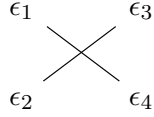
\begin{center}
\begin{tikzcd}
\epsilon_1 \arrow[dash]{dr} &\epsilon_3 \arrow[dash]{dl}\\
\epsilon_2 &\epsilon_4
\end{tikzcd}
\end{center}
and denote by $S'=S_{\Gamma'}$. The torus $S'$ is given by \begin{equation}\label{expressionS'}
S'=\Biggl\{\left(\dfrac{1}{x^q-x}\begin{pmatrix}
ax^q-bx &-a+b\\
(a-b)xx^q &-ax+bx^q
\end{pmatrix},\dfrac{1}{x^q-x}\begin{pmatrix}
ax^q-bx &-a+b\\
(a-b)xx^q &-ax+bx^q
\end{pmatrix}\right) \ | \ a,b \in \overline{\F_q}^* \Biggr\}  .\end{equation}

In this case, the graph $\Gamma'$ has $2$ connected components, which are swapped by $\sigma$. We have therefore a single orbit $O_1$ of length $d_1=2$  with $$\mathcal{B}^{\Gamma'}_{1,1}=\{\epsilon_1,\epsilon_3\} \ \ \  \mathcal{B}^{\Gamma'}_{1,2}=\{\epsilon_2,\epsilon_4\} .$$

Notice that the associated dimension vector $\beta_1'$ is $\beta_1'=(1,1)$.
 Proposition \ref{propositionmultitypeadmissibletori} states therefore that $S'$ is $\Gl_{\alpha}(\F_q)$-conjugated to the torus $$(Z_{(1,1)})_2 \subseteq \Gl_{\alpha} ,$$ which is also directly seen by the expression of $S'$ in eq.(\ref{expressionS'}).

Notice that $$\sigma^2=\mathrm{Id}:\mathcal{B}^{\Gamma'}_{1,1} \to \mathcal{B}^{\Gamma'}_{1,1} .$$ From Lemma \ref{mobiusprahs} and Lemma \ref{mobiustorigraphs}, we find therefore that $$\mu(S',T)=\mu_{\mathcal{B}^{\Gamma'}_{1,1},\mathrm{Id}}(\Gamma'_{1,1},\Gamma_{T_2})=\mu(Z_{(1,1)},Z_{(1,0)} \times Z_{(0,1)}) .$$ 

\end{esempio}

\subsection{Inclusion of admissible subtori}
\label{inclusion1}

Let $\nu,\nu' \in \mathbb{T}^{ss}_{\alpha}$ . Fix a maximal torus $T$ such that $S_{\nu'}\subseteq T \subseteq \Gl_{\alpha}.$ Define the set $P_{\nu,\nu'}$ as  $$P_{\nu,\nu'}\coloneqq \{S \in \mathcal{Z}_{\alpha} \   \ | \ [S]=\nu \ , S \leq S_{\nu'}\} .$$ In this paragraph we give a combinatorial description of $P_{\nu,\nu'}$ which will be used in the proof of Theorem \ref{mainteo}. 

\vspace{10 pt}

Assume that  $$\nu=\psi_{d_1}(\omega_{\beta_1})\ast\cdots \ast\psi_{d_r}(\omega_{\beta_r})$$ and $$\nu'=\psi_{d'_1}(\omega_{\beta'_1})\ast\cdots\ast\psi_{d_t'}(\omega_{\beta'_t}) $$ respectively. 

Up to reordering the factors in the product $\nu=\prod_{j=1}^r \psi_{d_j}(\omega_{\beta_j})$, we can assume that there exists a strictly increasing sequence $i_1< \dots <i_k \in \{1,\dots, r\}$   such that:
\begin{itemize}

\item $(d_j,\beta_j)=(d_1,\beta_1)$ for $j=1,\dots,i_1$

\item $(d_j,\beta_j)=(d_{i_p},\beta_{i_p})$ for all $i_{p-1} < j \leq i_{p}$ for $p \in \{2,\dots,k\}$

\item $(d_{i_p},\beta_{i_p}) \neq (d_{i_s},\beta_{i_s})$ for any $p \neq s$.
\end{itemize}

We have that $i_1=\nu((d_1,\beta_1))$ and, for each $h \in \{2,\dots,k\}$, we have that $\displaystyle i_h-i_{h-1}=\nu((d_{i_h},\beta_{i_h}))$. Let $M_{\nu,\nu'}$ be the set  of partitions  of $\{1,\dots ,t \}$ into $r$ non-empty disjoint subsets $X_1,\dots,X_r$ with the following properties:

\begin{itemize}
\item If $h$ belongs to  $X_i$, then $d_i |d'_h$
\item For every $j=1,\dots, r$, it holds $\displaystyle \sum_{h \in X_j}\dfrac{d'_h}{d_j} \beta'_h=\beta_j$.

\end{itemize}

We will denote the element of $M_{\nu,\nu'}$ associated to the subsets $X_1,\dots,X_r$ by $(X_1,\dots,X_r)$.

Consider now the group $W'_{\nu}$ defined as $$W'_{\nu} \coloneqq S_{\nu((d_{i_1},\beta_{i_1}))} \times \cdots \times S_{\nu((d_{i_k},\beta_{i_k}))} .$$ The set $M_{\nu,\nu'}$ is endowed with an action of the group $W'_{\nu}$ 
defined  as follows. Consider an element $\sigma=(\sigma_1,\dots,\sigma_k) \in W'_{\nu}$ and an element $(X_1,\dots,X_r) \in M_{\nu,\nu'}$. We put $$\sigma \cdot (X_1,\dots,X_r) \coloneqq (X_{\sigma_1(1)},\dots,X_{\sigma_1(i_1)},X_{\sigma_2(i_1+1)},\dots,X_{\sigma_k(r)}) .$$

The group $W_{\nu}'$ acts freely on $M_{\nu,\nu'}$. We denote by $\overline{M_{\nu,\nu'}}$ the quotient set $M_{\nu,\nu'}/W'_{\nu}$. We will now define the following morphism $$\pi_{\nu,\nu'}:P_{\nu,\nu'} \to \overline{M_{\nu,\nu'}} .$$

\vspace{8 pt}

We denote 
 by $\Gamma'$  the graph associated to $S_{\nu'}$ with respect to the torus $T$. Let $\{\mathcal{B}_{j,h}^{\Gamma'}\}_{\substack{j=1,\dots,t \\ h=1,\dots,d_j'}}$ be the partition of the set $\mathcal{B}$ introduced in paragraph \cref{graph} for the graph $\Gamma'$.
 
 Consider an $F$-stable admissible torus $S \subseteq S_{\nu'}$ with $[S]=\nu$ and the corresponding  $\sigma$-stable  graph $\Gamma \leq \Gamma'$, i.e. $\Omega_{\Gamma} \supseteq \Omega_{\Gamma'}$.  

Let $O_1,\dots,O_r$ be the $r$ orbits for the action of $\sigma$ on the connected components of $\Gamma$ of length $d_1,\dots,d_r$ respectively and assume to have fixed representatives $\Gamma_{1,1},\dots,\Gamma_{r,1}$ for each of the orbits.

For each $j=1,\dots,r$, there exists a subset $X_j \subseteq \{1,\dots,t\}$ and, for each $l \in X_j$, a subset $Z_l \subseteq \{1,\dots,d'_l\}$, such that $\Gamma_{j,1}$ is the complete graph with vertices $$\displaystyle \bigsqcup_{l \in X_j}\bigsqcup_{z \in Z_l}\mathcal{B}^{\Gamma'}_{l,z} .$$

The subsets $X_j$ do not depend on the choice of the representatives $\Gamma_{1,1},\dots,\Gamma_{r,1}$ and form a partition of the set $\{1,\dots,t\}$. The partition $(X_1,\dots,X_r)$ belongs to $M_{\nu,\nu'}$. 

Indeed, since the orbit $O_j$ has length $d_j$, we must have that $$\sigma^s\left(\displaystyle \bigsqcup_{l \in X_j}\bigsqcup_{z \in Z_l}\mathcal{B}^{\Gamma'}_{l,z}\right) \cap \displaystyle \bigsqcup_{l \in X_j}\bigsqcup_{z \in Z_l}\mathcal{B}^{\Gamma'}_{l,z}=\emptyset $$  
for any $0 < s \leq d_j-1$ and $$\sigma^{d_j}\left(\displaystyle \bigsqcup_{l \in X_j}\bigsqcup_{z \in Z_l}\mathcal{B}^{\Gamma'}_{l,z}\right)= \displaystyle \bigsqcup_{l \in X_j}\bigsqcup_{z \in Z_l}\mathcal{B}^{\Gamma'}_{l,z} .$$ Recall that $\sigma(\mathcal{B}^{\Gamma'}_{l,z})=\mathcal{B}^{\Gamma'}_{l,z+1}$, where the index $z$ of $\mathcal{B}^{\Gamma'}_{l,z}$ is  considered modulo $d_l'$. We deduce therefore that $Z_l$ is such that $$(Z_l +s) \cap Z_s=\emptyset \mod d_j $$ for each $ 0 < s \leq d_j-1$ and $$Z_l+d_j=Z_j \mod d_j .$$

This implies that $d_j | d_l'$ and that there exists $a_l \in Z_l$ such that $$Z_l=\left\{a_l+d_jk \ | \ k=1,\dots, \dfrac{d_l'}{d_j}\right\} .$$ 
 In particular, it holds that $|Z_l|=\dfrac{d_l'}{d_j}$, from which we deduce that $$\sum_{l \in X_j}\dfrac{d_l'}{d_j}\beta'_l=\beta_j .$$

We define then $$\pi_{\nu,\nu'}(S)\coloneqq [(X_1,\dots,X_r)] $$

where $[(X_1,\dots,X_r)]$ is the class of the element $(X_1,\dots,X_r)$ in the quotient $\overline{M_{\omega_1,\omega_2}}$.

The morphism $\pi_{\nu,\nu'}$ is well-defined, i.e. does not depend on the choice of an ordering of the orbits $O_1,\dots,O_r$  and it is surjective as we are taking the class $[(X_1,\dots,X_r)]$ in the quotient $\overline{M_{\nu,\nu'}}$ for the action of $W_{\nu}'$.

\vspace{10 pt}

From the description of the subsets $Z_l$ given above, we deduce that, for each $[(X_1,\dots,X_r)]$, the fiber $\pi_{\nu,\nu'}^{-1}([(X_1,\dots,X_r)])$ has cardinality \begin{equation}
\label{cardinality}
|\pi^{-1}_{\nu,\nu'}([(X_1,\dots,X_r)])|=\prod_{j=1}^r d_j^{|X_j|-1}.
\end{equation}

\section{Log compatible functions and plethystic identities}

\label{plethysticidentities}

In this Section, we recall the Definition of a Log compatible family, first introduced by Letellier \cite{letellierDT} and we prove our main Theorem \ref{mainteo} about these families, which will be the key tool  to compute the E-series  of non-generic character stacks.

We first introduce the following notation.

\vspace{6 pt}
Consider  a family of rational functions indexed by multypes  $\{F_{\omega}(t)\}_{\omega \in \mathbb{T}_{I}} \subseteq \Q(t)$. For any $V \subseteq \N^I$, we define the rational function $F_{\alpha,V}(t) \in \Q(t)$ in the following way:

\begin{equation}
\label{polynomialV}
F_{\alpha,V}(t)\coloneqq \sum_{\omega \in \mathbb{T}_{\alpha}}\dfrac{F_{\omega}(t)}{w(\omega)}\left(\sum_{\substack{S' \leq S_{\omega} \\ S' \text{ of level } V}}P_{[S']}(t) \mu(S',S_{\omega})\right).
\end{equation}

For $V=\{\alpha\}$ we will use the notation $F_{\alpha,gen}(t) \coloneqq F_{\alpha,\{\alpha\}}(t)$. From Proposition \ref{reviewhanlon}, we have an equality $$F_{\alpha,gen}(t)=\sum_{\omega \in \mathbb{T}_{\alpha}} \dfrac{F_{\omega}(t)}{w(\omega)}(t-1)\mu(Z_{\alpha},S_{\omega})=\sum_{\omega \in \mathbb{T}_{\alpha}}\dfrac{F_{\omega}(t)}{w(\omega)}(t-1)C^o_{\omega^{ss}}  .$$

\subsection{Plethysm and log compatibility: main result}

We give the following definition of a Log compatible family $\{F_{\omega}(t)\}_{\omega \in \mathbb{T}_I}$.

\begin{definizione}
\label{Logcompatible}
We say that $\{F_{\omega}(t)\}_{ \omega \in \mathbb{T}_I}$ is \textit{Log compatible} if for any $\alpha \in \N^I$, $\omega \in \mathbb{T}_{\alpha}$ and for every  multitypes $\nu_1,\dots,\nu_r$ and integers $d_1,\dots,d_r$ such that $\psi_{d_1}(\nu_1) \ast \cdots \ast \psi_{d_r}(\nu_r)=\omega$, we have $$ \prod_{j=1}^r F_{\nu_j}(t^{d_j})=F_{\omega}(t) .$$
\end{definizione}

  We have the following theorem:

\begin{teorema}
\label{mainteo}
For a Log compatible family $\{F_{\omega}(t)\}_{\omega \in \mathbb{T}_I}$ and any $V \subseteq \N^I$, we have the following equality:

\begin{equation}
\label{mainteo2}
\Plexp\left(\sum_{\beta \in V} F_{\beta,gen}(t)y^{\beta}\right)=\sum_{\alpha \in \N^I}F_{\alpha,V}(t)y^{\alpha}.
\end{equation}

\end{teorema}

\vspace{2 pt}

\begin{oss}
The notion of a Log compatible family and the definition of the polynomials $F_{\alpha,gen}(t)$ had already been introduced in \cite[Paragraph 2.1.2]{letellierDT}

Letellier  \cite[Theorem 2.2]{letellierDT} used these objects to show the case in which $V=\N^I$ of Theorem \ref{mainteo}. However, Letellier's proof does not seem to extend immediately to the case of any level $V$.

 Notice that, for $V=\N^I$, we have that $\displaystyle F_{\alpha,\N^I}(t)=\sum_{\omega \in \mathbb{T}_{\alpha}}\dfrac{F_{\omega}(t)}{w(\omega)}|(S_{\omega}^{reg})^F|$.

To obtain the equality $$\Plexp\left(\sum_{\beta \in \N^I} F_{\beta,gen}(t)y^{\beta}\right)=\sum_{\alpha \in \N^I}\sum_{\omega \in \mathbb{T}_{\alpha}}\dfrac{F_{\omega}(t)}{w(\omega)}|(S_{\omega}^{reg})^F| ,$$ the author \cite{letellierDT} uses in a key way two classical lemmas regarding plethystic operations, i.e. \cite[Lemma 2.1]{letellierDT} and \cite[Lemma 22]{mozgovoy}.

These two latter statements do not have an analog when we truncate the sum inside the plethystic exponential $\displaystyle\Plexp\left(\sum_{\beta \in V} F_{\beta,gen}(t)y^{\beta}\right)$ at a general $V$. For this reason, we needed to study the combinatorial objects introduced in \cref{chaptermultitypes} and \cref{subtoriandmultitypes}.

In particular, in the proof, we will need the following three main tools.

\begin{itemize}
\item We will use the ring of multitypes introduced in \cref{paragraphring} and in particular Lemma \ref{Vgenlevi} to better understand the truncated exponentials $\displaystyle\Plexp\left(\sum_{\beta \in V} F_{\beta,gen}(t)y^{\beta}\right)$.
\item We use the results of \cref{madremia} and in particular Proposition \ref{inclusion} to relate the terms of the form $C^o_{\omega_1}\cdots C^o_{\omega_r}$ appearing in the LHS of eq.(\ref{mainteo2}) to the terms of the form $\mu(S,S_{\omega})$ appearing in the RHS of eq.(\ref{mainteo2}).
\item We use the results of \cref{inclusion1} to relate combinatorially the sums of the type $$\displaystyle\sum_{\substack{S' \leq S_{\omega} \\ S' \text{ of level } V}}\mu(S',S_{\omega})P_{[S']}(t)$$ appearing in the RHS of eq.(\ref{mainteo2}) to the LHS of eq.(\ref{mainteo2}).
\end{itemize}

\end{oss}

\vspace{4 pt}

\begin{proof}[Proof of Theorem 4.1.2]
We are going to give the proof through multiple steps.

\textit{I Step: rewriting the LHS of eq.(\ref{mainteo2})}.

We start by using the results of \cref{lambdarings} concerning the ring of multitypes. In particular, we remark that there is a unique morphism $\Theta$ of $\lambda$-rings $$\Theta:\mathcal{K}^{ss}_I \to \Q(t)  $$ obtained by extending  $$ \Theta(\omega_{\alpha})=F_{\alpha,gen}(t) .$$
By Lemma \ref{Vgenlevi}, Identity (\ref{mainteo2}) is thus equivalent to the following Identity \begin{equation}
\label{mainteo3}
\sum_{\substack{\nu \in \mathbb{T}^{ss}_{\alpha} \\ \text{level } V}}\dfrac{\Theta(\nu)}{w(\nu)}=F_{\alpha,V}(t)
\end{equation}

 The RHS of eq.(\ref{mainteo3}) is given by \begin{equation}
\label{mainteo4}
F_{\alpha,V}(t)=\sum_{\omega \in \mathbb{T}_{\alpha}}\dfrac{F_{\omega}(t)}{w(\omega)}\left(\sum_{\substack{S \subseteq S_{\omega}\\  \text{ level } V }}P_{[S]}(t)\mu(S,S_{\omega})\right).
\end{equation}

We are left to understand the LHS of eq.(\ref{mainteo3}).

\vspace{8 pt}
\textit{II Step: understanding the LHS of eq.(\ref{mainteo3})}.

Consider now $\nu \in \mathbb{T}^{ss}_I$ and $d_1,\dots,d_r \in \N_{>0}$ and $\beta_1,\dots,\beta_r \in \N^I$ such that  $$\nu=\psi_{d_1}(\omega_{\beta_1}) \ast \cdots \ast \psi_{d_r}(\omega_{\beta_r}) .$$ The value of $\dfrac{\Theta(\nu)}{w(\nu)}$  is given by \begin{equation}
\label{eqlogcompatibilityuseful}
\dfrac{1}{w(\nu)}\prod_{j=1}^r \psi_{d_j}(\Theta(\beta_j))=\dfrac{1}{w(\nu)}\prod_{j=1}^r\left(\sum_{\omega_j \in \mathbb{T}_{\beta_j}}\dfrac{F_{\omega_{j}}(t^{d_j})}{w(\omega_j)}(t^{d_j}-1)\mu(Z_{\beta_j},S_{\omega_j})\right)=
\end{equation}

\begin{equation}
\label{mainteo5}
\sum_{\omega \in \mathbb{T}_{\alpha}}\dfrac{F_{\omega}(t)P_{\nu}(t)}{w(\nu)}\left(\sum_{\substack{\omega_1 \in \mathbb{T}_{\beta_1},\dots,\omega_r \in \mathbb{T}_{\beta_r}\\ \psi_{d_1}(\omega_1)\ast\cdots\ast\psi_{d_r}(\omega_r)=\omega}}\dfrac{1}{w(\omega_1)\cdots w(\omega_r)}\prod_{j=1}^r \mu(Z_{\beta_j},S_{\omega_j}) \right)
\end{equation}

where the equality of eq.(\ref{eqlogcompatibilityuseful}) and eq.(\ref{mainteo5}) is a direct consequence of the Log compatibility of the family $\{F_{\omega}(t)\}_{\omega \in \mathbb{T}_I}$. 

\vspace{4 pt}

\textit{III Step: the function $f_{\nu,\omega}$}.

Fix  a multitype $\omega \in \mathbb{T}_{\alpha}$ with $\omega=(d'_1,\bm \lambda_1)\cdots(d_t',\bm\lambda_t)$ for multipartitions $\bm\lambda_1 , \dots ,\bm\lambda_t$ 
and its associated admissible torus $S_{\omega} \subseteq \Gl_{\alpha}$. 

Denote by $H_{\nu,\omega}$ the set defined by $$H_{\nu,\omega}\coloneqq \{(\omega_1,\dots,\omega_r) \in   \mathbb{T}_{\beta_1}\times \cdots \times \mathbb{T}_{\beta_r} \ | \  \psi_{d_1}(\omega_1)\ast\cdots \ast \psi_{d_r}(\omega_r)=\omega\} $$ and by $\delta_{\nu}:H_{\nu,\omega} \to \Z$ the function defined as $$\delta_{\nu}((\omega_1,\dots,\omega_r))\coloneqq \prod_{j=1}^r \mu(Z_{\beta_j},S_{\omega_j}) .$$

Let $M_{\nu,\omega^{ss}}$ be the set introduced in \cref{inclusion1}. Consider the following function $$f_{\nu,\omega}:M_{\nu,\omega^{ss}} \to H_{\nu,\omega}$$ defined as: $$f_{\nu,\omega}((X_1,\dots,X_r))=(\omega_1,\dots,\omega_r) $$ where $$\omega_i(d,\bm\lambda)=\#\{h \in X_i \text{ such that } \left(\dfrac{d'_h}{d_i},\bm\lambda_h\right)=(d,\bm\lambda)\}$$ for every $(d,\bm\lambda) \in \N \times \mathcal{P}^I$.

\vspace{6 pt}

We will now show that the function $f_{\nu,\omega}$ is surjective and for each $(\omega_1,\dots,\omega_r)$, the cardinality of the fiber is given by \begin{equation}
|f^{-1}_{\nu,\omega}(\omega_1,\dots,\omega_r)|=\prod_{(d,\bm\lambda) \in \N \times \mathcal{P}^I}\dfrac{\omega(d,\bm\lambda)!}{\omega_1(d,\bm\lambda)!\cdots \omega_r(d,\bm\lambda)!}
\end{equation}

As done in \cref{inclusion1} for the semisimple multitype $\nu$, for the multitype $\omega=(d_1',\bm \lambda_1')\dots (d_t',\bm \lambda_t')$ there exists a a strictly increasing sequence $j_1< \dots <j_c \in \{1,\dots, t\}$   such that:
\begin{itemize}

\item $(d'_j,\bm \lambda'_j)=(d'_1,\bm\lambda'_1)$ for $j=1,\dots,j_1$.

\item $(d'_j,\bm\lambda'_j)=(d'_{j_p},\bm \lambda'_{j_p})$ for all $j_{p-1} < j \leq j_{p}$ for $p \in \{2,\dots,c\}$.

\item $(d'_{j_p},\bm \lambda'_{j_p}) \neq (d'_{j_s},\bm \lambda_{j_s}')$ if $p \neq s$.
\end{itemize}

Recall that $j_p-j_{p-1}=\omega((d_{j_p}',\bm \lambda_{j_p}'))$ for every $1 \leq p \leq c$, where we put $j_0=0$. To prove the surjectivity of $f_{\nu,\omega}$, fix $(\omega_1,\dots,\omega_r) \in H_{\nu,\omega}$ and define  subsets $X_{\omega_1},\dots,X_{\omega_r} \subseteq \{1,\dots,t\}$ as follows.

$$X_{\omega_j}=\bigsqcup_{l=1}^c \{h \in \{j_{l-1},\dots,j_l\} \ | \ \sum_{m=1}^{j-1}\psi_{d_m}(\omega_m)((d_{j_l}',\bm \lambda_{j_l}'))< h-j_{l-1} \leq \sum_{m=1}^{j}\psi_{d_m}(\omega_m)((d_{j_l}',\bm \lambda_{j_l}')) \} .$$

We have that $(X_{\omega_1},\dots,X_{\omega_r}) \in H_{\nu,\omega}$ and $f_{\nu,\omega}((X_{\omega_1},\dots,X_{\omega_r}))=(\omega_1,\dots,\omega_r)$. 

\vspace{2 pt}

Through a similar argument, we see that the cardinality of $|f^{-1}_{\nu,\omega}(\omega_1,\dots,\omega_r)|$ is the product, for $l=1,\dots,c$ of the number of partitions of a set of cardinality $\omega((d_{j_l}',\bm \lambda_{j_l}'))=j_l-j_{l-1}$ into $r$ sets of cardinality $\psi_{d_1}(\omega_1)((d_{j_l}',\bm \lambda_{j_l}')),\dots,\psi_{d_r}(\omega_r)((d_{j_l}',\bm \lambda_{j_l}'))$ respectively, i.e.
\begin{equation}
\label{madonnasanta1}
|f^{-1}_{\nu,\omega}(\omega_1,\dots,\omega_r)|=\prod_{l=1}^c \dfrac{\omega((d_{j_l}',\bm \lambda_{j_l}'))!}{(\psi_{d_1}(\omega_1)((d_{j_l}',\bm \lambda_{j_l}')))!\cdots (\psi_{d_r}(\omega_r)((d_{j_l}',\bm \lambda_{j_l}')))!}=
\end{equation}
\begin{equation}
=\prod_{(d,\bm\lambda) \in \N \times \mathcal{P}^I}\dfrac{\omega(d,\bm\lambda)!}{\omega_1(d,\bm\lambda)!\cdots \omega_r(d,\bm\lambda)!}.
\end{equation}

\vspace{4 pt}

\textit{IV Step: describe the summands inside the parenthesis of eq.(\ref{mainteo5}) for any $\omega \in \mathbb{T}_{\alpha}$}.

\vspace{6 pt}

For any $(\omega_1,\dots,\omega_r) \in H_{\nu,\omega}$, we have the following equality: \begin{equation}
\label{eq1}
\prod_{(d,\bm\lambda) \in \N \times \mathcal{P}^I}\dfrac{\omega(d,\bm\lambda)!}{\omega_1(d,\bm\lambda)!\cdots \omega_r(d,\bm\lambda)!}=\dfrac{w(\omega)}{w(\omega_1)\cdots w(\omega_r)}\dfrac{\prod_{j=1}^r \prod_{h \in X_j}\left(\frac{d'_h}{d_j}\right)}{\prod_{l=1}^t d'_l}.
\end{equation}

As $\displaystyle \bigsqcup X_j=\{1,\dots,t\}$, the right hand side of eq.(\ref{eq1}) is equal to \begin{equation}
\label{eq2}
\dfrac{w(\omega)}{w(\omega_1)\cdots w(\omega_r)}\dfrac{1}{\prod_{j=1}^r d_j^{|X_j|}}.
\end{equation}
For an element $m=(X_1,\dots,X_r) \in M_{\nu,\omega^{ss}}$, denote by $d_m=\prod_{j=1}^r d_j^{|X_j|}$.

For each $\omega \in \mathbb{T}_{\alpha}$, we can thus rewrite the corresponding summand appearing in the RHS of eq.(\ref{mainteo5}) as:
\begin{equation}
\label{mainteo6}
\dfrac{F_{\omega}(t)P_{\nu}(t)}{w(\nu)}\left(\sum_{(\omega_1,\dots,\omega_r)\in H_{\nu,\omega}}\dfrac{1}{w(\omega_1)\cdots w(\omega_r)}\delta_{\nu}((\omega_1,\dots,\omega_r))\right)=
\end{equation}

\begin{equation}
\label{eq3}
\dfrac{F_{\omega}(t)P_{\nu}(t)}{w(\nu)w(\omega)}\left(\sum_{m \in M_{\nu,\omega^{ss}}} \dfrac{w(\omega)}{w(\omega_1)\cdots w(\omega_r)}\dfrac{\delta_{\nu}(f_{\nu,\omega}(m))}{|f_{\nu,\omega}^{-1}(m)|} \right)=\dfrac{F_{\omega}(t)P_{\nu}(t)}{w(\nu)w(\omega)}\left(\sum_{m \in M_{\nu,\omega^{ss}}} \delta_{\nu}(f_{\nu,\omega}(m))d_m \right),
\end{equation}
where the last equality is a consequence of eq.(\ref{madonnasanta1}), eq(\ref{eq1}),eq(\ref{eq2}).
\vspace{2 pt}

\textit{V Step: use the results of \cref{inclusion1} to rewrite the RHS of eq.(\ref{eq3})}

The set $H_{\nu,\omega}$ is endowed with the following  action of $W'_{\nu}$. We use the notations of \cref{inclusion1}. An element $\sigma=(\sigma_1,\dots,\sigma_k) \in W'_{\nu}$ acts on $(\omega_1,\dots,\omega_r) \in H_{\nu,\omega}$ by $$\sigma \cdot (\omega_1,\dots,\omega_r)=(\omega_{\sigma_1(1)},\dots,\omega_{\sigma_1(i_1)},\omega_{\sigma_2(1)},\dots,\omega_{\sigma_k(r)}) .$$ Notice that the function $\delta_{\nu}$ is $W_{\nu}'$ invariant and $f_{\nu,\omega}$ is $W_{\nu}'$ equivariant. The function $$\delta_{\nu}\circ f_{\nu,\omega}:M_{\nu,\omega^{ss}} \to \Z$$ is therefore $W_{\nu}'$ invariant and descends to a function $\overline{M_{\nu,\omega^{ss}}} \to \Z$ which we still denote by $\delta_{\nu}\circ f_{\nu,\omega}$. Notice moreover that the quantity $d_m$ is $W'_{\nu}$ invariant too and so $d_m$ is well defined for an element $m \in \overline{M_{\nu,\omega^{ss}}}$.

The right hand side of eq.(\ref{mainteo6}) is therefore equal to \begin{equation}
\label{eq4}
\dfrac{F_{\omega}(t)P_{\nu}(t)}{w(\omega)}\left(\sum_{m \in \overline{M_{\nu,\omega^{ss}}}}\delta_{\nu}(f_{\nu,\omega}(m))\dfrac{d_m |W_{\nu}'|}{w(\nu)}\right).
\end{equation}

For any $m \in \overline{M}_{\nu,\omega^{ss}}$, we have $$\dfrac{d_m |W_{\nu}'|}{w(\nu)}=\prod_{j=1}^r\dfrac{d_j^{|X_j|}}{d_j}=\prod_{j=1}^r d_j^{|X_j|-1} $$ By eq.(\ref{cardinality}), we can thus rewrite the sum of eq.(\ref{eq4}) as \begin{equation}
\label{eq5}
\dfrac{F_{\omega}(t)P_{\nu}(t)}{w(\omega)}\left(\sum_{S \in P_{\nu,\omega^{ss}}}\delta_{\nu}(f_{\nu,\omega}(\pi_{\nu,\omega^{ss}}(S)))\right).
\end{equation}

From  Proposition \ref{inclusion}, we see that $\delta_{\nu}(f_{\nu,\omega}(\pi_{\nu,\omega^{ss}}(S)))=\mu(S,S_{\omega})$ and so, from eq.(\ref{mainteo6}), we deduce that  $\dfrac{\Theta(\nu)}{w(\nu)}$ is equal to \begin{equation}
\label{mainteo9}
\sum_{\omega \in \mathbb{T}_{\alpha}}\dfrac{F_{\omega}(t)P_{\nu}(t)}{w(\omega)}\sum_{S \in P_{\nu,\omega^{ss}}}\mu(S,S_{\omega})
\end{equation}

Summing over the $\nu \in \mathbb{T}^{ss}_{\alpha}$ of level $V$, we have therefore: \begin{equation}
\label{mainteo10}
\Coeff_{\alpha}\left(\Plexp\left(\sum_{\beta \in V}F_{\beta,gen}(t)y^{\beta}\right)\right)=\sum_{\substack{\nu \in \mathbb{T}^{ss}_{\alpha} \\ \text{of level } V }}\dfrac{\Theta(\nu)}{w(\nu)}=
\end{equation}

\begin{equation}
\label{mainteo11}
\sum_{\substack{\nu \in \mathbb{T}^{ss}_{\alpha} \\ \text{of level } V}}\sum_{\omega \in \mathbb{T}_{\alpha}}\dfrac{F_{\omega}(t)P_{\nu}(t)}{w(\omega)}\sum_{S \in P_{\nu,\omega^{ss}}}\mu(S,S_{\omega})=\sum_{\omega \in \mathbb{T}_{\alpha}}\dfrac{F_{\omega}(t)}{w(\omega)}\left(\sum_{\substack{S \leq S_{\omega}\\ \text{ of level } V}}P_{[S]}(t)\mu(S,S_{\omega})\right)
\end{equation}

The right hand side is equal to $F_{\alpha,V}(t)$  by eq.(\ref{mainteo4}).
\end{proof}

\section{Computation of convolution product for products of finite general linear groups}
\label{plethysticidentitiesconvolution}

In this section, we will use Theorem \ref{mainteo} to compute some invariants  of certain class functions of $\Gl_{\alpha}(\F_q)$. These results will be used to understand the cohomology of multiplicative quiver stacks and character stacks for Riemann surfaces in \cref{duallogmomentmapchapter} and \cref{cohocharstacks}.

We will start by reviewing some generalities about convolution product for the class functions of a finite group. The finite groups that will interest us will be \textit{finite reductive groups.} 

In particular, consider the following situation. In this section, $G$ is a reductive group over $\F_q$ with associated Frobenius $F:G \to G$.

The finite group $G^F$ is called a finite reductive group. We will recall some generalities about the representation theory of the group $G^F$ in \cref{DLU}, \cref{DLU1}. We follow the book \cite{DM}.

We fix a prime $\ell$ such that $(\ell,q)=1$ and an isomorphism $\overline{\Q_{\ell}} \cong \C$. We will identify $\overline{\Q_{\ell}}$-vector spaces with $\C$-vector spaces through this isomorphism in the following.

\subsection{Convolution product of class functions of a  finite group}

Let $H$ be a finite group. We denote by $\mathcal{C}(H)$ the set of complex class functions, i.e. the functions $f:H \to \C$ which are constant on the conjugacy classes of $H$. We denote by $1 \in \mathcal{C}(H)$ the constant function equal to $1$.

For $f,g \in \mathcal{C}(H)$, we denote by $\langle f,g \rangle$ the quantity
$$\langle f,g\rangle=\dfrac{1}{|H|}\sum_{h \in H}f(h)\overline{g(h)} .$$
 
\vspace{8 pt}

Recall that a basis of $\mathcal{C}(H)$ is given by the irreducible characters of $H$. We will denote the set of irreducible characters of $H$  by $H^{\vee}$.

The vector space $\mathcal{C}(H)$ can be endowed with the following ring structure $(\mathcal{C}(H),\ast)$. Given two class functions $f_1,f_2 \in \mathcal{C}(H)$, the convolution  $f_1 \ast f_2$ is the class function  defined as $$f_1 \ast f_2(g)=\sum_{h \in H} f_1(gh)f_2(h^{-1}) .$$

Denote by $\Cl(H)$ the set of conjugacy classes of $H$. For any  $\mathcal{O} \in \Cl(H)$, we denote by $1_{\mathcal{O}} \in \mathcal{C}(H)$ the characteristic function of $\mathcal{O}$. For a central element $\eta \in H$, we denote by $1_\eta$ the characteristic function $1_{\{\eta\}}$. Notice that for any class function $f$, there is an equality $$\langle f \ast 1_{\eta^{-1}},1_{e} \rangle=\dfrac{f(\eta)}{|H|} .$$ Recall now that
\begin{equation}
\label{identityelementgroup}
\displaystyle 1_e=\sum_{\chi \in H^{\vee}}\dfrac{\chi(e)}{|H|}\chi .\end{equation} We have therefore  

\begin{equation}
\label{conv1}
\dfrac{f(\eta)}{|H|}=\sum_{\chi \in H^{\vee}}\langle f \ast 1_{\eta^{-1}},\chi \rangle \dfrac{\chi(e)}{|H|}.
\end{equation}

For any two class functions $f_1,f_2:H \to \C$ and an irreducible character $\chi \in H^{\vee}$, we have \begin{equation}
\label{convol}
\langle f_1 \ast f_2,\chi \rangle=\langle f_1,\chi \rangle \langle f_2,\chi \rangle \dfrac{|H|}{\chi(e)} 
\end{equation}
(see for example \cite[Theorem 2.13]{isaacs}).
In particular, from eq.(\ref{conv1}) and eq.(\ref{convol}), we deduce the identity:

\begin{equation}
\label{conv2}
\dfrac{f(\eta)}{|H|}=\sum_{\chi \in H^{\vee}}\langle f,\chi \rangle \langle 1_{\eta^{-1}},\chi \rangle=\sum_{\chi \in H^{\vee}}\langle f,\chi \rangle \dfrac{\chi(\eta^{-1})}{\chi(e)}\dfrac{\chi(e)}{|H|}.
\end{equation}

\subsection{Deligne-Lusztig induction}
\label{DLU}
In this section, we review the definition of the Deligne-Lusztig induction for finite reductive groups. This construction, which allows to give a geometric description of the irreducible characters of finite reductive groups, was first introduced in \cite{DLu}. 

\vspace{6 pt}

Consider an $F$-stable Levi subgroup $L$ of $G$,  a parabolic subgroup $P$ having $L$ as Levi factor  and denote by $U_P$ the unipotent radical of $P$. Recall that there is an isomorphism $$P/U_P \cong L$$ and denote by $\pi_L:P \to L$ the associated quotient map.

In general $P$ is not $F$-stable. We can find an $F$-stable parabolic subgroup $P \supseteq L$ if and only if $\epsilon_L=\epsilon_G$.

\vspace{6 pt}

Denote by $\mathcal{L}$ the Lang map $\mathcal{L}: G \to G$ given by $\mathcal{L}(g)=g^{-1}F(g)$. The variety $X_L \coloneqq \mathcal{L}^{-1}(U_P)$ has a left $G^F$-action and a right $L^F$-action by multiplication on the left/right respectively.

This action induces an action on the compactly supported \'etale cohomology groups $H^i_c(X_L,\overline{\Q}_{\ell})$ and so endows the virtual vector space $$\displaystyle H^*_c(X_L,\overline{\Q}_{\ell})\coloneqq \bigoplus_{i \geq 0}(-1)^i H^i_c(X_L,\overline{\Q}_{\ell})$$ with the structure of a virtual $G^F$-representation-$L^F$. For an $L^F$-representation $M$, we define the Deligne-Lusztig induction $R^G_L(M)$ as the virtual $G^F$-representation given by $$R^G_L(M)=H^*_c(X_L,\overline{\Q}_{\ell}) \otimes_{\C[L^F]} M .$$ We will denote by $R^G_L$ the induced linear map $$R^G_L:\mathcal{C}(L^F) \to \mathcal{C}(G^F) .$$

\begin{oss}
In the cases that will interest us in this article, i.e. when $G$ is a product of factors of the type $(\Gl_n)_d$, it will always be true that the functor $R^G_L$ does not depend on the choice of the parabolic subgroup $P \supseteq L$ (see for example \cite[Proposition 6.1.1]{bonnafe}).

\end{oss}

\vspace{9 pt}

Consider now a split Levi subgroup $L$, i.e. such that $\epsilon_L=\epsilon_G$. In this case, we can take an $F$-stable parabolic subgroup $P \supseteq L$. The variety  $X_L$ is a $U_P$-principal bundle over the finite variety $G^F/U_P^F$ and  $H_c^{*}(X_L,\overline{\Q}_{\ell})$ is therefore concentrated in degree $2\dim(U_P)$, see the discussion before \cite[Lemma 9.1.5]{DM} for more details.

\vspace{6 pt}

If $L$ is a split Levi subgroup, we have thus an equality $$R^G_L(M)=\C[G^F/U_P^F] \otimes_{\C[L^F]} M$$ for every $L^F$-representation $M$.

In the split case, we can give the following equivalent description of this functor. For an $L^F$-representation $M$, denote by $\Infl_{L^F}^{P^F}(M)$ the natural lift to a $P^F$-representation through the quotient map $\pi_L$. In \cite[Proposition 5.18 (1)]{DM}, it is shown the following Lemma:

\begin{lemma}
If $L \subseteq G$ is a split $F$-stable Levi subgroup, we have an isomorphism of functors:
$$R^G_L \cong \Ind_{P^F}^{G^F}(\Infl_{L^F}^{P^F}) .$$
\end{lemma}

\vspace{8 pt}

Lastly, we recall the following properties of Deligne-Lusztig induction.

\begin{lemma}
\label{delignel1}
For a reductive group $G$ and a Levi subgroup $L \subseteq G$, we have:

\begin{enumerate}
    \item Given an $F$-stable Levi subgroup $L'\supseteq L$, there is an isomorphism of functors: $R^G_{L'}(R^{L'}_L) \cong R^G_L$. 
    \item Assume there exist reductive groups $G_1,G_2$ and Levi subgroups $L_1,L_2$ such that $G=G_1 \times G_2$ and $L=L_1 \times L_2$. For an $L_1^F$-representation $M_1$ and an $L_2^F$-representation $M_2$, there is a natural isomorphism $R^G_L(M_1 \boxtimes M_2)=R^{G_1}_{L_1}(M_1) \boxtimes R^{G_2}_{L_2}(M_2)$.
    
\end{enumerate}

\end{lemma}

\begin{proof}

The first point is shown in \cite[Proposition 9.1.8]{DM}. For the second point, we can choose parabolic subgroups $G_1 \supseteq P_1 \supseteq L_1$ and $G_2 \supseteq P_2 \supseteq L_2$ such that $P_1, P_2$ have as Levi factor $L_1,L_2$ respectively. The group $P=P_1 \times P_2 \subseteq G$ is thus a parabolic subgroup having $L$ as Levi factor and $U_P=U_{P_1} \times U_{P_2}$.

We have therefore $$\mathcal{L}^{-1}(U_P)=\mathcal{L}^{-1}(U_{P_1}) \times \mathcal{L}^{-1}(U_{P_1}) $$ and  $$ H^*_c(\mathcal{L}^{-1}(U_P),\overline{\Q_{\ell}}) \cong H^*_c(\mathcal{L}^{-1}(U_{P_1}),\overline{\Q_{\ell}}) \otimes H^*_c(\mathcal{L}^{-1}(U_{P_2}),\overline{\Q_{\ell}}) ,$$ as $\C[L_1^F] \otimes \C[L_2^F]$-modules, from which we deduce that $$R^G_L(M_1 \boxtimes M_2)=R^{G_1}_{L_1}(M_1) \boxtimes R^{G_2}_{L_2}(M_2) .$$

\end{proof}

\begin{oss}
\label{twist5}
Consider an $F$-stable Levi subgroup $L'\supseteq L$ and a linear character $\theta:(L')^F \to \C^*$. By restriction, we can consider it as a character $\theta:L^F \to \C^*$.

For any $f \in \mathcal{C}(L^F)$, we have an identity $R^{L'}_L(\theta f)=\theta R^{L'}_L(f)$ and therefore, by Lemma \ref{delignel1}(1), an equality $R^G_L(\theta f)=R^G_{L'}(\theta R^{L'}_L(f))$.

\end{oss}

\subsection{Unipotent characters}
\label{DLU1}

Fix now an $F$-stable maximal torus $T \subseteq G$. We follow the notations of \cite[Chapter 11]{DM}.

Denote by $W$ the Weyl group $W_G(T)$ and by $\widetilde{W}$ the semidirect group $W \rtimes \langle F \rangle$, where $\langle F \rangle$ is the group generated by the finite order automorphism induced by $F$ on $W$. 

Denote by $\mathcal{C}(WF)$ the vector space of functions $f:W \to \C$ constant on $F$-conjugacy classes. Equivalenty, a function $f \in \mathcal{C}(WF)$ can be seen as 
 a function on the coset $WF \subseteq \widetilde{W}$, invariant under $W$-conjugation. 
 
 The vector space $\mathcal{C}(WF)$ is endowed with the  Hermitian product defined as $$\displaystyle \langle f,g \rangle_{WF}=\dfrac{1}{|W|}\sum_{w \in WF} f(w) \overline{g(w)} ,$$ for $f,g \in \mathcal{C}(WF) .$

\vspace{8 pt}

\begin{esempio}
\label{twist1}
Let $G=(\Gl_n)_d$ and $T$ be the torus of Example \ref{twisted}. In this case, the morphism $\psi: \mathcal{C}(S_n) \to \mathcal{C}(WF)$ defined as $\psi(f)(\sigma_1,\dots,\sigma_d)=f(\sigma_d\cdots \sigma_1)$ is an isomorphism. It is not difficult to see that the isomorphism $\psi$ is an isometry too.
\end{esempio}

\vspace{8 pt}

In \cite[Chapter 11.6]{DM}, it is shown how to associate to each $\chi \in (W^{\vee})^F$ a function in $\mathcal{C}(WF)$, which we denote  by $\tilde{\chi}.$ This association is isometric with respect to the natural Hermitian products on both sides. The elements $\{\tilde{\chi} \}_{\chi \in (W^{\vee})^F}$ form moreover a basis of $\mathcal{C}(WF)$. 

\vspace{8 pt}

\begin{oss}
\label{rmkmult}
Consider the group $G$ and the torus $T$ of Example \ref{twist1} above. An irreducible character $\chi \in W^{\vee}$ is determined by partitions $\lambda^1,\dots,\lambda^d \in \mathcal{P}_n$ such that $\chi=\chi^{\lambda^1} \boxtimes \cdots \boxtimes \chi^{\lambda^d}$. The character $\chi $ is $F$-stable if and only if $\lambda^1=\cdots =\lambda^d=\lambda$, so that $(W^{\vee})^F$ is in bijection with $\mathcal{P}_n$.

For $\lambda \in \mathcal{P}_n$, consider the associated function $\widetilde{(\chi^{\lambda})^{\boxtimes d}} \in \mathcal{C}(WF)$.
It is possible to check that $\psi^{-1}(\widetilde{(\chi^{\lambda})^{\boxtimes d}})=\chi^{\lambda} \in \mathcal{C}(S_n)$.

\end{oss}
\vspace{8 pt}

In \cite[Chapter 11.6]{DM} for $f \in \mathcal{C}(WF)$, it is defined a class function $R_{f}:G^F \to \C$ as

\begin{equation}
\label{defr}
R_{f}\coloneqq \dfrac{1}{|W|}\sum_{w \in W}f(w) R^G_{T_w}(1).
\end{equation}

In \textit{locus cit}, it shown that the map $f \to R_f$ induces an isometry $\mathcal{C}(WF) \to \mathcal{C}(G^F)$. In particular, the elements $\{R_{\Tilde{\chi}}\}_{\chi \in (W^{\vee})^F}$ have norm $1$ and are pairwise orthogonal in $\mathcal{C}(G^F)$.

\vspace{12 pt}

Consider now an $F$-stable Levi subgroup $L \supseteq T$ and the corresponding Weyl group $W_L \coloneqq W_L(T)$ which is an $F$-stable subgroup of $W$. Define the induction map $\Ind_{W_LF}^{WF}:\mathcal{C}(W_LF) \to \mathcal{C}(WF)$ as $$\displaystyle \Ind_{W_LF}^{WF}(f)(w)=\dfrac{1}{|W_L|}\sum_{\substack{h \in W_L \\h^{-1}wF(h) \in W_L}}f(h^{-1}wF(h)) .$$

In \cite[Proposition 11.6.6]{DM} it is shown the following Lemma.

\begin{lemma}
\label{twist33}
For any $f \in \mathcal{C}(W_LF)$, we have:
\begin{equation}
\label{twist3}
R^G_L(R_f)=R_{\Ind_{W_LF}^{WF}(f)}
\end{equation}
\end{lemma}

\vspace{8 pt}

Let $G=(\Gl_{n})_d$ and $T$ be the torus considered above. In \cite[Chapter 11.7]{DM}, it is shown the following Lemma.

\begin{lemma}
\label{unipotentlemma}
For each $\chi \in (W^{\vee})^F$, the class function $R_{\tilde{\chi}}$ is an irreducible character of $(\Gl_n)_d^F$.
\end{lemma}

The irreducible characters of this form are called \textit{unipotent characters}. In particular, for every $\lambda \in \mathcal{P}_n$, there is a corresponding irreducible character $R_{\widetilde{\chi_{\lambda}^{\boxtimes d}}} \in \mathcal{C}((\Gl_n)_d^F)$, which we will denote by $R_{\lambda}$.

\vspace{6 pt}

From Lemma \ref{delignel1}(2), we deduce that Lemma \ref{unipotentlemma} is true more generally for any group $G$ of the form $(\Gl_{n_1})_{d_1} \times \cdots (\Gl_{n_r})_{d_r}$. In this case, the unipotent characters are  in bijection with the multipartitions $\bm \lambda \in \mathcal{P}_{n_1} \times \cdots \times \mathcal{P}_{n_r}$ and we denote by $R_{\bm \lambda}$ the associated irreducible unipotent character. 

\vspace{2 pt}

Consider now a group $G=(\Gl_{n_1})_{d_1} \times \cdots (\Gl_{n_r})_{d_r}$. An argument similar to that at the end of Example \ref{keyexampleLevi} shows that  all $F$-stable Levi subgroups are isomorphic to a group of the form $(\Gl_{n'_1})_{d'_1} \times \cdots (\Gl_{n'_s})_{d'_s}$. By eq.(\ref{twist3}), we deduce the following Proposition.

\begin{prop}
\label{twist4}
Let $G=(\Gl_{n_1})_{d_1} \times \cdots (\Gl_{n_r})_{d_r}$ and $L \subseteq G$ an $F$-stable Levi subgroup such that $$(L,F) \cong (\Gl_{n'_1})_{d'_1} \times \cdots (\Gl_{n'_s})_{d'_s} .$$ For any $\bm \mu \in \mathcal{P}_{n'_1} \times \cdots \times \mathcal{P}_{n'_s}$, the character $R^G_L(R_{\bm \mu})$ belongs to the vector space spanned by the unipotent characters of $G^F$.
\end{prop}
\vspace{12 pt}

Lastly, consider an $F$-stable Levi subgroup $L \subseteq G$, a class function $R_f \in \mathcal{C}(L^F)$ for $f \in \mathcal{C}(W_LF)$ and a linear character $\theta:L^F \to \C^*$. Fix a central element $\gamma \in G^F$. Notice that $\gamma \in L^F$ too. The Mackey's formula \cite[Proposition 10.1.2]{DM} implies the following Proposition.

\begin{prop}
\label{valueatcentralelement}
We have an equality:
\begin{equation}
\label{valuecentralelement}
R^G_L(\theta R_f)(\gamma)=R^G_L(\theta R_f)(e)\theta(\gamma).
\end{equation}
\end{prop}

\subsection{Characters of tori and graphs}
\label{charLevi}

Consider the group $\Gl_{\alpha}$ and fix an $F$-stable maximal torus $T \subseteq \Gl_{\alpha}$. We follow the notations of  \cref{graph}.

Recall that we have the dual root system $\Phi^{\vee} \subseteq Y_*(T)$ which is endowed with a canonical bijection $\Phi \leftrightarrow \Phi^{\vee}$ and that, for each $h\neq j \in \{1,\dots,m\}$, we denote by $\epsilon^{\vee}_{h,j}$ the element associated to $\epsilon_{h,j}$
through this bijection.

\vspace{8 pt}

Consider now a character $\theta:T^F \to \C^*$. In this paragraph, we  show how to associate an admissible graph $\Gamma_{\theta}$ with vertices $\mathcal{B}$ to the character $\theta$.

 In \cite[Proposition 11.7.1]{DM}, it is shown that there exists a canonical short exact sequence
\begin{center}
\begin{tikzcd}
1 \arrow[r,""] &Y_*(T) \arrow[r,""] &Y_*(T) \arrow[r,""] &T^F \arrow[r,""] &1.
\end{tikzcd}
\end{center}

In particular, the character $\theta:T^F \to \C^*$ induces by restriction a morphism $\widetilde{\theta}:Y_*(T) \to \C^*$.
\vspace{8 pt}

The graph $\Gamma_{\theta}$ is defined as follows. The set of the vertices of $\Gamma_{\theta}$ is $\mathcal{B}=\{\epsilon_1,\dots,\epsilon_m\}$ and, for each $h > j$, there is an edge between $\epsilon_h$ and $\epsilon_j$ if and only if $$\epsilon_{h,j}^{\vee} \in \Ker(\widetilde{\theta}) .$$

From Remark \ref{admissiblecondition}, we deduce the following Lemma.

\begin{lemma}
For any $\theta:T^F \to \C^*$, the graph $\Gamma_{\theta}$ is admissible.
\end{lemma}

In particular, from Lemma \ref{lemmaadmissible}, there exists a unique  admissible subtorus $S_{\theta} \subseteq \Gl_{\alpha}$ such that $\Gamma_{S_{\theta}}=\Gamma_{\theta}$. 

We will denote by $L_{\theta}=C_{\Gl_{|\alpha|}}(S_{\theta})$ and  by $\widetilde{L_{\theta}}=L_{\theta} \cap\Gl_{\alpha}$. The Levi subgroup $L_{\theta}$ is the \textit{connected centralizer} of $\theta$ in $\Gl_{|\alpha|}$, defined in \cite[Definition 5.19]{DLu} and $\widetilde{L_{\theta}}$ is the connected centralizer of $\theta$ in $\Gl_{\alpha}$.

\vspace{6 pt}

\begin{esempio}
Consider the subset $I=\{1,2,3,4\}$, the dimension vector $\alpha=(2,1,1,1)$ and the torus $T$ of diagonal matrices of Example \ref{exampletoriputlate}. Notice that $$T^F=\Biggl\{\left(\begin{pmatrix} \lambda &0 \\ 0 &\mu  \end{pmatrix}, \gamma ,\delta,\eta \right) \ | \ \lambda,\mu, \gamma ,\delta,\eta  \in \ \F_q^* \Biggr\} .$$ 

Consider $(\beta_1,\beta_2) \in \Hom(\F_q^*,\C^*)^2$ and let $$\theta_{\beta_1,\beta_2}:T^F \to \C^*$$ defined as $$\theta_{\beta_1,\beta_2}(\lambda,\mu, \gamma ,\delta,\eta)=\beta_1(\lambda \gamma \delta\eta) \beta_2(\mu) .$$

If $\beta_1 \neq \beta_2$, the graph $\Gamma_{\theta_{\beta_1,\beta_2}}$ is \begin{center}
\begin{tikzcd}
\epsilon_1 \arrow[dash]{r} \arrow[bend left=30, dash]{rr} 
\arrow[bend left=40, dash]{rrr}
&\epsilon_3 \arrow[dash]{r} \arrow[bend right=30, dash]{rr}  &\epsilon_4 \arrow[dash]{r} &\epsilon_5\\
\epsilon_2
\end{tikzcd}
\end{center}

and the admissible torus $S_{\theta_{\beta_1,\beta_2}}$ is therefore the torus $S_1$ of Example \ref{exampleadmissible2}.

If $\beta_1=\beta_2$, the graph $\Gamma_{\theta_{\beta_1,\beta_1}}$ is \begin{center}
\begin{tikzcd}
\epsilon_1 \arrow[dash]{r} \arrow[bend left=30, dash]{rr} 
\arrow[bend left=40, dash]{rrr}
&\epsilon_3 \arrow[dash]{r} \arrow[bend right=30, dash]{rr}  &\epsilon_4 \arrow[dash]{r} &\epsilon_5\\
\epsilon_2 \arrow[dash]{u} \arrow[dash]{ur} \arrow[dash]{urr} \arrow[dash]{urrr}
\end{tikzcd}
\end{center}
i.e. the complete graph with $5$ vertices and  the admissible torus $S_{\theta_{\beta_1,\beta_1}}$ is thus $Z_{\alpha}$.

\end{esempio}

From \cite[Proposition 5.11]{DLu}, we deduce the following Proposition.

\begin{prop}
\label{propositionnotclear}
For any $\theta:T^F \to \C^*$, the character $\theta$ can be extended to a character $\theta:L_{\theta}^F \to \C^*$.
\end{prop}

\subsection{Characters of tori and Levi subgroups of the finite general linear groups}

\subsubsection{Characters of Levi subgroups}

Consider now $G=\Gl_{\alpha}$, an admissible torus $S$, the associated Levi subgroups $L_S=C_{\Gl_{|\alpha|}}(S) \subseteq \Gl_{|\alpha|}$ and $\widetilde{L_S}=L_S \cap \Gl_{\alpha}$. Let $T$ be an $F$-stable maximal torus $T \subseteq \Gl_{\alpha}$ such that $\widetilde{L_S} \supseteq T \supseteq S$. 

We use the notations introduced in \cref{graphdefinitions}, \cref{charLevi} for the graphs associated to $T, S , \theta$ in this situation.

Consider now a character $\theta:L_S^F \to \C^*$. By restriction, we obtain a character $\theta:T^F \to \C^*$, from which we define an associated admissible torus $S_{\theta}$ and the corresponding Levi subgroups $L_{\theta}, \widetilde{L_{\theta}}$, as in \cref{charLevi}.

In general, we have that $S_{\theta} \leq S$, i.e. $L_{\theta} \supseteq L_S$ (or equivalently $\Gamma_{\theta} \leq \Gamma_S$). In this case, $\Gamma_{\theta}=\Gamma_S$ if and only if $\widetilde{L_{\theta}}=\widetilde{L_S}$.

\vspace{2 pt}

As recalled in Proposition \ref{propositionnotclear}, the character $\theta$ can be extended to the connected centralizer $\theta:L_{\theta}^F \to \C^*$.  Since $L_{\theta} \supseteq \widetilde{L_{\theta}}$, in particular, $\theta$ can be extended to a character $\theta: \widetilde{L_{\theta}}\to \C^*$.

\vspace{4 pt}

Conversely, for each character $\theta:T^F \to \C^*$ such that $L_{\theta} \supseteq L_S$ (i.e. $\Gamma_{\theta} \leq \Gamma_S)$, the character $\theta$ can be first extended to $\theta:L_{\theta}^F \to \C^*$ and then restricted to obtain a linear character $\theta:L_S^F \to \C^*$ and $\theta:\widetilde{L_S}^F \to \C^*$.

We obtain therefore the following correspondence.

\begin{prop}
\label{correspondenceadmissible}
There are bijections:
\begin{equation}
\Hom(L_S^F,\C^*) \leftrightarrow \ \{\theta \in \Hom(T^F,\C^*) \ | \ \Gamma_{\theta} \leq \Gamma_S \} \leftrightarrow \{\theta \in \Hom(\widetilde{L_S}^F,\C^*)\ | \ \Gamma_{\theta} \leq \Gamma_S\}.
\end{equation}
\end{prop}

\vspace{4 pt}

Lastly, we give the following definition of a \textit{reduced} character. 

\begin{definizione}
Given an $F$-stable Levi subgroup $L \subseteq \Gl_{\alpha}$ and a character $\theta:L^F \to \C^*$, we say that $\theta$ is reduced if there exists an admissible $F$-stable subtorus $S \subseteq \Gl_{\alpha}$ and an $F$-stable subtorus $T \supseteq S$ such that

\begin{itemize}
\item $L=\widetilde{L_S}$
\item For the connected centralizer $L_{\theta}$ defined from $T$, we have $\widetilde{L_{\theta}}=\widetilde{L_S}$.
\end{itemize}

\end{definizione}

\subsubsection{Reduced characters and connected centralizers}
\label{connectedcentralizer1}

Let now $G=\Gl_n$ (i.e. $|I|=1$) and consider an $F$-stable Levi subgroup $L \subseteq \Gl_n$ and a linear character $\theta:L^F \to \C^*$. The Levi subgroup $L$ contains an $F$-stable maximal torus $T$.

From $\theta,T$, we determine the connected centralizer  $L_{\theta} \supseteq L$, as defined above. In this case, the character $\theta$ is reduced if $L_{\theta}=L$.

\vspace{6 pt}

\begin{oss}
While the connected centralizer $L_{\theta}$ does depend on the choice of the torus $T$, from \cite[Proposition 5.11(ii), Proposition 5.20]{DLu} we deduce that for any two $F$-stable maximal tori $T,T'\subseteq L$ and the corresponding connected centralizers $L_{\theta},L'_{\theta}$, there exists an element $g \in \Gl_n(\F_q)$ such that $gL_{\theta}g^{-1}=L_{\theta}'$.

In particular, the property of being reduced does not depend on the choice of the maximal torus $T$.
\end{oss}

We will now describe the connected centralizers $L_{\theta}$ for  certain Levi subgroups $L \subseteq \Gl_n$ and certain linear characters $\theta:L^F \to \C^*$. This description is going to be useful both for recalling the construction of irreducible characters of $\Gl_n(\F_q)$ in \cref{construcirredchar} and for the proof of Lemma \ref{firstep}, which is the key technical point to prove our main result about multiplicative quiver stacks over $\F_q$, Theorem \ref{charstack}. 

\vspace{6 pt}

For any two positive integers $r,d$ such that $r |d$, the norm map $N_{\mathbb{F}^*_{q^d}/\mathbb{F}^*_{q^r}}: \mathbb{F}^*_{q^d} \to \mathbb{F}_{q^r}^*$ induces by precomposition an injective homomorphism $$\Gamma_{r,d}\coloneqq \Hom(\mathbb{F}_{q^r}^*,\C^*) \to \Hom(\F^*_{q^d},\C^*) .$$ We denote by $\Gamma$ the inductive limit via these maps  $$\Gamma \coloneqq \varinjlim \Hom(\mathbb{F}_{q^d}^*,\C^*).$$ Notice that, for any $n \geq 1$, we can view $\Hom(\F^*_{q^n},\C^*) $ as a subgroup of $\Gamma$ through the universal maps of the limit. The Frobenius morphism acts by precomposition on each term $\Hom(\F_{q^d}^*,\C^*)$ (i.e. $F(\gamma)=\gamma \circ F$) and so defines a morphism $F:\Gamma \to \Gamma$.

Consider the Levi subgroup $$L=(\Gl_{n_{1}})_{d_1} \times  \cdots \times (\Gl_{n_{r}})_{d_r} $$ with $n_{1},\dots,n_{r},$ $d_1,\dots,d_r$ positive integers such that $d_1n_{1}+\cdots +d_rn_{r}=n$ and let $T$ be the maximal torus $$(T_{n_1})_{d_1} \times \cdots (T_{n_r})_{d_r} .$$

The group $L^F$ is isomorphic to $\Gl_{n_1}(\F_{q^{d_1}}) \times \cdots \times \Gl_{n_r}(\F_{q^{d_r}})$. A character $\theta:L^F \to \C^*$ corresponds thus to an element $(\theta_1,\dots,\theta_r) \in \Hom(\F_{q^{d_1}}^*,\C^*) \times \cdots \times  \Hom(\F_{q^{d_r}}^*,\C^*)$ such that $$\displaystyle \theta(M_1,\dots,M_r)=\prod_{j=1}^r \theta_j(\det(M_j))$$ with $M_j \in \Gl_{n_j}(\F_{q^{d_j}})$. We have the following Lemma

\begin{lemma}
\label{connectedcentralizer11}
The character $\theta$ is \textit{reduced} if and only if the $F$-orbits of $\theta_1,\dots,\theta_r$ inside $\Gamma$ have length $d_1,\dots,d_r$ respectively and are pairwise disjoint.
\end{lemma}

\begin{proof}
Notice that, for any $h\in\{1,\dots,n\}$ there exist unique $i_h\in\{1,\dots,r\}$ and $j_h\in\{1,\dots,d_{i_h}\}$ such that $$\sum_{s=1}^{i_h-1} d_s n_s < h \leq \sum_{s=1}^{i_h} d_s n_s $$ and $$\sum_{s=1}^{i_h-1} d_s n_s+ n_{i_h}(j_h-1) < h \leq \sum_{s=1}^{i_h-1} d_s n_s+ n_{i_h}j_h .$$

For $h_1,h_2 \in \{1,\dots,n\}$, we have that $$\widetilde{\theta}(\epsilon_{h_1,h_2}^{\vee})=1$$ if and only if $$ \theta_{i_{h_1}}^{q^{j_{h_1}}}=\theta_{i_{h_2}}^{q^{j_{h_2}}}$$ as elements of $\Gamma$, from which we deduce the Lemma above.

\end{proof}

\vspace{6 pt}

Consider now two $F$-stable Levi subgroups $L \subseteq \Gl_n$ and $L' \subseteq \Gl_{n'}$ and the Levi subgroup $M=L \times L' \subseteq \Gl_m$ embedded block diagonally, where $m=n+n'$. 

Assume that $L=(\Gl_{n_1})_{d_1} \times \cdots \times (\Gl_{n_r})_{d_r}$ and $L'=(\Gl_{n'_1})_{d_1'} \times \cdots \times (\Gl_{n'_s})_{d'_s}$ and consider two reduced characters $\theta:L^F \to \C^*$ and $\theta':(L')^F \to \C^*$ corresponding to $(\theta_1,\dots,\theta_r),(\theta_1',\dots,\theta'_s)$ where $\theta_i \in \Hom(\F_{q^{d_i}}^*,\C^*),\theta'_j \in \Hom(\F_{q^{d'_j}}^*,\C^*)$ for $i=1,\dots,r$, $j=1,\dots,s$.

Consider the character $$\gamma=\theta \times \theta':M^F \to \C^* .$$ Its connected centralizer $M_{\gamma}$ admits the following description. For $i\in \{1,\dots,r\}$, consider the subset $J_i \subseteq \{1,\dots,s\}$ defined as $$J_i \coloneqq \bigl\{j \in \{1,\dots, s\}\ | \ d'_j=d_i \text{ and the } F \text{-orbits 
 inside } \Gamma \text{ of } \theta_i,\theta'_j \text{ have nonempty intersection} \bigr\} .$$

We have that either $J_i=\emptyset$ or $J_i=\{j_i\}$ for an element $j_i \in \{1,\dots,s\}$, since the characters $\theta,\theta'$ are both reduced. Denote by $I'\subseteq \{1,\dots,r\}$ the subset $I'\coloneqq\{i \ | \ J_i=\emptyset \}$ and by $J' \subseteq \{1,\dots,s\}$ the subset $\displaystyle J'=\{1,\dots,s\}\setminus\bigsqcup_{i=1}^r J_i$.

\vspace{6 pt} 

 In a similar way to the one used to prove Lemma \ref{connectedcentralizer11}, we see that any connected centralizer $M_{\gamma}$ is $\Gl_n(\F_q)$-conjugated to the Levi subgroup $$M'_{\gamma}=\displaystyle \prod_{i \in I'}(\Gl_{n_i})_{d_i} \prod_{j \in J'} (\Gl_{n'_j})_{d'_j} \prod_{i \in (I')^c} (\Gl_{n_i+n'_{j_i}})_{d_i} .$$ Through this conjugation, the character $\gamma$ corresponds to the character $\gamma:(M'_{\gamma})^F \to \C^*$ associated to $((\theta_i)_{i \in I'},(\theta'_j)_{j \in J'},(\theta_i)_{i \in (I')^c})$.

\subsection{Irreducible characters of finite general linear group}\label{construcirredchar}
In this paragraph, we quickly recall how to build the character table of the groups $\Gl_{\alpha}(\F_q)$. We start from the following Lemma, which will also be needed later. Its proof is a consequence of \cite[Lemma 11.4.3,Lemma 11.4.4]{DM} and \cite[Theorem 8]{Lusztig}.

\begin{lemma}
\label{lemmamult}
Consider $G=\Gl_{\alpha}$, an $F$-stable Levi subgroup $L \subseteq G$ 
and two characters $R_{\phi_1},R_{\phi_2} \in \mathcal{C}(L^F)$ with $\phi_1,\phi_2 \in \mathcal{C}(W_LF)$. Let $\theta:L^F \to \C^*$ be a reduced character. 

We have that:
$$\langle R^G_L(\theta R_{\phi_1}),R^G_{L}(\theta R_{\phi_2}) \rangle_{G^F}=\langle R_{\phi_1},R_{\phi_2} \rangle_{L^F} .$$

\end{lemma}

Notice in particular that if $\phi_1=\phi_2=\tilde{\psi}$ with $\psi \in (W_L^{\vee})^F$, from Lemma \ref{unipotentlemma}, we have that $$\langle R^G_L(\theta R_{\Tilde{\psi}}), R^G_L(\theta R_{\Tilde{\psi}})\rangle=\langle R_{\Tilde{\psi}},R_{\Tilde{\psi}}\rangle=1 .$$ In particular, the character $R^G_L(\theta R_{\Tilde{\psi}})$ is a virtual irreducible character, i.e. an irreducible character up to a sign. 

From these remarks, in \cite[Theorem 3.2]{LSr}, it is shown the following Theorem.

\begin{teorema}
\label{theoremirreduciblecharacters}
For an irreducible character $\chi \in \Gl_{\alpha}(\F_q)^{\vee}$, we have $$\chi=\epsilon_G \epsilon_L R^G_L(\theta R_{\Tilde{\phi}}) ,$$ where $L$ is an $F$-stable Levi subgroup, $\phi \in (W^{\vee}_L)^F$ and $\theta:L^F \to \C^*$ is a reduced character.

Two characters $\chi_1,\chi_2$ with associated data $(L_1,\theta_1,\phi_1)$ and $(L_2,\theta_2,\phi_2)$ are equal if and only if the triples $(L_1,\theta_1,\phi_1),(L_2,\theta_2,\phi_2)$ are $\Gl_{\alpha}(\F_q)$-conjugated.
\end{teorema}

For an irreducible character $\chi$ with associated datum $(L,\theta,\phi)$, we will call the couple $(L,\theta)$ the \textit{semisimple part} of $\chi$. This is well defined up to $\Gl_{\alpha}(\F_q)$-conjugacy.

\vspace{10 pt}

\begin{oss}
\label{irreducibleinduced1}

Let $G=\Gl_{n}$ and consider now an $F$-stable Levi subgroup $L$, a character $\gamma:L^F \to \C^*$ (not necessarily reduced) and a unipotent irreducible character $R_{\tilde{\psi}}$ for $\psi \in (W^{\vee}_L)^F$.

Let $L_{\gamma}$ be a connected centralizer of $\gamma$. By Remark \ref{twist5}, we
have an equality \begin{equation}R^G_L(\gamma R_{\tilde{\psi}})=R^G_{L_{\gamma}}(\gamma R^{L_{\gamma}}_{L}(R_{\Tilde{\psi}})) .\end{equation}

\end{oss}

Notice that from Proposition \ref{twist4}, we have that  $R^{L_{\gamma}}_L(R_{\Tilde{\psi}})$ belongs to the vector space spanned by
the unipotent characters of $L_{\gamma}$. We deduce thus the following Proposition.

\begin{prop}
\label{irreducibleinduced}
For $G=\Gl_n$, any $F$-stable Levi subgroup $L \subseteq G$, any $\gamma:L^F \to \C^*$ and any $\psi \in (W^{\vee}_L)^F$, the character $R^G_L(\gamma R_{\Tilde{\psi}})$ belongs to the vector space spanned by the
irreducible characters of $\Gl_n(\F_q)$ with semisimple part $(L_{\gamma},\gamma)$.
\end{prop}
\vspace{6 pt}

From Proposition \ref{valueatcentralelement}, we deduce in particular the following Proposition:

\begin{prop}
\label{valueatacentralelement1}
Given $\chi \in \Gl_{\alpha}^{\vee}(\F_q)$ with $\chi=\epsilon_G \epsilon_L R^G_L(\theta R_{\Tilde{\psi}})$ and a central element $\eta=(\eta_iI_{\alpha_i})_{i \in I}$, we have \begin{equation}
\dfrac{\chi(\eta)}{\chi(e)}=\theta(\eta).
\end{equation}
\end{prop}

\subsection{Type of an irreducible character} 
\label{typesirreduciblechar}
Let $\chi \in \Gl_n(\F_q)^{\vee}$ with associated datum $(L,\theta,\phi)$. Up to $\Gl_n(\F_q)$-conjugacy, $L$ is equal to $(\Gl_{n_1})_{d_1} \times \cdots \times (\Gl_{n_r})_{d_r}$ and $\phi=R_{\bm \lambda}$ for $\bm \lambda=(\lambda_1,\dots,\lambda_r) \in \mathcal{P}_{n_1} \times \cdots \times \mathcal{P}_{n_r}$ a multipartition. Up to reordering, the $(d_i,\lambda_i)$s define a  type $$\omega_{\chi}\coloneqq(d_1,\lambda_1)\dots (d_r,\lambda_r) \in \mathbb{T}_n$$ which is called the type of the irreducible character $\chi$.

\vspace{8 pt}

\begin{esempio}
Consider a partition $\mu=(\mu_1,\dots,\mu_h) \in \mathcal{P}_n$ and the associated split Levi subgroup $L_{\mu}=\Gl_{\mu_1} \times  \Gl_{\mu_2}\times \cdots 
 \times \Gl_{\mu_h} \subseteq \Gl_n$. For each reduced character $\theta:L_{\mu}^F \to \C^*$, the type of the character $R^G_{L_{\mu}}(\theta)$ is $(1,(\mu_1))\cdots (1,(\mu_h))$.

\end{esempio}

\vspace{10 pt}

In a similar way, for any finite set $I$ and any $\alpha \in \N^I$, to each irreducible character $\chi \in \Gl_{\alpha}(\F_q)^{\vee}$, we can associate a multitype $\omega_{\chi} \in \mathbb{T}_{\alpha}$ in the following way.

Let $\chi=\epsilon_{\Gl_{\alpha}} \epsilon_L R^{\Gl_{\alpha}}_L( R_{\Tilde{\phi}}\theta)$, where $\theta:L^F \to \C^*$ is a reduced character and $R_{\Tilde{\phi}}$ is a unipotent character of $L^F$ with $\phi \in (W^{\vee}_L)^F$.

\vspace{6 pt}

Consider an $F$-stable torus $T \subseteq L$ and the restricted character $\theta:T^F \to \C^*$. As explained in \cref{charLevi}, this determines a Levi subgroup $L_{\theta} \subseteq \Gl_{|\alpha|}$ with admissible center $S_{\theta} \subseteq T$ 
 and such that $L_{\theta} \cap \Gl_{\alpha}=L$, since $\theta$ is reduced.

Consider the semisimple multitype $[S_{\theta}]=(d_1,(1^{\beta_1}))\dots (d_r,(1^{\beta_r}))$. As in the case of $\Gl_n(\F_q)$, the character $\phi$ determines multipartition $\bm \lambda_1 ,\dots ,\bm \lambda_r \in \mathcal{P}^I$ such that $|\bm \lambda_j|=\beta_j$. Up to reordering, the $(d_i,\bm \lambda_i)$s define a multitype  $$\omega_{\chi} \coloneqq (d_1,\bm \lambda_1)\dots (d_r,\bm \lambda_r) .$$

For $\omega \in \mathbb{T}_I$ and $\chi \in \Gl_{\alpha}(\F_q)^{\vee}$, we will use the notation $\chi \sim \omega$ if $\omega_{\chi}=\omega$.

\begin{esempio}
Let $I=\{1,2,3,4\}$ and $\alpha=(2,1,1,1)$. Let $T \subseteq \Gl_2$ be the $F$-stable torus of diagonal matrices,  consider $ \beta \neq \gamma  \in \Hom(\F_q^*,\C^*)$ and the associated character $(\beta,\gamma):T^F \to \C^*$. Let $\chi$ be the character $\chi \in \Gl_{\alpha}(\F_q)^{\vee}$ $$ \chi=R^G_T((\beta,\gamma)) \boxtimes \gamma\circ \det \boxtimes \gamma \circ\det \boxtimes \gamma\circ\det .$$

Let $\beta_1=(1,1,1,1)$ and $\beta_2=(1,0,0,0)$. The associated multitype is $$ \omega_{\chi}=(1,(\beta_1))(1,(\beta_2)) .$$

\end{esempio}

\vspace{8 pt}

\begin{oss}
Given $\omega \in \mathbb{T}_{\alpha}$, consider  an irreducible character $\chi \in \Gl_{\alpha}(\F_q)^{\vee}$ of type $\omega$. Fix $S \in \mathcal{Z}_{\alpha}$ such that $[S]=\omega^{ss}$ (for instance $S=S_{\omega}$). We can assume then that $$\chi=\epsilon_{\widetilde{L_S}} \epsilon_{\Gl_{\alpha}}R^{\Gl_{\alpha}}_{\widetilde{L_S}}(\theta R_{\widetilde{\phi}})$$ with $\theta:\widetilde{L_S}^F \to \C^*$ such that $S_{\theta}=S$ and $R_{\widetilde{\phi}}$ a unipotent character of $\widetilde{L_S}^F$. Notice that for any $\gamma:\widetilde{L_S}^F \to \C^*$ such that $S_{\gamma}=S$, the character $ \epsilon_{\widetilde{L_S}} \epsilon_{\Gl_{\alpha}}R^{\Gl_{\alpha}}_{\widetilde{L_S}}(\gamma R_{\widetilde{\phi}})$ is irreducible and of type $\omega$.

The map from $\{\gamma:\widetilde{L_S}^F \to \C^* \ | \ S_{\gamma}=S\}$ to $\{\chi \in \Gl_{\alpha}(\F_q)^{\vee} \text{ of multitype } \omega\}$ which sends $\gamma$ to $\epsilon_{\widetilde{L_S}} \epsilon_{\Gl_{\alpha}}R^{\Gl_{\alpha}}_{\widetilde{L_S}}(\gamma R_{\widetilde{\phi}})$ is surjective and its fibers have cardinality $w(\omega)$, see for example \cite[Proof of Theorem 2.2]{letellierDT} for the analogous statement for conjugacy classes of $\Gl_{\alpha}(\F_q)$.

\end{oss}

\vspace{10 pt}

Recall that the value $\dfrac{\chi(e)}{|\Gl_{\alpha}(\F_q)|}$ for $\chi \in \Gl_{\alpha}(\F_q)^{\vee}$ depends only on the multitype of $\chi.$  More precisely, for a partition $\lambda \in \mathcal{P}$, let $H_{\lambda}(t)$ be the \textit{hook polynomial} $$\displaystyle H_{\lambda}(t) \coloneqq \prod_{s \in \lambda}(1-t^{h(s)}) .$$ For a multipartition $\bm \lambda=(\lambda^i)_{i \in I} \in \mathcal{P}^I$, we define $\displaystyle H_{\bm \lambda}(t)\coloneqq\prod_{i \in I}H_{\lambda^i}(t)$. Given a multitype $\omega=(d_1,\bm \lambda_1) \dots (d_r,\bm \lambda_r)$, define $H_{\omega}^{\vee}(t)$ as
\begin{equation}
\label{duallog5}
H^{\vee}_{\omega}(t)\coloneqq \dfrac{(-1)^{f(\omega)}}{q^{\left(\sum_{i \in I}\frac{\alpha_i(\alpha_i-1)}{2}-n(\omega)\right)}\prod_{j=1}^r H_{\bm \lambda_j}(t^{d_j}) },
\end{equation}

where if $|\bm \lambda_1|=\beta_1,\dots, |\bm \lambda_r|=\beta_r$, then $$f(\omega) \coloneqq \displaystyle\sum_{j=1}^r |\beta_j| $$ and $$n(\omega)\coloneqq\sum_{j=1}^r d_j n(\bm \lambda_j) .$$ We have the following Proposition (see \cite[IV, 6.7]{macdonald}).
\begin{prop}
For any $\chi \in \Gl_{\alpha}(\F_q)^{\vee}$, we have:
\begin{equation}
\label{valueirreduciblecharacter}
\dfrac{\chi(e)}{|\Gl_{\alpha}(\F_q)|}=H_{\omega_{\chi}}^{\vee}(q).
\end{equation}
\end{prop}

\subsection{Multiplicative parameters}
\label{zioperonemattinalevanto}
Given an element $\eta=(\eta_i)_{i \in I} \in (\F_q^*)^I$ and $\delta \in \N^I$, we define $$\eta^{\delta}\coloneqq \prod_{i \in I}\eta_i^{\delta_i} .$$ We  denote by $\mathcal{H}_{\eta}$ the subset of $\N^I$ defined as $$\mathcal{H}_{\eta}\coloneqq\{\delta \in \N^I \ | \ \eta^{\delta}=1\}$$ and, for any $\alpha \in \N^I$, we denote by $\mathcal{H}_{\eta,\alpha}$  the intersection $\mathcal{H}_{\eta,\alpha}\coloneqq\mathcal{H}_{\eta} \cap \N^I_{\leq \alpha}$. For an admissible torus $S \in \mathcal{Z}_{\alpha}$ , we say that $S$ is of level $\eta$ if it is of level $\mathcal{H}_{\eta,\alpha}$. 

\vspace{6 pt}

Fix now $\alpha \in \N^I$. We still denote by $\eta$ the central element $\eta\coloneqq (\eta_i I_{\alpha_i} ) \in \Gl_{\alpha}(\F_q)$. Assume now to have fixed, for each $S \in \mathcal{Z}_{\alpha}$, an $F$-stable maximal torus $\Gl_{\alpha} \supseteq T_S \supseteq S$ in such a way that if $S \leq S'$ then $T_{S}=T_{S'}$. Define then the functions $g_{\eta},f_{\eta}:\mathcal{Z}_{\alpha} \to \C$ as $$g_{\eta}(S)\coloneqq\sum_{\substack{\theta :T_S^F \to \C^* \\ \Gamma_{\theta}=\Gamma_S}}\theta(\eta) $$ and $$f_{\eta}(S) \coloneqq \sum_{\substack{\theta :T_S^F \to \C^* \\ \Gamma_{\theta} \leq \Gamma_S}}\theta(\eta) .$$

\vspace{2 pt}

The understanding of the functions $f_{\eta},g_{\eta}$ is a key part of the proof of Theorem \ref{mainteochar}, which is going to be the main technical result needed to prove Theorem \ref{Epolynomialtheorem}.

\vspace{4 pt}

By Identity (\ref{mobius}), we have
\begin{equation}
\label{dualmobius}
g_{\eta}(S)=\sum_{S'\leq S}\mu(S',S)f_{\eta}(S')
\end{equation}

Notice that by the bijection of eq.(\ref{correspondenceadmissible}),  for each $S \in \mathcal{Z}_{\alpha}$, we have $$f_{\eta}(S)=\sum_{\theta : L_S^F \to \C^* }\theta(\eta) .$$ Fix now $S$ with associated semisimple multitype $[S]=(d_1,\beta_1)\dots(d_r,\beta_r)$, where $\beta_1,\dots,\beta_r \in \N^I$. Notice that there exists $h \in \Gl_{\alpha}(\F_q)$ such that $\displaystyle hSh^{-1}=\prod_{j=1}^r(Z_{\beta_j})_{d_j}$ and so $$\displaystyle hL_Sh^{-1}=\prod_{j=1}^r (\Gl_{|\beta_j|})_{d_j} .$$ In particular, a character $\theta:L_S^F \to \C^*$, thorugh the conjugation by $h$, corresponds to  an element $(\theta_1,\dots,\theta_r) \in \Hom(\F^*_{q^{d_1}},\C^*) \times \cdots \times  \Hom(\F_{q^{d_r}}^*,\C^*)$ such that $$\displaystyle \theta(M_1,\dots,M_r)=\prod_{j=1}^r \theta_j(\det(M_j))$$ with $M_j \in \Gl_{|\beta_j|}(\F_{q^{d_j}})$.

As the element $\eta \in \Gl_{\alpha}(\F_q)$ is central, we have the following equality \begin{equation}
\label{charactersLevi1}
\theta(\eta)=\prod_{j=1}^r \theta_j(\eta^{\beta_j}).
\end{equation}

In particular, eq.(\ref{charactersLevi1}) implies that $f_{\eta}(S) \neq 0$ if and only if $\eta^{\beta_j}=1$ for each $j=1,\dots,r$, i.e. if and only if $S$ is of level $\eta$. From eq.(\ref{dualmobius}), we therefore deduce that, for each $S \in \mathcal{Z}_{\alpha}$, we have

\begin{equation}
\label{sangueelacrime}
g_{\eta}(S)=\sum_{\substack{S'\leq S \\ \text{ of level } \eta}}|\Hom(L_{S'}^F,\C^*)|\mu(S',S)=\sum_{\substack{S'\leq S \\ \text{ of level } \eta}}P_{[S']}(q)\mu(S',S).
\end{equation}

\subsection{Dual Log compatiblity}
\label{duallogcompdefin}
Consider a family of class functions $\{c_{\alpha} \in \mathcal{C}(\Gl_{\alpha}(\F_q))\}$.

\begin{definizione}
The family $\{c_{\alpha}\}_{\alpha \in \N^I}$ is said to be \textit{dual Log compatible} if the product $\langle c_{\alpha},\chi \rangle$ depends only on the multitype of $\chi$ and the value of $\langle c_{\alpha},\chi \rangle$ for $\chi  \sim \omega$ is of the form $C_{\omega}(q)$ where $\{C_{\omega}(t) \in \Q(t)\}_{\omega \in \mathbb{T}_I} $ is a family of rational functions such that  for any $d_1,\dots,d_r \in \N$ and $\omega_1 \in \mathbb{T}_{\beta_1},\dots,\omega_r \in \mathbb{T}_{\beta_r}$ such that $\psi_{d_1}(\omega_1) \ast \cdots \ast \psi_{d_r}(\omega_r)=\omega$, it holds \begin{equation}\label{duallog2}C_{\omega_1}(t^{d_1})\cdots C_{\omega_r}(t^{d_r})\prod_{j=1}^r H^{\vee}_{\omega_j}(t^{d_j})=C_{\omega}(t)H_{\omega}^{\vee}(t) .\end{equation} 
i.e. the family $\{C_{\omega}(t)H^{\vee}_{\omega}(t)\}_{\omega \in \mathbb{T}_I}$ is Log compatible.
\end{definizione}

\vspace{8 pt}

For each $\omega \in \mathbb{T}$, denote by $\displaystyle\widetilde{C}_{\omega}(t)\coloneqq C_{\omega}(t)H^{\vee}_{\omega}(t)$. The family $\{\widetilde{C}_{\omega}\}_{\omega \in \mathbb{T}_I}$ is therefore Log compatible. For each $\alpha \in \N^I$ and for each $\eta \in (\F_q^*)^I$, we will denote by $\widetilde{C}_{\alpha,\eta}(t)$ the polynomial $\widetilde{C}_{\alpha,\mathcal{H}_{\eta,\alpha}}(t)$ introduced in eq.(\ref{polynomialV}).

\vspace{6 pt}

\begin{esempio}
\label{conv3}
Let $I=\{\cdot\}$ and, for any $n \in \N$, let $f_n:\Gl_n(\F_q) \to \C$ be the class function $$f_n(h)\coloneqq\# \{(x,y) \in \Gl_n(\F_q) \times \Gl_n(\F_q) \ | \ [x,y]=h\} $$

for $h \in \Gl_n(\F_q)$. For any $\chi \in \Gl_n(\F_q)^{\vee}$ of type $\omega$, it holds $\langle f_n,\chi \rangle=\dfrac{1}{H^{\vee}_{\omega}(q)}$. More generally for any finite group $G$ and any irreducible character  $\chi \in G^{\vee}$ we have: \begin{equation}
\label{commutator}\sum_{(a,b) \in G^2}\chi([a,b])=\dfrac{|G|}{\chi(1)} .\end{equation}
This equality is obtained by applying Schur's lemma in a classical way as explained in \cite[Paragraph 2.3]{HRV}. Notice that, from the identity $\langle f_n,\chi \rangle=\dfrac{1}{H^{\vee}_{\omega}(q)}$, we immediately deduce that the family $\{f_{n}\}_{n \in \N}$ is  dual Log compatible and, more precisely, for each $\omega$, the associated function $\widetilde{F_{\omega}}(t)$ is equal to $1$.

\end{esempio}

The notion of dual Log compatible families will be one of the key elements to show Theorem \ref{Epolynomialtheorem} about E-series of non-generic character stacks. Their importance comes from the following Theorem.

\begin{teorema}
\label{mainteochar}
For any dual Log compatible family $\{c_{\alpha}\}_{\alpha \in \N^I}$ and any $\eta \in (\F_q^*)^I$, there is an equality

\begin{equation}
\label{mainteochar1}
\dfrac{c_{\alpha}(\eta)}{|\Gl_{\alpha}(\F_q)|}=\Coeff_{\alpha}\left(\Plexp\left(\sum_{\beta \in \mathcal{H}_{\eta}}\widetilde{C}_{\beta,gen}(q)y^{\beta}\right)\right)
\end{equation}
\end{teorema}

\begin{proof}

By Theorem \ref{mainteo} and the Log compatibility of the family $\{\widetilde{C}_{\omega}(t)\}_{\omega \in \mathbb{T}_I}$, to show eq.(\ref{mainteochar1}) it is enough to show that for any $\eta \in (\F_q^*)^I$, we have $$\dfrac{c_{\alpha}(\eta)}{|\Gl_{\alpha}(\F_q)|}=\widetilde{C}_{\alpha,\eta}(q) .$$

 By eq.(\ref{conv2}), we deduce the following  equality \begin{equation}
\label{convolution0}
\dfrac{c_{\alpha}(\eta)}{|\Gl_{\alpha}(\F_q)|} =\sum_{\chi \in \Gl_{\alpha}(\F_q)^{\vee}} \langle c_{\alpha},\chi \rangle\dfrac{\chi(\eta^{-1})}{\chi(e)} \dfrac{\chi(e)}{|\Gl_{\alpha}(\F_q)|}
\end{equation}

Rearranging the sum of the RHS of eq.(\ref{convolution0}) in terms of the multitypes of the characters, we have:

\begin{equation}
\label{convolution}
\sum_{\omega \in \mathbb{T}_{\alpha}}\sum_{\substack{\chi \in \Gl_{\alpha}(\F_q)^{\vee}\\ \chi \sim \omega}}\langle c_{\alpha},\chi \rangle \dfrac{\chi(\eta^{-1})}{\chi(e)}H_{\omega}^{\vee}(q)= \sum_{\omega \in \mathbb{T}_{\alpha}}\widetilde{C}_{\omega}(q)\left(\sum_{\substack{\chi \in \Gl_{\alpha}(\F_q)^{\vee}\\ \chi \sim \omega}} \dfrac{\chi(\eta^{-1})}{\chi(e)}\right).
\end{equation}

By Proposition \ref{valueatacentralelement1} and the description of the irreducible characters of $\Gl_{\alpha}(\F_q)$ given in \cref{construcirredchar}, we can rewrite the RHS of eq.(\ref{convolution}) as follows, using the notations of \cref{zioperonemattinalevanto}: 
\begin{equation}
\label{secondeqdual}
\sum_{\omega \in \mathbb{T}_{\alpha}}\dfrac{\widetilde{C}_{\omega}(q)}{w(\omega)}\left(\sum_{\substack{\theta :T_{S_{\omega}} \to \C^* \\ \Gamma_{\theta}=\Gamma_{S_{\omega}}}}\theta(\eta^{-1})\right)=\sum_{\omega \in \mathbb{T}_{\alpha}}\dfrac{\widetilde{C}_{\omega}(q)}{w(\omega)}\left(\sum_{\substack{S'\leq S_{\omega} \\ \text{ of level } \eta^{-1}}}P_{[S']}(q)\mu(S',S_{\omega})\right)=\widetilde{C}_{\alpha,\eta}(q)
\end{equation}
where the equality at the middle is a consequence of eq.(\ref{sangueelacrime}) and the last equality is a consequence of the fact that $\mathcal{H}_{\eta,\alpha}=\mathcal{H}_{\eta^{-1},\alpha}$.

\end{proof}

\vspace{12 pt}

\begin{oss}
\label{duallogcompatiblegeneric}
For $\beta \in \N^I$ and $\eta \in (\F_q^*)^I$, we say that $\eta$ is \textit{generic} with respect to $\beta$ if $\mathcal{H}_{\eta,\beta}=\{\beta\}$.

For any $\beta$, if $q$ is sufficienty big, there exists an element $\eta$ of $(\F_q^*)^I$ generic with respect to it. From eq.(\ref{mainteochar1}), we deduce that, if $\eta$ is generic with respect to $\beta$, we have  \begin{equation}
\label{duallogcompatibleclassfunctions2}
\dfrac{c_{\beta}(\eta)}{|\Gl_{\beta}(\F_q)|}=\widetilde{C}_{\beta,gen}(q) .
\end{equation}

In particular, we remark that, for generic central elements, the quantity  $\dfrac{c_{\beta}(\eta)}{|\Gl_{\beta}(\F_q)|}$ is  given by a rational function in $q$ which does not depend on the choice of the generic $\eta$ but only on the dimension vector $\beta$.

\vspace{8 pt}

Fix now $\alpha \in \N^I$. Assume that $q$ is sufficiently big and for any $0 < \beta \leq \alpha$, choose $\eta_{\beta,gen} \in (\F_q^*)^I$ generic with respect to $\beta$.

Eq.(\ref{mainteochar1}) and eq.(\ref{duallogcompatibleclassfunctions2}) give therefore a way to express the multiplicity $\dfrac{c_{\alpha}(\eta)}{|\Gl_{\alpha}(\F_q)|} $ for any central element $\eta \in (\F_q^*)^I$ in terms of the analogous values for the generic parameters $\eta_{\beta,gen}$.

\end{oss}

\vspace{12 pt }

The notion of dual Log compatibility is well behaved with respect to convolution, as explained by the following Lemma.

\begin{lemma}
\label{duallog1}

Let $\{f_{\alpha}\}_{\alpha \in \N^I},\{f'_{\alpha}\}_{\alpha \in \N^I}$ be two families of dual Log compatible class functions. The family $\{k_{\alpha}\}_{\alpha \in \N^I}$, defined as $k_{\alpha}\coloneqq \dfrac{f_{\alpha} \ast f'_{\alpha}}{q^{\sum_{i \in I}\alpha_i^2}}$ is dual Log compatible.

\end{lemma}

\begin{proof}

Let $F_{\omega}(t),F'_{\omega,\alpha}(t)$ be the polynomials such that $\langle f_{\alpha},\chi \rangle=F_{\omega}(q)$ and $\langle f'_{\alpha},\chi \rangle=F'_{\omega}(q) $ for every $\chi \in \Gl_{\alpha}(\F_q)^{\vee}$ of multitype $\omega \in \mathbb{T}_{\alpha}$. By eq.(\ref{convol}), we see that $$\langle k_{\alpha},\chi \rangle =\dfrac{F_{\omega}(q)F'_{\omega}(q)}{H^{\vee}_{\omega}(q)q^{\sum_{i \in I}\alpha_i^2}} .$$

Fix $d_1,\dots,d_r \in \N$ and $\omega_1 \in \mathbb{T}_{\beta_1},\dots,\omega_r \in \mathbb{T}_{\beta_r}$ such that $\psi_{d_1}(\omega_1) \ast \cdots \ast \psi_{d_r}(\omega_r)=\omega.$ To check eq.(\ref{duallog2}) for the functions $k_{\alpha}$, we need to verify that \begin{equation}
\label{duallog3}
\prod_{j=1}^r \dfrac{F_{\omega_j}(t^{d_j})F'_{\omega_j}(t^{d_j})}{H^{\vee}_{\omega_j}(t^{d_j})t^{d_j\sum_{i \in I}(\beta_j)_i^2}}\prod_{j=1}^r H^{\vee}_{\omega_j}(t^{d_j})=\dfrac{F_{\omega}(t)F'_{\omega}(t)}{H^{\vee}_{\omega}(t)t^{\sum_{i \in I}\alpha_i^2}}H^{\vee}_{\omega}(t).
\end{equation} Since the families $\{f_{\alpha}\}_{\alpha \in \N^I},\{f'_{\alpha}\}_{\alpha \in \N^I}$ are dual Log compatible, this is equivalent to  verify the equality \begin{equation}
\label{duallog4}
\left(\dfrac{\prod_{j=1}^r H^{\vee}_{\omega_j}(t^{d_j})}{H^{\vee}_{\omega}(t)}\right)^2=\dfrac{t^{\sum_{i \in I}\alpha_i^2}}{\prod_{j=1}^r t^{\sum_{j=1}^r d_j\sum_{i \in I}(\beta_j)_i^2}}
\end{equation}

which is a direct consequence of Identity (\ref{duallog5}).

\end{proof}

\begin{oss}
\label{inductivestep}
From Lemma \ref{duallog1} above and Example \ref{conv3}, we deduce that for any $g \geq 1$, the family of class function $\{f^g_n:\Gl_n(\F_q) \to \C\}$, where $$f^g_n(h)\coloneqq\dfrac{\#\{(x_1,y_1,\dots,x_g,y_g) \ | \ \prod_{i=1}^g [x_i,y_i]=h\}}{q^{(n^2(g-1))}} $$ is dual Log compatible.

In \cite[Section 2.3]{HRV}, the authors prove by a direct computation that $\langle f^g _n,\chi \rangle =\left(\dfrac{|\Gl_n(\F_q)|}{\chi(1)}\right)^{2g-1}$, from which it is possible to check dual Log compatibility directly from eq.(\ref{duallog2}).
\end{oss}

\section{Multiplicative quiver stacks for star-shaped quivers and character stacks for Riemann surfaces}
\label{chaptercharstack}

In this chapter, we will apply the results of Theorem \ref{mainteochar} to  the count of points over finite fields of multiplicative quiver stacks and character stacks for Riemann surfaces. We start by recalling the definitions and the constructions of these objects.

\subsection{Character stacks for Riemann surfaces}

\label{paragraphcharstacks}

Fix integers $g,k \in \N$, a Riemann surface 
 $\Sigma$ of genus $g$ and a divisor $D=\{d_1,\dots,d_k\} \subseteq \Sigma$. In this paragraph we recall the definition of the character stacks for the Riemann surface $\Sigma$ with fixed monodromies along $D$.

\vspace{10 pt}

Let $\mathcal{C}=(\mathcal{C}_1,\dots,\mathcal{C}_k)$ be a $k$-tuple of conjugacy classes of $\Gl_n(\C)$. Denote by $X_{\mathcal{C}}$ the following variety $$X_{\mathcal{C}}\coloneqq \{\rho \in \Hom(\pi_1(\Sigma \setminus D),\Gl_n(\C)) \ | \ \rho(\delta_h) \in \overline{\mathcal{C}_h}  \ \text{ for } h=1,\dots,k\} $$ where, for each $h=1,\dots,k$,  $\delta_h$ is a loop around the point $d_h$. The variety $X_{\mathcal{C}}$ is the variety of representations of the fundamental group of $\Sigma \setminus D$ with prescribed monodromy around the points of $D$ lying in $\overline{\mathcal{C}_1},\dots,\overline{\mathcal{C}_k}$ respectively.

\vspace{8 pt}

Recall that the fundamental group $\pi_1(\Sigma\setminus D)$
admits the following explicit presentation $$\pi_1(\Sigma \setminus D)=\langle a_1,b_1,\dots,a_g,b_g,\delta_1,\dots,\delta_k \ | \  [a_1,b_1]\cdots [a_g,b_g]\delta_1\cdots \delta_k=1 \rangle .$$
Therefore, the variety $X_{\mathcal{C}}$ has the following explicit expression in terms of matrix equations:
$$X_{\mathcal{C}}=\Biggl\{(A_1,B_1,\dots,A_g,B_g,X_1,\dots,X_k)\in \Gl_n(\C)^{2g} \times \overline{\mathcal{C}_1} \times \cdots \times \overline{\mathcal{C}_k} \ | \ [A_1,B_1]\cdots[A_g,B_g]X_1\cdots X_k=1\Biggr\} .$$

The character stack $\mathcal{M}_{\mathcal{C}}$ associated to $(\Sigma,D,\mathcal{C})$ is defined as the quotient stack $$\mathcal{M}_{\mathcal{C}}\coloneqq [X_{\mathcal{C}}/\Gl_n(\C)] .$$

We also consider the character variety $M_{\mathcal{C}}$, defined as the GIT quotient $$M_{\mathcal{C}}\coloneqq X_{\mathcal{C}}/\!\!/ \Gl_n(\C) .$$

\vspace{4 pt}

We will also consider certain quotient stacks, defined from Springer resolution of conjugacy classes, whose definition we briefly review here.

\subsubsection{Springer resolutions of conjugacy classes}
\label{springerresolutions}
In this paragraph, the base field is $\C$. Consider a Levi subgroup $L \subseteq \Gl_n(\C)$ and  a parabolic subgroup $P \supseteq L$ having $L$ as Levi factor. Let $U_P \subseteq P$ be the unipotent radical. Fix an element $z \in Z_L$ and let $Y_z$ be the variety $$Y_z \coloneqq \{(X,gP) \in \Gl_n(\C) \times \Gl_{n}(\C)/P \ | \ g^{-1}Xg \in zU_P \} .$$ Let $\pi_z:Y_z \to \Gl_n(\C)$ be the projection $\pi_z((gP,X))=X$. The following proposition is well known (see for instance \cite{Bao} and the reference therein for unipotent orbits).

\begin{prop}
The image of $\pi_z$ is the Zariski closure $\overline{C}$ of a conjugacy class $ C \subseteq \Gl_n(\C)$ and the morphism $\pi_z$ is a resolution of singularities.

If $z \in (Z_L)^{\reg}$, the map $\pi_z$ is an isomorphism between $Y_z$ and the conjugacy class of $z$ in $\Gl_n(\C)$.
\end{prop}

The morphism $\pi_z:Y_z \to \overline{C} \subseteq \Gl_n(\C)$ is called a \textit{partial Springer resolution}. Its image does not depend on the choice of the parabolic subgroup $P$.

\begin{oss}
The variety $Y_z$ can be described in the following equivalent way. Consider $n_0,\dots,n_s$ such that $L \cong \Gl_{n_0} \times \cdots \times \Gl_{n_s}$. The element $z \in Z_L$ corresponds therefore to $z_0,\dots,z_s \in \C^*$ such that $$z=(z_0 I_{n_0}, \dots, z_s I_{n_s}) .$$

Let $P$ be the parabolic subgroup containing $L$ as a Levi factor and which contains the subgroup of upper triangular matrices. Identify $\Gl_n/P$ with the corresponding partial flag variety in the classical way, i.e. \begin{equation}\label{flagvariety}
\Gl_n/P \cong \{\mathcal{F}= (\mathcal{F}_s \subseteq \mathcal{F}_{s-1} \subseteq \cdots \subseteq \mathcal{F}_{0}=\C^n) \ | \ \dim(\mathcal{F}_i)=\sum_{j=i}^s n_j \}. \end{equation} We have $$Y_z=\Bigg\{(X,\mathcal{F}) \in \Gl_n(\C) \times \Gl_{n}(\C)/P  \ | \ X(\mathcal{F}_j) \subseteq \mathcal{F}_j \text{ for each $j=0,\dots,s$} $$ $$ \text{ and the morphism induced by } X \text{ on } \mathcal{F}_j/\mathcal{F}_{j+1} \text{ is } z_jI_{n_j} \Bigg\} $$
\end{oss}

Consider now a $k$-tuple of Levi subgroups $\mathbf{L}=(L_1,\dots,L_k)$, a $k$-tuple of parabolic subgroups $\mathbf{P}=(P_1,\dots,P_k)$ where each $P_i$ has $L_i$ as Levi factor and of central elements $\mathbf{z}=(z_1,\dots,z_k) \in Z_{L_1} \times \cdots Z_{L_k}$.

Let now $X_{\mathbf{L},\mathbf{P},\mathbf{z}}$ be the variety defined as $$X_{\mathbf{L},\mathbf{P},\mathbf{z}}\coloneqq\Biggl\{(A_1,B_1,\dots,A_g,B_g,g_1P_1,X_1,\dots) \in \Gl_n^{2g}(\C) \times  \prod_{h=1}^k Y_{z_h} \ | \ 
 \prod_{i=1}^g [A_i,B_i]X_1\cdots X_k=1 
\Biggr\} $$

and $\mathcal{M}_{\mathbf{L},\mathbf{P},\mathbf{z}}$ the quotient stack $$\mathcal{M}_{\mathbf{L},\mathbf{P},\mathbf{z}}\coloneqq [X_{\mathbf{L},\mathbf{P},\mathbf{z}}/\Gl_n(\C)] .$$

\vspace{2 pt}

Let $\mathcal{C}=(\mathcal{C}_1,\dots,\mathcal{C}_k)$ be the $k$-tuple such that $\overline{\mathcal{C}_j}$ is the image of the projection $Y_{z_j} \to \Gl_n$.

Notice that the projections $\pi_{z_1},\dots,\pi_{z_h}$ induce a morphism $$\pi: X_{\mathbf{L},\mathbf{P},\mathbf{z}} \to X_{\mathcal{C}} $$ $$(A_1,B_1,\dots,A_g,B_g,g_1P_1,X_1,\dots,g_kP_k,X_k) \to (A_1,B_1,\dots,A_g,B_g,X_1,\dots,X_k)  .$$

As $\pi$ is $\Gl_n(\C)$-equivariant, it descends to a morphism of quotient stacks , which we still denote by $\pi:\mathcal{M}_{\mathbf{L},\mathbf{P},\mathbf{z}}\to \mathcal{M}_{\mathcal{C}}$.

\vspace{4 pt}

\begin{oss}
\label{isomorphismsemisimple}
Notice that, if  $z_1 \in (Z_{L_1})^{\reg},\dots,z_k \in (Z_{L_k})^{\reg}$, i.e. if each $\mathcal{C}_h$ is semisimple, the morphism $\pi$ is actually an isomorphism.

\end{oss}

\vspace{6 pt}

\begin{oss}
The morphism $\pi$ is obtained by restricting the product of the partial Springer resolutions $Y_{z_h} \to \mathcal{C}_h$ and then quotienting by $\Gl_n$. The decomposition theorem (and its equivariant version) for partial Springer resolutions are well understood in terms of the representation theory of Weyl groups. 

Although we will not cover this in this article, it is natural to expect that the cohomological properties of the morphism $\pi$ could have a similar description.
\end{oss}

\vspace{6 pt}

In what follows, we will show how to relate the stacks $\mathcal{M}_{\mathbf{L},\mathbf{P},\mathbf{z}}$ to multiplicative quiver stacks for star-shaped quivers. We start by recalling some generalities about quivers and their multiplicative moment maps and fixing some notations.

\subsection{Quiver representations}
\label{quiverrep}
A quiver $Q$ is an oriented graph $Q=(I,\Omega)$, where $I$ is its set of vertices and  $\Omega$ is its set of arrows. We will always assume that $I,\Omega$ are finite. For an arrow $a:i \to j$  in $\Omega$ we denote by $i=t(a)$ its \textit{tail} and by $j=h(a)$ its \textit{head}. 

Fix a field $K$. A representation $M$ of $Q$ over $K$ is given by a (finite dimensional) $K$-vector space $M_i$ for each vertex $i \in I$ and by linear  maps $M_a:M_{t(a)} \to M_{h(a)}$ for each $a \in \Omega$. 

Given two representations $M,M'$ of $Q$, a morphism $f:M \to M'$  is given by maps $f_i:M_i \to M'_i$ such that, for all $a \in \Omega $, the following equality holds:  $f_{h(a)} M_a=M'_a f_{t(a)} $.
The category of representations of $Q$ over $K$ is denoted by $\Rep_K(Q)$. For a representation $M$, the \textit{dimension vector} $\dim M\in \N^I$ is the vector $\dim M=(\dim M_i)_{i \in I}  $. It is an isomorphism invariant of the category $\Rep_K(Q)$.

For a representation $M$ of dimension $\alpha$, up to fixing a basis of the vector spaces $M_i$ for each $i \in I$, we can assume that $M_i=K^{\alpha_i}$. For $a \in \Omega$, the linear map $M_{a}:K^{t(a)} \to K^{h(a)}$ can be therefore identified with a matrix in $\Mat(\alpha_{h(a)},\alpha_{t(a)},K)$. 

Consider then the affine space $$R(Q,\alpha)\coloneqq \bigoplus_{a \in \Omega} \Mat(\alpha_{h(a)},\alpha_{t(a)},K) .$$
 We can endow $R(Q,\alpha)$ with the action of the group $\Gl_{\alpha}=\displaystyle \prod_{i \in I}\Gl_{\alpha_i}$  defined as  $$g \cdot (M_a)_{a \in \Omega}=(g_{h(a)}M_ag_{t(a)}^{-1})_{a \in \Omega} .$$ The orbits of this action are exactly the isomorphism classes of representations of $Q$ of $\dim=\alpha$.

\vspace{8 pt}

Finally, denote by $(-,-):\Z^I \times \Z^I \to \Z$ the Euler form of $Q$, defined as $$(\alpha,\beta)=\sum_{i \in I}\alpha_i \beta_i- \sum_{a \in \Omega}  \alpha_{t(a)}\beta_{h(a)} .$$

\subsection{Star-shaped quivers and multiplicative quiver stacks}
\label{defcharstack}
Let $Q=(I,\Omega)$ be the following star-shaped quiver with $g$ loops on the central vertex 
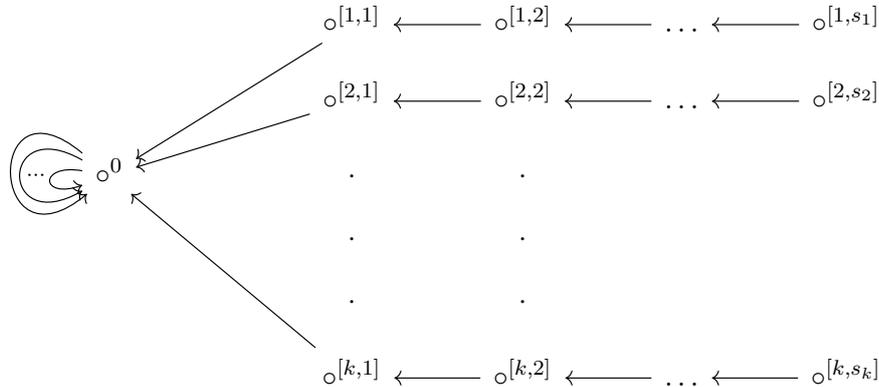
\begin{center}
    \begin{tikzcd}[row sep=1em,column sep=3em]
    & &\circ^{[1,1]} \arrow[ddll,""] &\circ^{[1,2]} \arrow[l,""]  &\dots \arrow[l,""] &\circ^{[1,s_1]} \arrow[l,""]\\
    & &\circ^{[2,1]} \arrow[dll,""] &\circ^{[2,2]} \arrow[l,""] &\dots \arrow[l,""] &\circ^{[2,s_2]} \arrow[l,""]\\
    \circ^0 \arrow[out=170,in=200,loop,swap] \arrow[out=150,in=210,loop,"\cdots"] \arrow[out=140,in=220,loop,swap]  & &\cdot &\cdot\\
    & &\cdot &\cdot\\
    & &\cdot &\cdot\\
    & &\circ^{[k,1]} \arrow[uuull,""] &\circ^{[k,2]} \arrow[l,""]  &\dots \arrow[l,""] &\circ^{[k,s_k]} \arrow[l,""]
    \end{tikzcd}
\end{center}

\vspace{8 pt}

Let $(\N^I)^* \subseteq \N^I$ be the subset of vectors with non-increasing coordinates along the legs and, more generally, for any subset $V \subseteq \N^I$, denote by $V^*\coloneqq V \cap (\N^I)^*$. 

Denote by $\overline{Q}$ the double quiver $\overline{Q}=(I,\overline{\Omega})$ with the same set of  vertices of $Q$ and as set of arrows $\overline{\Omega}=\{a,a^* \ | \ a \in \Omega\}$, where $a^*:h(a) \to t(a)$. 

\vspace{6 pt}

For a representation $x \in R(Q,\alpha)$, for each $h=1,\dots,k$ and $j=0,\dots,s_h$, we denote by $x_{h,j} \in \Mat(\alpha_{[h,j]},\alpha_{[h,j+1]},K)$ the matrix associated to the arrow $a$ having $t(a)=[h,j+1]$ and $h(a)=[h,j]$, where we put $x_{h,s_h}=0$. 
 
 Similarly, for an element $\overline{x} \in R(\overline{
Q},\alpha)$ we denote by $x_{h,j}^* \in \Mat(\alpha_{[h,j+1]},\alpha_{[h,j]},K) $ the matrix associated to the arrow $a^* \in \overline{\Omega}$.

Lastly, for a representation $\overline{x} \in R(\overline{Q},\alpha)$, we denote by $e_1,\dots,e_g,e_1^*,\dots,e_g^* \in \Mat(\alpha_0,K)$ the matrix associated to the $g$ loops of $Q$ and the corresponding reversed arrows of $\overline{Q}$ respectively.
\vspace{4 pt}

 For $\alpha \in \N^I$, let $R(\overline{Q},\alpha)^{\circ} \subseteq R(\overline{Q},\alpha)$ be the open subset of representations $(x_a,x_{a^*})_{a \in \Omega}$ such that $(1+x_ax_{a^*}),(1+x_{a^*}x_a)$ is invertible for every $a \in \Omega$. 

Let moreover $R(\overline{Q},\alpha)^{\circ,*} \subseteq R(\overline{Q},\alpha)^{\circ}$ be the open subset of representations $(x_a,x_{a^*})_{a \in \Omega}$ such that $x_a$ is injective for each $a \in \Omega$. Notice that $R(\overline{Q},\alpha)^{\circ,\ast}=\emptyset$ if $\alpha \notin (\N^I)^*$.

\vspace{8 pt}

Assume to have fixed an ordering $<$ on $\Omega$. The multiplicative moment map $\Phi_{\alpha}^{*}$ is the $\Gl_{\alpha}$-equivariant morphism
$$\Phi^{*}_{\alpha}: R(\overline{Q},\alpha)^{\circ,\ast} \to \Gl_{\alpha} $$ 
\begin{equation}
\label{definitionmultiplicativemomentmap}
(x_a,x_{a^*}) \rightarrow \prod_{a \in \Omega}(1+x_ax_{a^*})(1+x_{a^*}x_a)^{-1}
\end{equation}
where we are taking the ordered product with respect to $<$.

 For $\sigma \in (K^*)^I$, the fiber $(\Phi^{*}_{\alpha})^{-1}(\sigma)$ is $\Gl_{\alpha}$-stable and we define the multiplicative quiver stack $\mathcal{M}^{*}_{\sigma,\alpha}$ of parameter $\sigma,\alpha$ as the quotient stack $$\mathcal{M}^{*}_{\sigma,\alpha}\coloneqq [(\Phi^{*}_{\alpha})^{-1}(\sigma)/\Gl_{\alpha}] .$$

\begin{oss}
The isomorphism class of the multiplicative quiver stack $\mathcal{M}^{*}_{\sigma,\alpha}$ does not depend on the ordering $<$ of the arrows, see for example \cite[Theorem 1.4]{cb-shaw}.

\end{oss}

\vspace{10 pt}

\begin{esempio}
\label{Jordanquivermomentmap}
Let $Q=(I,\Omega)$ be the Jordan quiver, i.e. the quiver with one vertex and one arrow. For $n \in \N$, the variety $R(\overline{Q},n)$ is $\mathfrak{gl}_n(K) \times \mathfrak{gl}_n(K)$ and the variety $R(\overline{Q},n)^{\circ,\ast}$ is given by $$R(\overline{Q},n)^{\circ,\ast}=\{(e,e^*) \in \mathfrak{gl}_n(K) \times \mathfrak{gl}_n(K) \ | \ e, 1+ee^*,1+e^*e \in \Gl_n(K) \} .$$

Notice that the variety $R(\overline{Q},n)^{\circ,\ast}$ is isomorphic to $\Gl_n(K) \times \Gl_n(K)$ via the isomorphism $$R(\overline{Q},n)^{\circ,\ast} \to \Gl_n(K) \times \Gl_n(K) $$ $$(e,e^*) \to (e,e^{-1}+e^*) .$$

Through this identification, the multiplicative moment map $\Phi^{\ast}_n$ corresponds to the morphism $$\Phi_n^{\ast}:\Gl_n(K) \times \Gl_n(K) \to \Gl_n(K)$$  $$\Phi^{\ast}_n(A,B)=[A,B] .$$

\end{esempio}

\vspace{6 pt}

For a point $\overline{x} \in (\Phi_{\alpha}^{*})^{-1}(\sigma)$, we have the following relationships. At the central vertex, we have:

\begin{equation}
\label{centralvertexequation}
\prod_{l=1}^g (1+e_le_l^*)(1+e_l^*e_l)^{-1}\prod_{h=1}^k (1+x_{h,0}x_{h,0}^*)=\sigma_0 I_{\alpha_0}
\end{equation}

For any $h=1,\dots,k$ and $j=1,\dots, s_h$, we have
\begin{equation}
(1+x_{h,j}x^*_{h,j})(1+x_{h,j-1}^*x_{h,j-1})^{-1}=\sigma_{[h,j]}I_{\alpha_{[h,j]}}    
\end{equation}
 which can be rewritten as
\begin{equation}
\label{relationshipmultiplicative1}
x_{h,j}x^*_{h,j}-\sigma_{[h,j]}x_{h,j-1}^*x_{h,j-1}=(\sigma_{[h,j]}-1)I_{\alpha_{[h,j]}}.
 \end{equation}

For $j=s_h$, we have
\begin{equation}
\sigma_{[h,s_h]}x_{h,s_h-1}^*x_{h,s_h-1}=(1-\sigma_{[h,s_h]})I_{\alpha_{[h,s_h]}}.
\end{equation}

\subsection{Multiplicative quiver stacks and character stacks: the main isomorphism}

Consider a $k$-tuple of Levi subgroups $\mathbf{L}=(L_1,\dots,L_k)$, a $k$-tuple of parabolic subgroups $\mathbf{P}=(P_1,\dots,P_k)$ such that $P_h$ has $L_h$ as Levi factor for every $h$ and a $k$-tuple of central elements $\mathbf{z}=(z_1,\dots,z_k) \in Z_{L_1} \times \cdots \times Z_{L_k}$, as in \cref{springerresolutions}.

Assume that, for each $h=1,\dots,k$,  $L_h$ is the Levi subgroup of $\Gl_n(\C)$ $$L_h=\prod_{j=0}^{s_h}\Gl_{n_{h,j}}(\C) $$ and $$z_h=(z_{h,0}I_{n_{h,0}},\dots,z_{h,s_h}I_{n_{h,s_h}})$$ with $z_{h,0},\dots,z_{h,s_h} \in \C^*$.

\vspace{2 pt}

Consider now a star-shaped quiver $Q=(I,\Omega)$ with $g$ loops on the central vertex and $k$ legs of length $s_1,\dots,s_k$ respectively. Define the following dimension vector $\alpha_{\mathbf{L}} \in (\N^I)^*$. For each $h=1,\dots,k$ and $j=0,\dots,s_h$, put \begin{equation}
\label{partitionsofdimensionvector}
\displaystyle (\alpha_{\mathbf{L}})_{[h,j]} \coloneqq \sum_{l=j}^{s_h}n_{h,j}\end{equation} where we are identifying $[h,0]=0$. Define the following element $\sigma_{\mathbf{z}} \in (\C^*)^I$. $$\begin{cases}
(\sigma_{\mathbf{z}})_0\coloneqq \displaystyle \prod_{j=1}^k z^{-1}_{j,0} \\
(\sigma_{\mathbf{z}})_{[h,j]}\coloneqq z_{h,j-1}z_{h,j}^{-1} \text{ if } j \geq 1.
\end{cases}$$

We have the following result, relating multiplicative quiver stacks for star-shaped quivers and Springer resolutions of conjugacy classes.

\begin{teorema}
\label{teoremamultiplicative}
For any $\mathbf{L},\mathbf{P},\mathbf{z}$ as above, there is an isomorphism of stacks $$\mathcal{M}^*_{\alpha_{\mathbf{L}},\sigma_{\mathbf{z}}} \cong \mathcal{M}_{\mathbf{L},\mathbf{P},\mathbf{z}} .$$
\end{teorema}

In the proof, we suppose $g=0$ and  we put $\sigma=\sigma_{\mathbf{z}},\alpha=\alpha_{\mathbf{L}}$ to simplify the notations. The case of $g >0$ is a combination of the arguments used in this proof and that of \cite[Proposition 2]{cb-monod}. 

\begin{proof}
We define the following morphism $$f: (\Phi^*_{\alpha})^{-1}(\sigma) \to X_{\mathbf{L},\mathbf{P},\mathbf{z}}.$$

For an element $\overline{x} \in (\Phi^*_{\alpha})^{-1}(\sigma)$, consider the flag $$\mathcal{F}_{j,\overline{x}}=(\C^n \supseteq \Imm(x_{j,0}) \supseteq \Imm(x_{j,0}x_{j,1}) \supseteq \cdots \supseteq \Imm(x_{j,0}\cdots x_{j,s_j-1})) .$$

Notice that, for each $h=0,\dots,s_j-1$, we have $$\dim(\Imm(x_{j,0}\cdots x_{j,h}))=\alpha_{[j,h+1]} ,$$ since $x_{j,r}$ is injective for each $j$ and $r$. In particular, $\mathcal{F}_{j,\overline{x}}$ belongs to the partial flag variety $\Gl_n/P_j$. We define therefore $$f(\overline{x})=(\mathcal{F}_{1,\overline{x}},z_{1,0}+z_{1,0}x_{1,0}x_{1,0}^*,\mathcal{F}_{2,\overline{x}},\dots,z_{k,0}+z_{k,0}x_{k,0}x_{k,0}^* ) .$$

 For each $h=1,\dots,k$, put $X_h \coloneqq z_{h,0}+z_{h,0}x_{h,0}x_{h,0}^*$. Notice that from eq.(\ref{centralvertexequation}), we have that $$X_1 \cdots X_k=1 .$$

To check that the morphism $f$ is well defined we need to check thus the following two conditions.
\begin{enumerate}
\item The flag $\mathcal{F}_{h,\overline{x}}$ is $X_h$ invariant for each $h$.
\item The morphism that $X_h$ induces on the quotient space $$\Imm(x_{h,0} \cdots x_{h,j-1})/\Imm(x_{h,0} \cdots x_{h,j}) $$ and which we denote by $\overline{X}_{h,j}$ is equal to $z_{h,j}I_{n_{h,j}}$. Here we are putting $\Imm(x_{h,-1})=\C^n$.
\end{enumerate}

From eq.(\ref{relationshipmultiplicative1}), by recurrence, we deduce that for each $h=1,\dots,k$ and each $j=0,\dots,s_h-1$ and $v \in \C^{\alpha_{[h,j]}}$, we have

\begin{equation}
\label{charstackverifying}
X_j(x_{h,0}\cdots x_{h,j-1}(v))=z_{h,0}x_{h,0}\cdots x_{h,j}(v)+z_{h,0}x_{h,0}x_{h,0}^*x_{h,0}\cdots x_{h,j}(v)=
\end{equation}

\begin{equation}
\label{charstackverifying1}
=\dfrac{z_{h,0}}{\sigma_{[h,1]}\cdots \sigma_{[h,j]}}x_{h,0}\cdots x_{h,j-1}(v)+x_{h,0}\cdots x_{h,j}x_{h,j}x_{h,j}^*(v)
\end{equation}

where  we are putting $x_{h,s_h}=x_{h,s_h}^*=0$. 

Notice that $$\dfrac{z_{h,0}}{\sigma_{[h,1]}\cdots \sigma_{[h,j]}}=z_{h,j}.$$ From eq.(\ref{charstackverifying1}), we deduce therefore that properties $1),2)$ above are respected for each $h,j$ and therefore $f$ is well defined.

\vspace{6 pt}

Denote by $I'=I\setminus \{0\}$ and put $\displaystyle \Gl_{\alpha}'\coloneqq \prod_{i \in I'}\Gl_{\alpha_i}$. We have thus that \begin{equation}
    \label{splittinggroup}
\Gl_{\alpha}=\Gl_n \times \Gl_{\alpha}' .\end{equation}  It can be verified that the action of $\Gl_{\alpha}'$ on $ (\Phi_{\alpha}^*)^{-1}(\sigma)$ is schematically free and therefore the multiplicative quiver stack $[(\Phi_{\alpha}^*)^{-1}(\sigma)/\Gl'_{\alpha}]$ is actually an algebraic variety.

In addition, notice that the map $f$ is $\Gl_{\alpha}'$-invariant. Denote by $$\widetilde{f}:(\Phi_{\alpha}^*)^{-1}(\sigma)/\Gl_{\alpha}' \to X_{\mathbf{L},\mathbf{P},\mathbf{z}}$$ the associated morphism. From the identity (\ref{splittinggroup}), to show that $$\mathcal{M}^*_{\alpha,\sigma} \cong \mathcal{M}_{\mathbf{L},\mathbf{P},\mathbf{z}} ,$$ it is sufficient to show that $\widetilde{f}$ is an isomorphism. 

\vspace{6 pt}

We define the following morphism $\theta:X_{\mathbf{L},\mathbf{P},\mathbf{z}} \to  (\Phi_{\alpha}^*)^{-1}(\sigma)/\Gl_{\alpha}'$. Consider an element $$(\mathcal{F}_1,X_1,\dots,\mathcal{F}_k,X_k) \in X_{\mathbf{L},\mathbf{P},\mathbf{z}} .$$ For each $h=1,\dots,k$ and $j=1,\dots,s_{h}$, fix a basis of the vector space $\mathcal{F}_{h,j}$ and denote by $$x_{h,j-1}:\C^{\alpha_{[h,j]}} \to \C^{\alpha_{[h,j-1]}} $$ the matrix such that $z_{h,j-1}x_{h,j-1}$ is the matrix of the  inclusion  $\mathcal{F}_{h,j} \subseteq \mathcal{F}_{h,j-1}$ in the respective fixed basis.

By definition of $X_{\mathbf{L},\mathbf{P},\mathbf{z}}$, we have that $$(X_h -z_{h,j}I_{n})(\mathcal{F}_{h,j}) \subseteq \mathcal{F}_{h,j+1} ,$$ i.e. $X_h -z_{h,j}I_{n}$ defines a morphism $\mathcal{F}_{h,j}  \to \mathcal{F}_{h,j+1}$ and we denote by $$x_{h,j}^*:\C^{\alpha_{[h,j]}} \to \C^{\alpha_{[h,j+1]}} $$ its associated matrix in the fixed basis.

By definition, for each $h=1,\dots,k$, we have

\begin{equation}
\label{verification1}
X_h=z_{h,0}+z_{h,0}x_{h,0}x_{h,0}^*
\end{equation}

Moreover, for each $j=1,\dots,s_h$, we have that $z_{h,j}x_{j,h}x_{h,j}^*$ is the matrix associated to the morphism $$X_{j}-z_{h,j}I_n:\mathcal{F}_{h,j} \to \mathcal{F}_{h,j}$$ and $x_{h,j-1}^*z_{h,j-1}x_{h,j-1}$ is the matrix associated to the morphism $$X_j-z_{h,j-1}I_n:\mathcal{F}_{h,j} \to \mathcal{F}_{h,j} $$ in the respective basis.

In particular, we have that 
\begin{equation}
\label{verification2}
z_{h,j}x_{j,h}x_{h,j}^*-z_{h,j-1}x_{h,j-1}^*x_{h,j-1}=(z_{h,j}-z_{h,j-1})I_{\alpha_{[h,j]}}
\end{equation}
and, since $\dfrac{z_{h,j-1}}{z_{h,j}}=\sigma_{[h,j]}$, we find

\begin{equation}
\label{verification3}
x_{j,h}x_{h,j}^*-\sigma_{[h,j]}x_{h,j-1}^*x_{h,j-1}=(1-\sigma_{[h,j]})I_{\alpha_{[h,j]}}
\end{equation}

By eq.(\ref{relationshipmultiplicative1}), we deduce that $(x_{h,j},x_{h,j}^*)_{\substack{h=1,\dots,k\\j=0,\dots,s_h-1}}$ defines a point $\overline{x} \in (\Phi_{\alpha}^*)^{-1}(\sigma)$ and we put $$\theta(\mathcal{F}_1,X_1,\dots,\mathcal{F}_k,X_k)=\overline{x} .$$

From eq.(\ref{verification1}) and the definition of $\overline{x}$, we deduce that $\theta$ and $\widetilde{f}$ are inverse to one aother, i.e. that $\widetilde{f}$ is an isomorphism.
\end{proof}

\vspace{6 pt}

Consider now a $k$-tuple of semisimple conjugacy classes $\mathcal{C}=(\mathcal{C}_1,\dots,\mathcal{C}_k)$ such that each $\mathcal{C}_h$ is semisimple (i.e.  
  $\overline{\mathcal{C}_h}=\mathcal{C}_h$) and it is the conjugacy class of a diagonal matrix $C_h$ with distinct eigenvalues $\gamma_{h,0},\dots,\gamma_{h,s_h} \in \C^*$ and multiplicities $n_{h,0},\dots,n_{h,s_h}$ respectively. 

Let $Q=(I,\Omega)$ be the star-shaped quiver introduced in \cref{defcharstack}, with $g$ loops on the central vertex and $k$ legs of length $s_1,\dots,s_k$ respectively.

Let $\alpha_{\mathcal{C}} \in (\N^I)^*$ be the dimension vector defined as $$(\alpha_{\mathcal{C}})_{[h,j]}=\sum_{l=j}^{s_h}n_{h,l} 
 $$ and $\gamma_{\mathcal{C}} \in (\C^*)^I$ the element defined as 
$$\displaystyle (\gamma_{\mathcal{C}})_{[i,j]}=\begin{cases} \displaystyle\prod_{i=1}^k\gamma^{-1}_{i,0} \text{ if } j=0\\
\gamma^{-1}_{i,j}\gamma_{i,j-1} \text{ otherwise }

\end{cases}.$$

From Remark \ref{isomorphismsemisimple} and Theorem \ref{teoremamultiplicative}, we deduce the following Theorem:
\begin{teorema}
\label{isomstacks}
For any $k$-tuple of semisimple conjugacy classes $\mathcal{C}$, there is an isomorphism of stacks
$$\mathcal{M}_{\mathcal{C}} \cong \mathcal{M}^{\ast}_{\gamma_{\mathcal{C}},\alpha_{\mathcal{C}}} .$$
\end{teorema}

\vspace{2 pt}

\begin{oss}
The relation between multiplicative moment map for star-shaped quivers and character stacks and varieties was first introduced in the articles \cite[Proposition 2]{cb-monod}, \cite[Lemma 8.2]{cb-shaw}.

The analogous of Theorem \ref{isomstacks} for the corresponding GIT quotients, i.e. that $M_{\mathcal{C}} \cong M^*_{\gamma_{\mathcal{C}},\alpha_{\mathcal{C}}}$ has been shown in \cite[Proposition 4.1]{daisuke}. In \cite[Proposition 4.1]{daisuke}, the author actually shows a stronger statement. In particular, the author does not consider only $k$-tuples of semisimple conjugacy classes.

In this case, it is necessary to slightly modify the definition of the multiplicative moment map and define it as a map $$\Phi^l_{\beta}:R(\overline{Q},\alpha)^{\circ} \to \Gl_{\alpha} ,$$ i.e. we are not asking that all the maps $x_{h,j}$ are injective. The proof of \cite[Proposition 4.1]{daisuke} is quite similar to the one given for Theorem \ref{teoremamultiplicative}.

The main difference is that, in the case of GIT quotients, in the part of the proof of Theorem \ref{teoremamultiplicative} where we use that we restricted to $R(\overline{Q},\alpha)^{\circ,*}$, i.e. that the action of $\Gl_{\alpha}'$ is free, Daisuke instead uses the fundamental theorem of invariant theory. This latter type of technique has already been introduced by \cite{kraft}.

\end{oss}

\subsection{Dual Log compatibility for moment map}
\label{duallogmomentmapchapter}
In this chapter, we show how to relate the results about dual Log compatible families of \cref{duallogcompdefin} to the  study of multiplicative quiver stacks for star-shaped quivers.

\vspace{6 pt}

Consider now the construction of \cref{defcharstack} in the case in which $K=\F_q$. We denote by $m_{\alpha}:\Gl_{\alpha}(\F_q) \to \C$ the class function defined as $$m_{\alpha}(g)\coloneqq \dfrac{|(\Phi^{*}_{\alpha})^{-1}(g)^F|}{q^{-(\alpha,\alpha)}} .$$ 

Notice that, for $\sigma \in (\F_q^*)^I$, we have an equality $$ \dfrac{m_{\alpha}(\sigma)}{|\Gl_{\alpha}(\F_q)|}=\dfrac{\#\mathcal{M}^{*}_{\sigma,\alpha}(\F_q)}{q^{-(\alpha,\alpha)}} .$$

\vspace{10 pt}

For the family of class functions $\{m_{\alpha}\}_{\alpha \in \N^I}$ we have the following Theorem:

\begin{teorema}
\label{charstack}
The family $\{m_{\alpha}\}_{\alpha \in \N^I}$ is dual Log compatible
 \end{teorema}

The proof of Theorem \ref{charstack} will be given in paragraph \cref{bebbo8} below. For the consequences of Theorem \ref{charstack} on the computation of the cohomology of  character stacks of Riemann surfaces over $\C$ see \cref{cohocharstacks}.

\subsection{Proof of Theorem \ref{charstack} }
\label{bebbo8}

Theorem \ref{charstack} will be proved through several steps.

\vspace{6 pt}

We start by showing Theorem \ref{charstack} in the case where $Q=(I,\Omega)$ is  the star-shaped  quiver with two vertices $I=\{0,1\}$ and one arrow $a:1 \to 0$ between them (i.e. $g=0$ and $k=1$). This is usually called the Kronecker quiver.

\begin{oss}

The proof of the case of the Kronecker quiver (i.e. of Lemma \ref{firstep} below) is the main technical point of the proof of Theorem \ref{charstack}. The proof of this Lemma involves computations in the character ring of $\Gl_n(\F_q)$.

In particular, we will have to understand combinatorially multiplicities of the type $\langle \chi_0,R^G_L(\chi_1 \boxtimes \chi_2) \rangle$ where $\chi_0 \in \Gl_m(\F_q)^{\vee}$, $L=\Gl_n \times \Gl_{n'}$ with $n+n'=m$ and $\chi_1 \in \Gl_n(\F_q)^{\vee},\chi_2 \in \Gl_{n'}(\F_q)^{\vee}$.

We will mainly need two results about representations of finite reductive groups that we recalled before.
\begin{itemize}

 \item Proposition \ref{irreducibleinduced} regarding the decomposition into irreducible $\Gl_n(\F_q)$-representations of the Deligne-Lusztig induction of an irreducible $L^F$-representation, with $L \subseteq \Gl_n$ an $F$-stable Levi subgroup.
 
 \item The description of certain connected centralizers in the case of $\Gl_n$ given in \cref{connectedcentralizer1}. 

\end{itemize}

\vspace{2 pt}

We show how to extend the result from the case of the Kronecker quiver to any star-shaped quiver in Proof \ref{zioper}.

\end{oss}

\vspace{6 pt}

For the Kronecker quiver $Q$, a dimension vector $\alpha$ is thus a couple $\alpha=(\alpha_0,\alpha_1) \in \N^2$ and the function $m_{\alpha}(g_0,g_1)$ for a couple $(g_0,g_1) \in \Gl_{\alpha_0}(\F_q) \times \Gl_{\alpha_1}(\F_q)$ is given by $$m_{\alpha}(g_0,g_1)=\dfrac{\#\{f \in \Hom^{inj}(\F_q^{\alpha_1},\F_q^{\alpha_0}),f^* \in \Hom(\F_q^{\alpha_0},\F_q^{\alpha_1}) \ | \ 1+ff^*=g_0 ,\ 1+f^*f=g_1^{-1} \}}{q^{\alpha_0\alpha_1-\alpha_0^2-\alpha_1^2}} .$$

\vspace{10 pt}

\begin{oss}
Notice that given $f \in \Hom^{inj}(\F_q^{\alpha_1},\F_q^{\alpha_0})$ and $f^* \in \Hom(\F_q^{\alpha_0},\F_q^{\alpha_1})$ such that $1+ff^* \in \Gl_{\alpha_0}(\F_q)$, then $1+f^*f$ is invertible too. It is enough to check that $1+f^*f$ is injective. Given $x,y \in \F_q^{\alpha_1}$ such that $(1+f^*f)(x)=(1+f^*f)(y)$ we have indeed $$ f \circ (1+f^*f)(x)=f \circ (1+f^*f)(y) $$ and, given that $f \circ (1+f^*f)=(1+ff^*) \circ f$ and $1+ff^*$ is invertible,  we deduce that $f(x)=f(y)$ and so that $x=y$. 
\end{oss}

\vspace{8 pt}

\begin{lemma}
\label{firstep}

In the case in which $Q$ is the Kronecker quiver, the family $\{m_{\alpha}\}_{\alpha \in \N^I}$ is dual Log compatible.
\end{lemma}

\begin{proof}

We have that $m_{\alpha} \equiv 0$ if $\alpha \not\in (\N^I)^*$. Fix then $\alpha \in (\N^I)^*$ and denote by $\alpha_2=\alpha_0-\alpha_1$. Fix an irreducible character $\chi=\chi_0 \boxtimes \chi_1 \in \Gl_{\alpha}(\F_q)^{\vee}$ with $\chi_i \in \Gl_{\alpha_i}(\F_q)^{\vee}$ for $i=0,1$. We have: \begin{equation}\label{fatto1}
\langle m_{\alpha},\chi \rangle=\dfrac{1}{|\Gl_{\alpha}(\F_q)|q^{-(\alpha,\alpha)}}\sum_{\substack{f \in \Hom^{inj}(\F_q^{\alpha_1},\F_q^{\alpha_0})\\f^* \in \Hom(\F_q^{\alpha_0},\F_q^{\alpha_1}) \\ \text{ s.t } 1+ff^* \in \Gl_{\alpha_0}(\F_q) }}\chi_0(1+ff^*)\chi_1((1+f^*f)^{-1}) .\end{equation}

\vspace{4 pt}

\textit{I Step: rewriting the RHS of eq.(\ref{fatto1})}

\vspace{4 pt}

Let $J_{\alpha} \in \Hom^{inj}(\F_q^{\alpha_1},\F_q^{\alpha_0})$ be the block matrix given by the identity on the first $\alpha_1$ rows and $0$ everywhere else and let $P_{\alpha}$ be the stabiliser of $J_{\alpha}$ inside $\Gl_{\alpha}$. 

Notice that $P_{\alpha}$ is isomorphic to the parabolic subgroup $P \subseteq \Gl_{\alpha_0}$, given by the matrices which preserve the image of $J_{\alpha}$. 

We denote by $L \subseteq P$ the Levi subgroup given by $\Gl_{\alpha_1} \times \Gl_{\alpha_2}$ embedded block diagonally. 

\vspace{8 pt}

The action of $\Gl_{\alpha}(\F_q)$ on $\Hom^{inj}(\F_q^{\alpha_1},\F_q^{\alpha_0})$ is transitive and  we can  therefore identity the latter set with $\Gl_{\alpha}(\F_q)/P_{\alpha}(\F_q)$ via the map which sends $(g_0,g_1)P_{\alpha}(\F_q) \to g_0J_{\alpha}g_1^{-1}$.

We can thus rewrite the sum above as : \begin{equation}
\label{proof1}
\dfrac{1}{|\Gl_{\alpha}(\F_q)|q^{-(\alpha,\alpha)}}\sum_{(g_0,g_1)P_{\alpha}(\F_q) \in \Gl_{\alpha}(\F_q)/P_{\alpha}(\F_q)}
\end{equation}
\begin{equation}
\label{proof1'}
\sum_{\substack{f^* \in \Hom(\F_q^{\alpha_0},\F_q^{\alpha_1}) \\ \text{ s.t. } 1+ g_0J_{\alpha}g_1^{-1}f^{*}\in \Gl_{\alpha_0}(\F_q)}}\chi_0(1+g_0J_{\alpha}g_1^{-1}f^*)\chi_1((1+f^*g_0J_{\alpha}g_1^{-1})^{-1}).
\end{equation}

For each $(g_0,g_1)P_{\alpha}(\F_q)$ we can rewrite the last term of eq.(\ref{proof1'}) as \begin{equation}
\label{proof2}
\sum_{\substack{f^* \in \Hom(\F_q^{\alpha_0},\F_q^{\alpha_1}) \\ \text{ s.t. } 1+ J_{\alpha}g_1^{-1}f^{*}g_0\in \Gl_{\alpha_0}(\F_q)}}\chi_0(g_0(1+J_{\alpha}g_1^{-1}f^*g_0)g_0^{-1})\chi_1(g_1(1+g_1^{-1}f^*g_0J_{\alpha})^{-1}g_1^{-1}).
\end{equation}

As $\chi_0,\chi_1$ are class functions, we can rewrite the sum in eq.(\ref{proof2}) as \begin{equation}
\label{proof3}
\sum_{\substack{f^* \in \Hom(\F_q^{\alpha_0},\F_q^{\alpha_1}) \\ \text{ s.t. } 1+ J_{\alpha}g_1^{-1}f^{*}g_0\in \Gl_{\alpha_0}(\F_q)}}\chi_0(1+ J_{\alpha}g_1^{-1}f^{*}g_0)\chi_1((1+g_1^{-1}f^*g_0J_{\alpha})^{-1}).
\end{equation}

Moreover, for each $(g_0,g_1) \in \Gl_{\alpha}(\F_q)$, we have a bijection 

$$ \{f^* \in \Hom(\F_q^{\alpha_0},\F_q^{\alpha_1}) \ | \ 1+ J_{\alpha}f^{*}\in \Gl_{\alpha_0}\}  \leftrightarrow \{f^* \in \Hom(\F_q^{\alpha_0},\F_q^{\alpha_1}) \ | \  1+ J_{\alpha}g_1^{-1}f^{*}g_0\in \Gl_{\alpha_0}\} $$ $$f^* \to g_1f^*g_0^{-1} .$$

and so \begin{equation}
\label{devocambiaretuttop*}
\sum_{\substack{f^* \in \Hom(\F_q^{\alpha_0},\F_q^{\alpha_1}) \\ \text{ s.t. } 1+ J_{\alpha}g_1^{-1}f^{*}g_0\in \Gl_{\alpha_0}(\F_q)}}\chi_0(1+ J_{\alpha}g_1^{-1}f^{*}g_0)\chi_1((1+g_1^{-1}f^*g_0J_{\alpha})^{-1})=
\end{equation}
$$
\sum_{\substack{f^* \in \Hom(\F_q^{\alpha_0},\F_q^{\alpha_1}) \\ \text{ s.t. } 1+ J_{\alpha}f^{*}\in \Gl_{\alpha_0}(\F_q)}}\chi_0(1+ J_{\alpha}f^{*})\chi_1((1+f^*J_{\alpha})^{-1}).
$$

From eq.(\ref{devocambiaretuttop*}), we deduce that the sum in eq.(\ref{proof1}) can be rewritten as 

\begin{equation}
\label{zzziooo}
\dfrac{1}{|P_{\alpha}(\F_q)|q^{-(\alpha,\alpha)}}\sum_{\substack{f^* \in \Hom(\F_q^{\alpha_0},\F_q^{\alpha_1}) \\ \text{ s.t. } 1+ J_{\alpha}f^{*}\in \Gl_{\alpha_0}(\F_q)}}\chi_0(1+ J_{\alpha}f^{*})\chi_1((1+f^*J_{\alpha})^{-1})
\end{equation}

Writing $f^*$ as a block matrix $(A|B)$ with $A\in \Mat(\alpha_1,\F_q)$ and $B \in \Mat(\alpha_2,\alpha_1,\F_q)$, where $\alpha_2=\alpha_0-\alpha_1$, we have $$ 1+J_{\alpha}f^*=\begin{pmatrix}
1+A &B\\
0 &1
\end{pmatrix}$$
and $1+f^*J_{\alpha}=1+A$.  We can rewrite the sum of eq.(\ref{zzziooo}) as \begin{equation}
\label{proof4}
\dfrac{1}{|P_{\alpha}(\F_q)|q^{-(\alpha,\alpha)}}\sum_{\substack{M \in \Gl_{\alpha_1}(\F_q)\\ B \in \Mat(\alpha_2,\alpha_1,\F_q)}}\chi_0\left(\begin{pmatrix}
M &B\\
0 &1
\end{pmatrix}\right)\chi_1(M^{-1})
\end{equation}

Moreover, for any $M \in \Gl_{\alpha_1}(\F_q)$, we have that $\chi_1(M^{-1})=\chi_1^*(M)$ and we can thus rewrite the sum in eq.(\ref{proof4}) as

\begin{equation}
\label{proof44}
\dfrac{1}{|P_{\alpha}(\F_q)|q^{-(\alpha,\alpha)}}\sum_{\substack{M \in \Gl_{\alpha_1}(\F_q)\\ B \in \Mat(\alpha_2,\alpha_1,\F_q)}}\chi_0\left(\begin{pmatrix}
M &B\\
0 &1
\end{pmatrix}\right)\chi^*_1(M)
\end{equation}

\vspace{4 pt}

\textit{II Step: relate the sum in eq.(\ref{proof44}) to computations in the character ring of $\Gl_{\alpha_0}(\F_q)$}

\vspace{4 pt}

Consider now the class function $H \in \mathcal{C}(\Gl_{\alpha_2}(\F_q) \times \Gl_{\alpha_1}(\F_q))$ defined as $$H \coloneqq \sum_{\chi_2 \in \Gl_{\alpha_2}(\F_q)^{\vee}} (\chi_2 \boxtimes \chi^*_1)\dfrac{\chi_2(1)}{\Gl_{\alpha_2}(\F_q)} .$$
By eq.(\ref{identityelementgroup}), for any $(N,M) \in \Gl_{\alpha_2}(\F_q) \times \Gl_{\alpha_1}(\F_q)$ we have that $$H((N,M))=\begin{cases} 0 \text{ if } N \neq 1 \\
\chi_1^*(M) \text{ otherwise }.
\end{cases} $$ We deduce that the sum in eq.(\ref{proof44}) can be rewritten as \begin{equation}
\label{proof5}
\dfrac{1}{|P(\F_q)|q^{-(\alpha,\alpha)}}\sum_{\chi_2 \in \Gl_{\alpha_2}(\F_q)^{\vee}}\sum_{h \in P(\F_q)}\chi_0(h)\Infl_L^{P}(\chi^*_1 \boxtimes \chi_2)(h)\dfrac{\chi_2(1)}{|\Gl_{\alpha_2}(\F_q)|},
\end{equation}
where we are using that $|P(\F_q)|=|P_{\alpha}(\F_q)|$. Denote by $\Res_P:\mathcal{C}(\Gl_{\alpha_0}(\F_q)) \to \mathcal{C}(P(\mathbb{F}_q))$ the restriction of class functions. We have therefore that the quantity in eq.(\ref{proof5}) is equal to
\begin{equation}
\label{proof6}
=\dfrac{1}{q^{-(\alpha,\alpha)}}\sum_{\chi_2 \in \Gl_{\alpha_2}(\F_q)^{\vee}}\langle \Res_{P}(\chi_0),\Infl_L^P(\chi_1 \boxtimes \chi_2)\rangle \dfrac{\chi_2(1)}{|\Gl_{\alpha_2}(\F_q)|}
\end{equation}
which by Frobenius reciprocity is equal to
\begin{equation}
\label{fatto3}
=\dfrac{1}{q^{-(\alpha,\alpha)}}\sum_{\chi_2 \in \Gl_{\alpha_2}(\F_q)^{\vee}}\langle \chi_0,R^G_L(\chi_1 \boxtimes \chi_2) \rangle \dfrac{\chi_2(1)}{|\Gl_{\alpha_2}(\F_q)|}.
\end{equation}

\vspace{8 pt}

\textit{III Step: understand the RHS of eq.(\ref{fatto3}) for characters of unipotent type}

\vspace{4 pt}

Suppose now that the type of $\chi=\chi_0 \boxtimes\chi_1$ is $(1,\bm\lambda)$, where $\bm\lambda=( \lambda^0, \lambda^1)$ with $\lambda^0 \in \mathcal{P}_{\alpha_0},\lambda^1 \in \mathcal{P}_{\alpha_1}$. We have then $\chi_0=(\gamma \circ \det) R_{\lambda^0}$ and $\chi_1=(\gamma \circ \det) R_{ \lambda^1}$, with $\gamma\in \Hom(\F_q^*,\C^*)$.

\vspace{6 pt}

Let $\chi_2=\epsilon_{\Gl_{\alpha_2}}\epsilon_{L_2}R^{\Gl_{\alpha_2}}_{L_2}(\theta_2 R_{\tilde{\phi_2}})$ for a certain $\phi_2 \in (W_{L_2})^F$ and a certain reduced $\theta_2:L_2^F \to \C^*$. From Lemma \ref{delignel1}, we have an equality $$R^G_L(\chi_1 \boxtimes \chi_2)=\epsilon_{\Gl_{\alpha_2}}\epsilon_{L_2}R^G_{\Gl_{\alpha_1} \times L_2}(((\gamma \circ \det) \times \theta_2) (R_{\lambda^1} \boxtimes R_{\tilde{\phi_2}})) .$$

Let $L'$ be the connected centralizer of $(\gamma \circ \det) \times \theta_2: \Gl_{\alpha_1}(\F_q) \times L^F_2 \to \C^*$. By Remark \ref{irreducibleinduced}, the character $R^G_L(\chi_1 \boxtimes \chi_2)$ belongs to the vector space spanned by irreducible characters with semisimple part $(L',(\gamma \circ \det) \times \theta_2)$.

The multiplicity $\langle(\gamma\circ \det) R_{\lambda_0},R^G_L((\gamma \circ \det) R_{\lambda^1} \boxtimes \chi_2)\rangle $ is therefore equal to  $0$ if $L'$ is different from $\Gl_{\alpha_0}$. Since $\theta_2:L_2^F \to \C^*$ is reduced, from the remarks made in \cref{connectedcentralizer1}, we deduce that if $L'=\Gl_{\alpha_0}$, we must have that $L_2=\Gl_{\alpha_2}$ and that $\chi_2$ is given by $(\gamma \circ \det)R_{\lambda^2}$ for $\lambda^2 \in \mathcal{P}_{\alpha_2}$.

\vspace{4 pt}

From Remark \ref{twist5}, the RHS of eq. (\ref{proof6}) is thus equal to: 
\begin{equation}
\label{proof7}
\dfrac{1}{q^{-(\alpha,\alpha)}}\sum_{\lambda^2 \in \mathcal{P}_{\alpha_2}}\langle(\gamma \circ \det) R_{ \lambda^0},(\gamma \circ \det)R^G_L( R_{ \lambda^1} \boxtimes R_{\lambda^2})\rangle \dfrac{(-1)^{\alpha_2}}{q^{\frac{\alpha_2(\alpha_2-1)}{2}-n(\lambda^2)}H_{\lambda^2}(q)}.
\end{equation}

From eq.(\ref{twist3}) , the sum in eq.(\ref{proof7}) is equal to \begin{equation}
\label{proof8}
\dfrac{1}{q^{-(\alpha,\alpha)}}\sum_{\lambda^2 \in \mathcal{P}_{\alpha_2}}\langle \chi_{\lambda^0},\Ind_{S_{\alpha_1} \times S_{\alpha_2}}^{S_{\alpha_0}}( \chi_{\lambda^1} \boxtimes \chi_{\lambda^2})\rangle \dfrac{(-1)^{\alpha_2}}{q^{\frac{\alpha_2(\alpha_2-1)}{2}-n(\lambda^2)}H_{\lambda^2}(q)}=
\end{equation}
\begin{equation}
=\dfrac{1}{q^{-(\alpha,\alpha)}}\sum_{\lambda^2 \in \mathcal{P}_{\alpha_2}}c_{\lambda^1,\lambda^2}^{\lambda^0}\dfrac{(-1)^{\alpha_2}}{q^{\frac{\alpha_2(\alpha_2-1)}{2}-n(\lambda^2)}H_{\lambda^2}(q)}
\end{equation}

\vspace{10 pt}

For any couple of partitions $(\lambda,\mu)$ denote by $C_{\lambda,\mu}(t) \in \Q(t)$ the function defined as

$$
C_{\lambda,\mu}(t)= \begin{cases}0 \text{ if } |\lambda|< |\mu| \\ \displaystyle
\dfrac{1}{t^{|\lambda||\mu|-|\lambda|^2-|\mu|^2}}\sum_{\nu \in \mathcal{P}_{|\lambda|-|\mu|}}c^{\lambda}_{\mu,\nu}\dfrac{(-1)^{|\lambda|-|\mu|}}{t^{\frac{|\nu|(|\nu|-1)}{2}-n(\nu)}H_{\nu}(t)}.
\end{cases} .$$

The reasoning above shows that for any $\chi \in \Gl_{\alpha}(\F_q)^{\vee}$ of type $(1,\bm \lambda)$, there is an equality
$$\langle m_{\alpha},\chi \rangle=C_{\lambda^0,\lambda^1}(q) .$$

\vspace{8 pt}

\textit{IV Step: understand the RHS of eq.(\ref{fatto3}) for characters of any type}

\vspace{4 pt}

Let now $\delta \in \N^I$ and consider $\chi=\chi_0 \boxtimes \chi_1 \in \Gl_{\delta}(\F_q)^{\vee}$ of multitype $\omega \in \mathbb{T}_{\delta}$, where $\omega=(d_1,\bm \lambda_1)\cdots (d_r,\bm \lambda_r)$, where for $j=1,\dots,r$ we have $\bm \lambda_j=(\lambda_{j}^0,\lambda_{j}^1) \in \mathcal{P}^2$ and we put $\beta_j=|\bm \lambda_j|$. Consider the Levi subgroups  $\displaystyle L_0=\prod_{j=1}^r (\Gl_{(\beta_j)_0})_{d_j}$ and $\displaystyle L_1=\prod_{j=1}^r (\Gl_{(\beta_j)_1})_{d_j}$.

\vspace{6 pt}

 There exist reduced characters $\theta^0:L_0^F \to \C^*$ and $\theta^1:L_1^F \to \C^*$ such that $$\chi_0=R^G_{L_0}(\theta^0 R_{\lambda_{1}^0}\boxtimes \cdots \boxtimes R_{\lambda_{r}^0}) $$ and $$\chi_1=R^{\Gl_{\delta_1}}_{L_1}(\theta^1 R_{\lambda_{1}^1}\boxtimes \cdots \boxtimes R_{\lambda_{r}^1})  $$ and $\theta^0,\theta^1$ are associated to the same $r$-tuple $(\theta_1,\dots,\theta_r) \in \Hom(\F_{q^{d_1}}^*,\C^*) \times \cdots \times \Hom(\F^*_{q^{d_r}},\C^*)$, via 
the correspondence of \cref{connectedcentralizer1}. We denote by $\overline{\bm \lambda_0}, \overline{\bm \lambda_1} \in \mathcal{P}^r$ the multipartitions $\overline{\bm \lambda_0}=(\lambda_{1}^0,\dots,\lambda_{r}^0),$ $  \overline{\bm  \lambda_1}=(\lambda_{1}^1,\dots,\lambda_{r}^1)$.

To verify the dual Log compatibility of the family $\{m_{\alpha}\}_{\alpha \in \N^I}$, it is enough to check that we have:

\begin{equation}
\label{bebbo}
\langle m_{\delta},\chi \rangle H^{\vee}_{\omega}(q)=\prod_{j=1}^r C_{\lambda_{j}^0,\lambda_{j}^1}(q^{d_j})H^{\vee}_{(1,\bm \lambda_j)}(q^{d_j})
\end{equation}

Notice that, if $\delta \not\in (\N^I)^*$, there must exist $\beta_j$ such that $\beta_j \not\in (\N^I)^*$, i.e. such that $\lambda_j^0 < \lambda_j^1$. Eq.(\ref{bebbo}) therefore is true as both sides are equal to $0$.

\vspace{8 pt}

Assume then that $\delta \in (\N^I)^*$. From eq.(\ref{proof6}), there is an equality \begin{equation}
\label{bebbo1}
\langle m_{\delta},\chi \rangle=\dfrac{1}{q^{-(\delta,\delta)}}\sum_{\chi_2 \in \Gl_{\delta_2}(\F_q)^{\vee}}\langle \chi_0,R^G_M(\chi_1 \boxtimes \chi_2) \rangle \dfrac{\chi_2(1)}{|\Gl_{\delta_2}(\F_q)|}
\end{equation}
where $M=\Gl_{\delta_1} \times \Gl_{\delta_2} \subseteq G$. Let $\chi_2=\epsilon_{\Gl_{\delta_2}}\epsilon_{L_2}R^G_{L_2}(\theta^2 R_{\tilde{\phi_2}})$, with $L_2 \subseteq \Gl_{\delta_2}$ a Levi subgroup and $\theta^2:L_2^F \to \C^*$ a reduced character. From Lemma \ref{delignel1}, there is an equality $$R^G_M(\chi_1 \boxtimes \chi_2)=R^G_{L_1 \times L_2}((\theta^1 \times \theta^2)(R_{\overline{\bm \lambda_1}} \boxtimes R_{\Tilde{\phi}_2})) .$$

Let $L'$ be the connected centralizer of $\theta^1 \times \theta^2 :L_1 \times L_2 \to \C^*$. From Proposition \ref{irreducibleinduced}, we have that $\langle \chi_0,R^G_M(\chi_1 \boxtimes \chi_2) \rangle=0 $ if the semisimple part $(L',\theta^1 \times \theta^2)$ is not conjugated to $(L_0,\theta^0)$.

From the discussion made in  \cref{connectedcentralizer1}, we deduce that, if $(L',\theta^1 \times \theta^2)$ is conjugated to $(L^0,\theta^0)$, we must have that $|\lambda_{j}^0| \geq |\lambda_{j}^1|$, i.e. $\beta_j \in (\N^I)^*$, for each $j=1,\dots,r$.

If there exists $j\in \{1,\dots,r\}$ such that $\beta_j \not \in (\N^I)^*$, eq.(\ref{bebbo}) is thus verified as both sides are equal to $0$. 

\vspace{8 pt}

If $\beta_1,\dots,\beta_r \in (\N^I)^*$, from the remarks made in \ref{connectedcentralizer1}, we deduce that there exist a unique couple $(L_2,\theta^2)$, up to $\Gl_{\delta_2}(\F_q)$-conjugacy, such that $(L',\theta^1 \times \theta^2)$ is conjugated to $(L^0,\theta^0)$. In particular, we can take $L_2$ to be  $$L_2=\prod_{j=1}^r (\Gl_{(\beta_j)_2})_{d_j}$$ and $\theta^2:L_2^F \to \C^*$ the reduced character associated to the $r$-tuple $(\theta_1,\dots,\theta_r)$. We have therefore \begin{equation}
\label{bebbo2}
\langle m_{\delta},\chi \rangle=\sum_{\substack{\overline{\bm \lambda_2}=(\lambda_{1}^2,\dots,\lambda_{r}^2) \in \\ \mathcal{P}_{(\beta_1)_2} \times \cdots \times \mathcal{P}_{(\beta_r)_2}}}
\end{equation}

\begin{equation}
\dfrac{\langle R^G_{L_0}(\theta^0 R_{\overline{\bm \lambda_0}}),R^G_{L_1 \times L_2}((\theta^1 \times \theta^2 )R_{\overline{\bm \lambda_1}} \boxtimes R_{\overline{\bm \lambda_2}})\rangle}{q^{-(\delta,\delta)}}\dfrac{(-1)^{(\beta_1)_2+\cdots (\beta_r)_2}}{q^{\frac{\delta_2(\delta_2-1)}{2}-\sum_{j=1}^r d_j n(\lambda_{j}^2)}\prod_{j=1}^r H_{\lambda_{j}^2}(q^{d_j})}
\end{equation}

By Remark \ref{irreducibleinduced1}, we have that
\begin{equation}
\label{equazione1}
\langle R^G_{L_0}(\theta^0 R_{\overline{\bm \lambda_0}}),R^G_{L_1 \times L_2}((\theta^1 \times \theta^2 )R_{\overline{\bm \lambda_1}} \boxtimes R_{\overline{\bm \lambda_2}})\rangle=\langle R^G_{L_0}(\theta^0 R_{\overline{\bm \lambda_0}}),R^G_{L_0}(\theta^0 R^{L_0}_{L_1 \times L_2}(R_{\overline{\bm \lambda_1}} \boxtimes R_{\overline{\bm \lambda_2}}))\end{equation}

and by Lemma \ref{lemmamult} the RHS of eq.(\ref{equazione1}) is equal to 

$$\prod_{j=1}^r \langle R_{\lambda_{j}^0},R^{(\Gl_{(\beta_j)_0})_{d_j}}_{(\Gl_{(\beta_j)_1})_{d_j} \times (\Gl_{(\beta_j)_2})_{d_j}}(R_{\lambda_{j}^1} \boxtimes R_{\lambda_{j}^2}) \rangle_{(\Gl_{(\beta_j)_0})_{d_j}(\F_q)} .$$

By Remark \ref{rmkmult} and Lemma \ref{twist33}, for each $j=1,\dots,r $, we deduce that we  have an equality \begin{equation}
\label{bebbo3}
\langle R_{\lambda_{j}^0},R^{(\Gl_{(\beta_j)_0})_{d_j}}_{(\Gl_{(\beta_j)_1})_{d_j} \times (\Gl_{(\beta_j)_2})_{d_j}}(R_{\lambda_{j}^1} \boxtimes R_{\lambda_{j}^2}) \rangle_{(\Gl_{(\beta_j)_0})_{d_j}(\F_q)}=
\end{equation}
\begin{equation}
\label{bebbo3'}    
=\langle \chi_{\lambda_{j}^0},\Ind_{S_{(\beta_j)_1}}^{S_{(\beta_j)_0}}\times S_{(\beta_j)_2}(\chi_{\lambda_{j}^1}\boxtimes \chi_{\lambda_{j}^2}) \rangle_{S_{(\beta_j)_0}}=c^{\lambda_{j}^0}_{\lambda_{j}^1,\lambda_{j}^2}.
\end{equation}

From eq.(\ref{bebbo3'}), we deduce that we have \begin{equation}
\label{bebbo4}
\langle m_{\delta},\chi \rangle=\dfrac{1}{q^{-(\delta,\delta)+\frac{\delta_2(\delta_2-1)}{2}}}\prod_{j=1}^r \left(\sum_{\lambda_{j}^2\in \mathcal{P}_{(\beta_j)_2}}c^{\lambda_{j}^0}_{\lambda_{j}^1,\lambda_{j}^2} \dfrac{(-1)^{(\beta_j)_2}}{q^{-d_jn(\lambda_{j}^2)}H_{\lambda_{j}^2}(q^{d_j})}\right)
\end{equation}

From eq.(\ref{bebbo4}) above, we deduce that we have \begin{equation}
\label{bebbo5}
\dfrac{\langle m_{\delta},\chi \rangle}{\prod_{j=1}^rC_{\lambda_{j}^0,\lambda_{j}^1}(q^{d_j})}=\dfrac{\prod_{j=1}^r q^{d_j(|\lambda_{j}^0||\lambda_{j}^1|-|\lambda_{j}^0|^2-|\lambda_{j}^1|^2+\frac{|\lambda_{j}^2|^2-|\lambda_{j}^2|}{2})}}{q^{-(\delta,\delta)+\frac{\delta_2(\delta_2-1)}{2}}}
\end{equation}

From the fact that $\delta_2^2=\delta_0^2+\delta_1^2-2\delta_0\delta_1$ and, for each $j=1,\dots,r$, $$|\lambda_{j}^2|^2=|\lambda_{j}^0|^2+|\lambda_{j}^1|^2-2|\lambda_{j}^0||\lambda_{j}^1|$$ and Identity (\ref{duallog4}), we deduce that we have $$\dfrac{\prod_{j=1}^r q^{d_j(|\lambda_{j}^0||\lambda_{j}^1|-|\lambda_{j}^0|^2-|\lambda_{j}^1|^2+\frac{|\lambda_{j}^2|^2-|\lambda_{j}^2|}{2})}}{q^{-(\delta,\delta)+\frac{\delta_2(\delta_2-1)}{2}}}=\dfrac{q^{\sum_{j=1}^r d_j(-\frac{|\lambda_{j}^0|^2}{2}-\frac{|\lambda_{j}^1|^2}{2}) }}{q^{-\frac{\delta^2_0}{2}-\frac{\delta_1^2}{2}}}=\dfrac{H^{\vee}_{\omega}(q)}{\prod_{j=1}^r H^{\vee}_{(1,\bm \lambda_j)}(q^{d_j})} .$$

From the identity above and eq.(\ref{bebbo5}), we deduce therefore equality (\ref{bebbo}).

\end{proof}

We now show how Lemma \ref{firstep} implies Theorem \ref{charstack}.

\begin{proof}[Proof of Theorem 6.4.1]
\label{zioper}
We proceed by induction on the cardinality of $I$.

Let $|I|=1$. The quiver $Q$ has thus $1$ vertex and $g$ loops. Example \ref{Jordanquivermomentmap} shows that in this case, for each $n \in \N$, we have an equality $m_n=f^g_n$, where $f^g_n:\Gl_n(\F_q) \to \C$  is the  function defined as $$f^g_n(h)=\dfrac{\#\{(x_1,y_1,\dots,x_g,y_g) \in \Gl_n(\F_q)^{2g} \ | \ \prod_{i=1}^g [x_i,y_i]=h\}}{q^{(n^2(g-1))}}$$ introduced in Remark \ref{inductivestep}. It was thereby shown that $\{f^g_n\}_{n \in \N}$ is a dual Log compatible family.

\vspace{8 pt}
 Assume now to have shown Proposition \ref{charstack}  for all star-shaped quivers with $k$ legs and $m$ vertices and fix a star-shaped quiver $Q=(I,\Omega)$ with $|I|=m+1$. We can assume that $s_k >1$.
 
 Let $\Tilde{Q}=(\Tilde{I},\Tilde{\Omega})$ be the subquiver of $Q$ with set of vertices $\Tilde{I}=I-\{[k,s_k]\}$ and as set of arrows the arrows of $Q$ between elements of $\Tilde{I}$. 
 
 For a dimension vector $\alpha \in \N^I$, we denote by $\Tilde{\alpha}$ the element of $\N^{\Tilde{I}}$ obtained by the natural projection $\N^I \to \N^{\Tilde{I}}$ and we denote by $\pi_{\alpha}$ the natural projection $\pi_{\alpha}:\Gl_{\alpha}(\F_q) \to \Gl_{\Tilde{\alpha}}(\F_q)$.
 
 For $\alpha \in \N^I$, let $m_{\Tilde{\alpha}}$ be the function associated to the star-shaped quiver $\Tilde{Q}=(\Tilde{I},\Tilde{\Omega})$ and $\Tilde{\alpha}$ and denote by $n_{\alpha}:\Gl_{\alpha}(\F_q) \to \C$ the class function defined by 
 
 $$n_{\alpha}(h)=\begin{cases}\dfrac{m_{\Tilde{\alpha}}(\pi_{\alpha}(h))}{q^{-\alpha_{[k,s_k]}^2}} \text{ if } h_{[k,s_k]}=1,\\
 0 \text{ otherwise }
 \end{cases}$$
 
From eq.(\ref{identityelementgroup}), the function $n_{\alpha}$ can be  rewritten as:
 \begin{equation}
 \label{inductivestep1}
 n_{\alpha}(h)=\dfrac{m_{\Tilde{\alpha}}(\pi_{\alpha}(h))}{q^{-\alpha_{[k,s_k]}^2}}\sum_{\mathcal{X}\in \Gl_{\alpha_{[k,s_k]}}(\F_q)^{\vee}}\mathcal{X}(h_{[k,s_k]})\dfrac{\mathcal{X}(1)}{|\Gl_{\alpha_{[k,s_k]}}(\F_q)|}.
 \end{equation}

From Identity (\ref{inductivestep1}) and the dual Log compatibility of the functions $\{m_{\Tilde{\alpha}}\}$, we can deduce  that the family of functions $\{n_{\alpha}\}_{\alpha \in \N^I}$ is dual Log compatible. Indeed, for each $\chi \in \Gl_{\alpha}(\F_q)$, write $\chi=\widetilde{\chi} \boxtimes \chi_k$, with $\Tilde{\chi} \in \Gl_{\Tilde{\alpha}}(\F_q)^{\vee}$ and $\chi_k \in \Gl_{\alpha_{[k,s_k]}}^{\vee}(\F_q)$. We have \begin{equation}
\label{proofworks}
\langle n_{\alpha},\chi \rangle=\dfrac{1}{|\Gl_{\alpha}(\F_q)|}\sum_{h \in \Gl_{\alpha}(\F_q)}n_{\alpha}(h)\chi(h)=
\end{equation}

\begin{equation}
\label{proofworks2}
\left(\sum_{\Tilde{h} \in \Gl_{\Tilde{\alpha}}(\F_q)}\dfrac{m_{\Tilde{\alpha}}(\Tilde{h})\Tilde{\chi}(\Tilde{h})}{|\Gl_{\Tilde{\alpha}}(\F_q)|}\right)\left(\dfrac{q^{\alpha_{[k,s_k]}^2}}{|\Gl_{\alpha_{[k,s_k]}}(\F_q)|}\sum_{\substack{\mathcal{X}\in \Gl_{\alpha_{[k,s_k]}}(\F_q)^{\vee}\\h_k \in \Gl_{\alpha_{[k,s_k]}}}}\dfrac{\mathcal{X}(1)\mathcal{X}\chi_k(h_k)}{|\Gl_{\alpha_{[k,s_k]}}(\F_q)|}  \right)=
\end{equation}

\begin{equation}
\label{proofworks3}
=\langle m_{\Tilde{\alpha}},\Tilde{\chi} \rangle q^{\alpha^2_{[k,s_k]}} H^{\vee}_{\omega_{\chi_k}}(q).
\end{equation}

and therefore, if we put $\omega_{\chi}=\omega$, $\omega_{\Tilde{\chi}}=\Tilde{\omega}$ and $\omega_{\chi_k}=\omega_k$, we have \begin{equation}
\label{proofworks4}
\langle n_{\alpha},\chi \rangle H^{\vee}_{\omega}(q)=\langle m_{\Tilde{\alpha}},\Tilde{\chi} \rangle H^{\vee}_{\Tilde{\omega}}(q)q^{\alpha^2_{[k,s_k]}} (H^{\vee}_{\omega_{k}}(q))^2.
\end{equation}

Since the family $\{m_{\Tilde{\alpha}}\}_{\widetilde{\alpha} \in \N^{\widetilde{I}}}$ is dual Log compatible, from eq.(\ref{proofworks4}), we deduce that Identity (\ref{inductivestep1}) is equivalent to show that, for any $n \in \N$, any $\nu \in \mathbb{T}_n$ and any $d_1,\dots,d_r$ and types $\nu_1,\dots,\nu_r$ such that $\nu=\psi_{d_1}(\nu_1) \ast \cdots \ast \psi_{d_r}(\nu_r)$, we have

\begin{equation}
\label{proofwokrs5}
\dfrac{(H^{\vee}_{\nu}(t))^2}{\prod_{j=1}^r (H^{\vee}_{\nu_j}(t^{d_j}))^2}=\dfrac{\prod_{j=1}^r t^{d_j |\nu_j|^2}}{t^{n^2}}
\end{equation}
which is a direct consequence of eq.(\ref{valueirreduciblecharacter}). 

\vspace{6 pt}

Let now $\overline{I}=I-\{[k,s_k-1],[k,s_k]\}$ and, for $\alpha \in \N^I$, denote by $\overline{\alpha}$ the element of $\N^{\overline{I}}$ obtained by the natural projection $\N^I \to \N^{\overline{I}}$ and by $\overline{\pi}_{\alpha}:\Gl_{\alpha}(\F_q) \to \Gl_{\overline{\alpha}}(\F_q)$ the associated projection. For a couple $(\beta,\gamma) \in \N^2$, denote by $m_{(\beta,\gamma)}^{Kr}$ the class function associated to the Kronecker quiver and the dimension vector $(\beta,\gamma)$ for it, which was studied in Lemma \ref{firstep}. 

Consider then the function $k_{\alpha}:\Gl_{\alpha}(\F_q) \to \C$ defined by $$k_{\alpha}(h)=\begin{cases} \dfrac{m^{Kr}_{\alpha_{[k,s_k-1]},\alpha_{[k,s_k]}}(h_{[k,s_k-1]},h_{[k,s_k]})}{q^{-\sum_{i \in \overline{I}}\alpha_i^2}} \text{ if } \overline{\pi}_{\alpha}(h)=1 \\ 0 \text{ otherwise }\end{cases} .$$

As above, from eq.(\ref{identityelementgroup}), the function $k_{\alpha}$ can be rewritten as \begin{equation}
\label{inductivestep3}
k_{\alpha}(h)=\sum_{\chi \in \Gl^{\vee}_{\overline{\alpha}}(\F_q)}\chi(\overline{\pi}_{\alpha}(h))\dfrac{\chi(1)}{|\Gl_{\overline{\alpha}}(\F_q)|}\dfrac{m^{Kr}_{\alpha_{[k,s_k-1]},\alpha_{[k,s_k]}}(h_{[k,s_k-1]},h_{[k,s_k]})}{q^{-\sum_{i \in \overline{I}}\alpha_i^2}}. 
\end{equation}

By Identity (\ref{inductivestep3}) and Lemma \ref{firstep}, similarly to what we did to show eq.(\ref{inductivestep1}), we deduce that the family of functions $\{k_{\alpha}\}_{\alpha \in \N^I}$ is dual Log compatible. By Lemma \ref{duallog1}, the family of class functions  $$ \Biggl\{\dfrac{n_{\alpha} \ast k_{\alpha}}{q^{\sum_{i \in I}\alpha_i^2}}\Biggr\}_{\alpha \in \N^I}$$
is dual Log compatible too.

Lastly, a direct calculation shows that, for every $\alpha \in \N^I$, we have $$m_{\alpha}=\dfrac{n_{\alpha} \ast k_{\alpha}}{q^{\sum_{i \in I}\alpha_i^2}}. $$ Lemma \ref{duallog1} implies therefore the family $\{m_{\alpha}\}_{\alpha \in \N^I}$ is dual Log compatible.

\end{proof}

\subsection{Remarks}
Consider any quiver $Q=(I,\Omega)$. As done for a star-shaped quiver in paragraph \cref{defcharstack}, we can define the open subset $R(\overline{Q},\alpha)^{\circ} \subseteq R(\overline{Q},\alpha)$ and the multiplicative moment map $$\mu_{\alpha}^{\circ}:R(\overline{Q},\alpha)^{\circ} \to \Gl_{\alpha} $$ $$ (x_a,x_{a^*})_{a \in \Omega} \to \prod_{a \in \Omega}(1+x_ax_{a^*})(1+x_{a^*}x_a)^{-1} .$$

For an element $s \in (K^*)^I$, the multiplicative quiver stack $\mathcal{M}_{s,\alpha}$ associated to $s$ is defined as the quotient stack $$ \mathcal{M}_{s,\alpha} \coloneqq [(\mu_{\alpha}^{\circ})^{-1}(s)/\Gl_{\alpha}] $$

Consider the case in which $K=\F_q$. We denote by $c_{\alpha}:\Gl_{\alpha}(\F_q) \to \C$ the class function given by $$c_{\alpha}(g)\coloneqq \dfrac{|(\mu^{\circ}_{\alpha})^{-1}(g)^F|}{q^{-(\alpha,\alpha)}} .$$ 

\vspace{8 pt}

We expect that the following result is true:

\begin{conjecture}
\label{mainteomultvar}
The family of class functions $\{c_{\alpha}\}_{\alpha \in \N^I}$ is dual Log compatible.
\end{conjecture}

Similarly to how we proved  Theorem \ref{charstack}, to prove Conjecture \ref{mainteomultvar} above, it is enough to show it in the cases where $Q$ has a single vertex or $Q$ is the Kronecker quiver.

It would be natural to expect that the proofs in these cases are a natural generalization of the proofs of  Remark \ref{inductivestep} and Lemma \ref{firstep}. 

However, in both cases we  used in a key way the fact that we considered representations of $\overline{Q}$ which were injective along the arrows of $\Omega$. Without the injectivity hypothesis, the combinatorics involved becomes more complicated. 

\vspace{8 pt}

\begin{esempio}

Let $Q$ be the Jordan quiver and $\alpha=n \in \N$. The function $c_n:\Gl_n(\F_q) \to \C$ is given by $$c_n(g)=\# \{e,e^* \in \mathfrak{gl}_n(\F_q) \ | \ (1+e e^*),(1+e^*e) \in \Gl_n(\F_q) \text{ and } (1+ee^*)(1+e^*e)^{-1}=g\} .$$

Notice that there is no straightforward way to adapt the strategy of proof of Example \ref{conv3} to compute the quantity $\langle c_n,\chi \rangle$ for an irreducible character $\chi \in \Gl_n(\F_q)^{\vee}.$
\end{esempio}

\section{Cohomology of character stacks}
\label{cohocharstacks}
\subsection{Mixed Poincaré series of Artin stacks}
\label{finitefiledsandcoho}
Let $\mathcal{X}$ be an algebraic stack of finite type over an algebraically closed field $K$. Compactly supported cohomology for algebraic stacks have been fully developed in \cite{laszlo}, where the authors develop a full theory of six-functors formalism for stacks.

In the following, we will always assume that $K=\C$ and  $\mathcal{X}$ is a quotient stack $[X/G]$, where $X$ is a $G$-quasi projective variety and $G$ is a linear algebraic group.

For such a stack $[X/G]$, compactly supported cohomology can be defined in a simpler way, which is an algebro-geometric version of the Borel construction of  equivariant cohomology in differential geometry.

 Historically, this was initially taken as a definition of $G$-equivariant compactly supported cohomology or $G$-equivariant Borel-Moore homology of $X$. We briefly review this construction, for more details see \cite{eddin} where the authors define in a similar way the equivariant Chow groups or \cite[Section 2.5]{
 davisonequivariant}.

\vspace{6 pt}

Consider an embedding $G \subseteq \Gl_r(\C)$. For any $m \in \N$, denote by $V_m=\Hom(\C^m,\C^r)$. Notice that $G$ acts on the right on the vector space $V_m$ and acts freely on the open dense subset $U_m=\Hom^{surj}(\C^m,\C^r)$, given by surjective homomorphisms.

Consider the left action of $G$ on $X \times V_m$ defined as $g \cdot (x,u)=(g \cdot x, u \cdot g^{-1})$. It is a known fact that the action of $G$ on $U_m$ is schematically free and the quotient stack $[X \times U_m/G]$ is thus a scheme, which is usually denoted by $X \times_G U_m$, see for instance \cite[Proposition 23]{eddin}.

Let $Z_m=V_m\setminus U_m$. The codimension $\codim_{V_m}(Z_m)$ goes to $\infty$ for $m \to + \infty$. We put then $$H^i_c([X/G],\C)\coloneqq H^{i+2\dim (V_m)}_c(X \times_G U_m,\C)$$ for $m$ sufficiently big. This definition does not depend on $m$ if $m$ is sufficiently big, see for example \cite[Proposition 2.16]{davisonequivariant}.

We denote by $H^*_c([X/G])\coloneqq H^*_c([X/G],\C)$. We can endow these latter cohomology groups with the weight filtration $W^*_{\bullet}H^*_c([X/G])$, defined as follows $$W^i_jH^i_c([X /G])=W^i_{j+\dim (V_m)}H^{i+2\dim (V_m)}_c(X \times_G U_m) $$ for $m$ sufficiently big. In this case too, \cite[Proposition 2.16]{davisonequivariant} shows that this definition does not depend on $m$ if $m$ is sufficiently big.

\vspace{4 pt}

For $\mathcal{X}$ over $\C$, we define the mixed-Poincar\'e series $H_c(\mathcal{X},q,t)$ as  \begin{equation}
    \label{poinc}H_c(\mathcal{X};q,t)\coloneqq \sum_{k,m}\dim(W^k_m/W^k_{m-1})q^{\frac{m}{2}}t^k.
\end{equation}
Notice that the specialization  $H_c(\mathcal{X},1,t)$ of $H_c(\mathcal{X},q,t)$ at $q=1$ is equal to  the Poincar\'e series $P_c(\mathcal{X},t)$ of the stack $\mathcal{X}$. When $\displaystyle \sum_{k}(-1)^k \dim(W^k_m/W^k_{m-1})$ is finite for each $m$, we define the E-series: \begin{equation}
    \label{epol}
    E(\mathcal{X},q) \coloneqq H_c(\mathcal{X};q,-1)=\sum_{m,k}\dim(W^k_m/W^k_{m-1})(-1)^kq^{\frac{m}{2}}.
\end{equation}

We have the following Theorem, for a proof see \cite[Theorem 2.5]
{NonOr}.

\begin{teorema}
For a quotient stack $\mathcal{X}=[X/G]$ where $G$ is a connected linear algebraic group and $X$ an affine variety, the E-series $E(\mathcal{X},q)$ is well defined and $E(\mathcal{X},q)=E(X,q)E(BG,q)$ where $BG$ is the classifying stack of $G$, .

\end{teorema}
\vspace{8 pt}

 Consider now a complex stack $\mathcal{X}$ such that its E-series is well defined. Let now $E$ be a  $\Z$-finitely generated algebra and $\mathcal{Y}$ be an $E$-stack. Assume that there exists $\psi:E \to \C$ such that $\mathcal{Y} \times_{\spec(R),\psi} \spec(\C) \cong \mathcal{X}$. The stack $\mathcal{Y} $ is called a \textit{spreading out} of $\mathcal{X}$.

For any $\phi: E \to \F_q$, denote by $\mathcal{X}^{\phi}\coloneqq\mathcal{Y}\times_{\spec(R),\phi} \spec(\F_q)$. We say that the stack $\mathcal{X}$ is (strongly) \textit{rational count} if  there exists an open $U \subseteq \spec(E)$ and a rational function $Q(t)$ such that for any $\phi: E \to \F_q$ with $\phi(\spec(\F_q)) \subseteq U$, it holds
$$\# \mathcal{X}^{\phi}(\F_{q^n})=Q(q^n) $$ for every $n \geq 1$.

In \cite[Theorem 2.9]{NonOr} it is shown the following result.

\begin{teorema}
\label{Eseriesquotientstacks}
For a (strongly) rational count quotient stack $\mathcal{X}$, there is an identity:
\begin{equation}
\label{Epolynomialidentity}
E(\mathcal{X},q)=Q(q) .
\end{equation}
\end{teorema}
\vspace{10 pt}

\begin{oss}
\label{countingpointsfinitefields}
Let $\mathcal{X}=[X/G]$ be a complex quotient stack with $G$ a connected linear algebraic group. Let $E$ be a $\Z$-subalgebra as above and $Y_1,Y_2$ be $E$-schemes which are spreading out of $X,G$ respectively.
The stack $\mathcal{Y} \coloneqq [Y_1/Y_2]$ is a spreading out of $\mathcal{X}$. 

Notice that for any homomorphism $\phi: E \to \F_q$, there is an isomorphism $\mathcal{X}^{\phi}=[X^{\phi}/G^{\phi}]$ and by \cite[Lemma 2.5.1]{B} there is therefore an equality 
\begin{equation}
\label{countingoverfinitefields}
\#\mathcal{X}^{\phi}(\F_{q^n})=\dfrac{\#X^{\phi}(\F_{q^n})}{\#G^{\phi}(\F_{q^n})} .
\end{equation}

In particular, the stack $\mathcal{X}$ is rational count if and only the variety $X$ is polynomial count, i.e. there exists $S(t) \in \Q[t]$ such that $S(q^n)=\#X^{\phi}(\F_{q^n})$ for $\phi$ having image in an appropriate open of $\spec(E)$. This was already remarked in \cite[Remark 2.7]{NonOr}.

\end{oss}

Consider a reductive algebraic group $G$. In \cite{HodgeIII} it is shown that each cohomology group $H^m_c(BG)$ is pure of weight $m$. In \cite{HodgeIII} this is stated for cohomology rather than cohomology with compact support. The latter case is an immediate consequence thanks to Poincaré duality.
 
From Theorem \ref{Eseriesquotientstacks}, we deduce the following Lemma:

\begin{lemma}
Suppose  that $G$ is (strongly) polynomial count. The classifying stack $BG$ is strongly polynomial count and we have \begin{equation}
    H_c(BG,q,t)=E(BG,qt^2)=\dfrac{1}{E(G,qt^2)}.
\end{equation}

\end{lemma}

\begin{esempio}
We deduce, for instance, that, for each $n \in \N$, we have

\begin{equation}
\label{intuile}
H_c(B\Gl_n,q,t)=\dfrac{1}{(qt^2)^{\frac{n(n-1)}{2}}}\dfrac{1}{(qt^2-1)\cdots ((qt^2)^n-1)}    
\end{equation}

and 

\begin{equation}
\label{intuile1}
H_c(B\PGl_n,q,t)=\dfrac{1}{(qt^2)^{\frac{n(n-1)}{2}}}\dfrac{1}{((qt^2)^2-1)\cdots ((qt^2)^n-1)}    
\end{equation}

\end{esempio}

\vspace{6 pt}

Lastly, we will need the following Proposition about the cohomology of a quotient stack $[X/G]$. Assume that $G=\Gl_n$ and the center $\mathbb{G}_m \subseteq \Gl_n$ acts trivially on the scheme $X$. There is thus an induced action of $\PGl_n$ on $X$ and a canonical morphism $h:[X/\Gl_n] \to [X/\PGl_n]$.

\begin{prop}
\label{tensorprodisomgerbe}
Let $X$ be a $\C$-variety with a $\Gl_n$-action trivial on the center. We have  \begin{equation}
\label{cohomologygerbe}
H_c([X/\Gl_n],q,t)=\dfrac{H_c([X/\PGl_n],q,t)}{qt^2-1}
\end{equation}
\end{prop}

\begin{proof}

We start by the case in which $X=\spec(\C)$ and we look at the canonical morphism  $\pi:B\Gl_n \to B\PGl_n$. In this case, eq.(\ref{cohomologygerbe}) is a direct consequence of eq.(\ref{intuile}), eq.(\ref{intuile1}). There is a cartesian diagram:
\begin{center}
\begin{tikzcd}
B\mathbb{G}_m \arrow[r,""] \arrow[d,""] &B\Gl_n \arrow[d,"\pi"]\\
\spec(\C) \arrow[r,"\psi"] &B\PGl_n
\end{tikzcd}
\end{center}
where $\psi:\spec(\C) \to B\PGl_n$ is the canonical projection. Since $\psi$ is a smooth covering, for each $q \in \Z$, the sheaf $R^q\pi_!\C$ is a  local system with fiber $H^q_c(B\mathbb{G}_m)$. 

Moreover, as $\PGl_n$ is connected, each local system is trivial over $B\PGl_n$, see for example \cite[Proposition 6.13]{achar}. In particular, the Leray spectral sequence for compactly supported cohomology and the morphism $\pi$ in second page is $$E^{p,q}_2 :H^p_c(B\PGl_n) \otimes H^q_c(B \mathbb{G}_m) \Rightarrow H^{p+q}_c(B \Gl_n) .$$

From eq.(\ref{cohomologygerbe}), we deduce that the spectral sequence collapses at page $2$, i.e. that the canonical morphism $$H^p_c(B\Gl_n) \to H^p_c(B\mathbb{G}_m) $$ is surjective for every $p$.

\vspace{6 pt}

 Consider now a general $X$. In this case too, we have a Leray spectral sequence for compactly supported cohomology with second page $$E^{p,q}_2=H^p_c\left([X/\PGl_n],R^qh_{!}\C\right) \Rightarrow H^{p+q}_c([X/\Gl_n]) .$$

Notice that there is $2$-Cartesian diagram $$
\xymatrix{[X/\Gl_n]\ar[r]\ar[d]&B\Gl_n\ar[d]\\
[X/\PGl_n]\ar[r]& B\PGl_n}
 $$
where the morphism $\pi:B\Gl_n \to B\PGl_n$ on the right is the canonical morphism between classifying spaces. In particular, we have  $$R^qh_!\C \cong r^*R^q\pi_{!}\C $$ where we put $r:[X/\PGl_n] \to B\PGl_n$.  

We deduce thus that each $R^qh_!\C$ is trivial. Moreover, the associated map $$H^p_c([X/\Gl_n]) \to H^p_c(B\mathbb{G}_m) $$ is surjective, since the map $H^p_c(B\Gl_n)\to H^p_c(B\mathbb{G}_m)$ is surjective, as remarked above. Therefore, the spectral sequence $E^{p,q}_2$ collapses at page $2$ and we obtain an isomorphism $$H_c^*\left([X/\Gl_n]\right)=H_c^*\left([X/\PGl_n]\right) \otimes H^*_c(B \mathbb{G}_m) .$$

\end{proof}

\subsection{E-series of multiplicative quiver stacks and character stacks}
\label{section8.2}

In this paragraph, we apply Theorem \ref{charstack} to the computation of E-series of multiplicative quiver stacks and character stacks.

We first recall the results of Hausel, Letellier, Rodriguez-Villegas about the E-series $E(\mathcal{M}_{\mathcal{C}},q)$ in the case where $\mathcal{C}$ is a \textit{generic} $k$-tuple of semisimple conjugacy classes.

\subsubsection{Cohomological results for generic $k$-tuples}
\label{genericktuplesepoly}
Let $\mathcal{C}$ be a $k$-tuple of semisimple conjugacy classes of $\Gl_n(\C)$. Let $Q=(I,\Omega)$ be the associated quiver and $\gamma_{\mathcal{C}} \in (\C^*)^I, \alpha_{\mathcal{C}} \in (\N^I) ^*$ the associated parameters, introduced in \cref{chaptercharstack}.

\vspace{10 pt}

In \cite[Definition 2.1.1]{HA}, the authors give the following definition of a generic $k$-tuple $\mathcal{C}$.

\begin{definizione}
A $k$-tuple $\mathcal{C}$ of semisimple conjugacy classes is said generic if, given a subspace $W$ of  $\C^n$ which is stabilized by some $X_i \in \mathcal{C}_i$, for each $i=1,\dots,k$,  such that $\displaystyle\prod_{i=1}^k \det(X_i|_W)=1  $ then either $W=\{0\}$ or $W=\C^n$. 

\end{definizione}
\vspace{8 pt}

\begin{oss}
\label{genericexists}
For any $\beta \in (\N^I)^*$, there exists a generic $k$-tuple $\mathcal{C}'$ such that $\alpha_{\mathcal{C}'}=\beta$, see for example \cite[Lemma 2.1.2]{HA}. 

\end{oss}

\vspace{6 pt}

\begin{lemma}
\label{criteriongenericityorbits}
If $\mathcal{H}_{\gamma_{\mathcal{C}},\alpha_{\mathcal{C}}}^*=\{\alpha_{\mathcal{C}}\}$ the $k$-tuple $\mathcal{C}$ is generic. On the other side, if $\mathcal{C}$ is generic, there are no $\delta, \epsilon \in \mathcal{H}_{\gamma_{\mathcal{C}},\alpha_{\mathcal{C}}}^*\setminus \{\alpha_{\mathcal{C}}\}$ such that $\delta+\epsilon=\alpha_{\mathcal{C}}$.
\end{lemma}

\begin{proof}
Suppose that for a $k$-tuple $\mathcal{C}$ there exists a proper subspace $ 0 \subset W \subset \C^n$ and $X_1 \in \mathcal{C}_1,\dots,X_k \in \mathcal{C}_k$ such that $X_i (W) \subseteq W$ for each $i$ and $$\prod_{i=1}^k \det(X_i|_W)=1 .$$

For $i=1\dots,k$ and $j=0,\dots,s_i$, put $V_{\gamma_{i,j}}=\Ker(X_i-\gamma_{i,j}I_n)$ and $W_{\gamma_{i,j}}=W \cap V_{\gamma_{i,j}}$. Notice that, for each $i$, we have $$W=\bigoplus_{j=0}^{s_i} W_{\gamma_{i,j}} $$ and $$\det(X_i|_W)=\prod_{j=0}^{s_i} \gamma_{i,j}^{\dim W_{\gamma_{i,j}}} .$$

Consider now the dimension vector $\beta \in (\N^I)^*$ defined as $$\beta_{[i,j]}=\sum_{h=j}^{s_i} \dim(W_{\gamma_{i,j}}) .$$ We have that $\beta < \alpha_{\mathcal{C}}$. Moreover, we have \begin{equation}
\label{genericequivalence}
\gamma_{\mathcal{C}}^{\beta}=\prod_{i=1}^k \gamma_{i,0}^{-\beta_0}\prod_{j=1}^{s_i} (\gamma_{i,j}^{-1}\gamma_{i,j-1})^{\beta_{[i,j]}}=\prod_{i=1}^k \gamma_{i,0}^{-\sum_{h=0}^{s_i}\dim W_{\gamma_{i,j}} } \prod_{j=1}^{s_i} \gamma_{i,j}^{-\sum_{h=j}^{s_i}\dim W_{\gamma_{i,j}}}\gamma_{i,j-1}^{\sum_{h=j}^{s_i}\dim W_{\gamma_{i,j}}}=
\end{equation}

\begin{equation}
\label{genericequivalence2}
=\prod_{i=1}^k \prod_{j=0}^{s_i}\gamma_{i,j}^{-\dim W_{\gamma_{i,j}}}=\prod_{i=1}^k \det(X_i|_W)^{-1}=1.
\end{equation}

Conversely, suppose that there exists $\beta \in \mathcal{H}_{\gamma_{\mathcal{C}},\alpha_{\mathcal{C}}}^*\setminus\{\alpha_{\mathcal{C}}\}$ such that $\epsilon\coloneqq \alpha_{\mathcal{C}}-\beta$ belongs to $\mathcal{H}_{\gamma_{\mathcal{C}},\alpha_{\mathcal{C}}}^*\setminus\{\alpha_{\mathcal{C}}\}$ too. 

Since $\epsilon \in (\N^I)^*$, for each $j,h$, we have $$\beta_{[h,j]}-\beta_{[h,j=1]} \leq (\alpha_{\mathcal{C}})_{[h,j]}-(\alpha_{\mathcal{C}})_{[h,j+1]}=m_{h,j} ,$$ where $m_{h,j}$ is the multiplicity of the eigenvalue $\gamma_{h,j}$ in the orbit $\mathcal{C}_j$.

Put $m=\beta_0$ and let $W=\C^m \subseteq \C^n$ be the span of the first $m$ vectors of the canonical basis. We have that $m < n$, since $\epsilon \in (\N^I)^*$.

For each $i=1,\dots,k$, there exists  a diagonal matrix $X_i \in \mathcal{C}_i$ such that its first $m$ diagonal entries are given by $\beta_{[i,s_i]}$ times the element $\gamma_{i,s_i}$, then $\beta_{[i,s_i-1]}-\beta_{[i,s_i]}$ times the element $\gamma_{i,s_i-1}$ and so on.

The subspace $W$ is $X_i$-stable for each $i=1,\dots,k$ and, moreover, $$\det(X_i|_W)=\gamma_{\mathcal{C}}^{\beta}=1 ,$$ from which we deduce that $\mathcal{C}$ is not generic.

\end{proof}

\vspace{10 pt}
For generic $k$-tuples, we have the following general combinatorial formula computing the E-series of the associated character stacks, shown by Hausel, Letellier, Rodriguez-Villegas.

\vspace{8 pt}

 Let $\mathbf{x_1}=\{x_{1,1},x_{1,2} \dots \},\dots, \mathbf{x_k}=\{x_{k,1},\dots\}$ be $k$ sets of infinitely many variables and let us denote by $\Lambda_k\coloneqq\Lambda(\mathbf{x_1},\dots ,\mathbf{x_k})$ the ring of functions over $\Q(z,w)$ separetely symmetric in each set of variables.

On $\Lambda_k$ there is a natural $\lambda$-ring structure, induced by the operations $\psi_d:\Lambda_k \to \Lambda_k$ defined as $$\psi_d(f(\mathbf{x_1},\dots,\mathbf{x_k}))=f(\mathbf{x_1}^d,\dots,\mathbf{x_k}^d) $$
 
 On $\Lambda_k$ there is a natural bilinear form obtained by  extending by linearity $$\left < f_1(\mathbf{x_1})\cdots f_k(\mathbf{x_k}),g_1(\mathbf{x_1})\cdots g_k(\mathbf{x_k})\right>=\prod_{i=1}^k \left<f_i,g_i\right> $$ where $\left<,\right>$ is the  bilinear form on the ring of symmetric functions making the Schur functions $s_{\mu}$ an orthonormal basis. For a multipartition $\bm \mu=(\mu^1,\dots ,\mu^k) \in \mathcal{P}^k$ we denote by $h_{\bm{\mu}}=h_{\mu^1}(\mathbf{x_1})\cdots h_{\mu^k}(\mathbf{x_k})$ the associated complete symmetric function.

For any $\lambda \in \mathcal{P}$, let $\mathcal{H}_{\lambda}(z,w)$ be the hook function: 
\begin{equation}
 \label{defH}
 \mathcal{H}_{\lambda}(z,w)=\prod_{s \in \lambda} \dfrac{(z^{2a(s)+1}-w^{2l(s)+1})^{2g}}{(z^{2a(s)+2}-w^{2l(s)})(z^{2a(s)}-w^{2l(s)+2})}
\end{equation}

and the associated series $\Omega(z,w) \in \Lambda_k[[T]]$ \begin{equation}
    \label{omega}
    \Omega(z,w)=\sum_{\lambda \in \mathcal{P}}\mathcal{H}_{\lambda}(z,w) \prod_{i=1}^k H_{\lambda}(\mathbf{x_i},z^2,w^2)T^{|\lambda|}
\end{equation}
where $H_{\lambda}(\mathbf{x_i},q,t)$ are the (modified) Macdonald symmetric polynomials (for a definition see \cite[I.11]{garsia-haiman}).

 For any $\bm \mu \in \mathcal{P}_n^k$, in \cite{HA} it is defined the rational function $\mathbb{H}_{\bm \mu}(z,w)$ as follows:
\begin{equation}
    \label{formula}
    \mathbb{H}_{\bm \mu}(z,w)\coloneqq(z^2-1)(1-w^2)\left<\Coeff_{T^n}(\Plelog(\Omega(z,w))),h_{\bm \mu}\right>.
\end{equation}

\vspace{8 pt}

For any $\beta \in (\N^I)^*$ and for any $j=1,\dots,k$, the integers $(\beta_{[j,0]}-\beta_{[j,1]},\dots,\beta_{[j,s_j-1]}-\beta_{[j,s_j]},\beta_{[j,s_j]})$ up to reordering form a partition $\mu_{\beta}^j \in \mathcal{P}$. Denote by $\bm \mu_{\beta} \in \mathcal{P}^k$ the multipartition $\bm \mu_{\beta}=( \mu_{\beta}^1,\dots, \mu_{\beta}^k)$ and by $\mathbb{H}_{\beta}(z,w)$ the function $\mathbb{H}_{\bm \mu_{\beta}}(z,w)$. 

\vspace{8 pt}

 Hausel, Letellier, Rodriguez-Villegas \cite[Theorem 5.2.3]{HA} showed the following result:

\begin{teorema}
\label{genericresultcharvarieties}
For any generic $k$-tuple $\mathcal{C}$, we have:
\begin{equation}
\label{genericresultcharvarieties1}
\dfrac{E(\mathcal{M}_{\mathcal{C}},q)}{q^{-(\alpha_{\mathcal{C}},\alpha_{\mathcal{C}})}}=\dfrac{q\mathbb{H}_{\alpha_{\mathcal{C}}}\left(\sqrt{q},\dfrac{1}{\sqrt{q}}\right)}{q-1}.
\end{equation}

\end{teorema}

\begin{oss}
The result of \cite[Theorem 5.2.3]{HA} is stated in a slightly different way. In particular, to verify the equivalence of the results of \cite{HA} and Theorem \ref{genericresultcharvarieties}, it is needed to verify that $-2(\alpha_{\mathcal{C}},\alpha_{\mathcal{C}})+1=\dim(\mathcal{M}_{\mathcal{C}})$. The proof of the latter equality can be found at the beginning of \cite[Chapter 5.2]{AH}.

\end{oss}

In the same paper, the author \cite[Conjecture 1.2.1]{HA} proposed the following conjectural identity for the mixed Poincaré series of the character stack $\mathcal{M}_{\mathcal{C}}$, when $\mathcal{C}$ is generic, naturally deforming eq.(\ref{genericresultcharvarieties1}):

\begin{conjecture}
\label{conjmhs}
For any generic $k$-tuple $\mathcal{C}$ of semisimple conjugacy classes, we have
\begin{equation}
\dfrac{H_c(\mathcal{M}_{\mathcal{C}},q,t)}{(qt^2)^{-(\alpha_{\mathcal{C}},\alpha_{\mathcal{C}})}}=\dfrac{(qt^2)\mathbb{H}_{\alpha_{\mathcal{C}}}\left(-t\sqrt{q},\dfrac{1}{\sqrt{q}}\right)}{qt^2-1}.
\end{equation}
\end{conjecture}

\begin{oss}
Theorem \ref{genericresultcharvarieties} and Conjecture \ref{conjmhs} in the article \cite{HA} are stated for the corresponding generic character variety $M_{\mathcal{C}}$, rather than the character stack $\mathcal{M}_{\mathcal{C}}$. The equivalences of the statements of \cite{HA} and those presented here comes from Proposition \ref{tensorprodisomgerbe}.
\end{oss}

\subsubsection{Main result}
 \label{mainresultparagraph}
Consider a star-shaped $Q=(I,\Omega)$. For any $\sigma \in (\C^*)^I$ and any $\beta \in (\N^I)$,
we will construct a spreading out of the stack $\mathcal{M}^*_{\sigma,\beta}$ in the following way. Let $E_0=\Z[x_i,x_i^{-1}]_{i \in I}$ be the ring in $|I|$ invertible variables. For any $\delta \in \N^I$, denote by $x^{\delta} \in E_0$ the element $\displaystyle x^{\delta}\coloneqq \prod_{i \in I}x_i^{\delta_i}$.

Let $\mathcal{N}_{\sigma,\beta}=(\N^I_{\leq \beta})^*\setminus\mathcal{H}_{\sigma,\beta}.$ Consider the multiplicative set $S \subseteq E_0$ generated by the elements $x^{\delta}-1$ for $\delta \in \mathcal{N}_{\sigma,\beta}$.
Denote by $J \subseteq S^{-1}E_0$ the ideal generated by $(x^{\delta}-1)$ for $\delta \in \mathcal{H}_{\sigma,\beta}$ and let $E$ be the quotient  $$E\coloneqq S^{-1}E_0/J. $$

Notice that, given a field $K$, a map $\phi:E \to K$ corresponds to an element $\gamma_{\phi} \in (K^*)^I$ such that $\mathcal{H}_{\gamma_{\phi},\beta}=\mathcal{H}_{\sigma,\beta}$. 

\vspace{8 pt}

Let $\mathcal{A}_0$ be the polynomial $E$-algebra in $\displaystyle 2 \sum_{a \in \Omega}s(a)t(a)$ variables corresponding to the entries of matrices $(x_a,x_{a^*})_{a \in \Omega}$. Let $\mathcal{W} \subseteq \mathcal{A}_0$ be the multiplicative system generated by $\det(1+x_ax_{a^*}),\det(1+x_{a^*}x_a)$ for $a \in \Omega$ and let $\mathcal{A}_0'\coloneqq \mathcal{W}^{-1}\mathcal{A}_0$.

Consider the ideal $\mathcal{I} \subseteq \mathcal{A}'_0$ generated by the entries of $$\prod_{a \in \Omega} (1+x_ax_{a^*})(1+x_{a^*}x_a)^{-1}-\prod_{i \in I}(x_iI_{\alpha_i}) $$ and let $$\mathcal{A}=\mathcal{A}'_0/\mathcal{I} .$$

Let $Y=\spec(A)$ and let $Y^* \subseteq Y$ be the open subset given by $y \in Y$ such that for any algebraically closed field $K$ and any morphism $\spec(K) \to Y$ with image $y$, corresponding to an element $(x_a,x_{a^*})_{a \in \Omega} \in R(\overline{Q},\alpha,K)$, the maps $(x_a)_{a \in \Omega}$ are injective.

Let now $\psi:E \to \C$ be the map induced by the element $\sigma \in (\C^*)^I$. Notice that $$Y^* \times_{\spec(E),\psi} \spec(\C) \cong (\Phi^{\ast}_{\beta})^{-1}(\sigma) $$ and therefore $Y^*$ is a spreading out of $(\Phi^{\ast}_{\beta})^{-1}(\sigma)$. Similarly, for any $\phi: E \to \F_q$ corresponding to an element $\gamma_{\phi} \in (\F_q^*)^I$ with $\mathcal{H}_{\gamma_{\phi},\beta}=\mathcal{H}_{\sigma,\beta}$, we have $$((\Phi^{\ast}_{\beta})^{-1}(\sigma))^{\phi}=(\Phi^{\ast}_{\beta})^{-1}(\gamma_{\phi}) .$$

Let $\Gl_{\alpha,E}$ be the $E$-group scheme $\prod_{i \in I}\Gl_{\alpha_i,E}$. The stack $\mathcal{Y}^*=[Y^*/\Gl_{\alpha,E}]$ is therefore a spreading out of $\mathcal{M}^*_{\sigma,\beta}$.

By Remark \ref{countingpointsfinitefields} and the results of Theorem \ref{mainteochar} and Theorem \ref{charstack}, we deduce that the stack $\mathcal{M}^*_{\sigma,\beta}$ is rational count and we have

\begin{equation}
\label{identitymultquivestack}
\dfrac{E(\mathcal{M}^*_{\sigma,\beta},q)}{q^{-(\beta,\beta)}}=\Coeff_{\beta}\left(\Plexp\left(\sum_{\delta \in \mathcal{H}_{\sigma}}\widetilde{M}_{\delta,gen}(q)y^{\beta}\right)\right)
\end{equation}

where $\widetilde{M}_{\delta,gen}(t)$ are the rational functions associated to the dual Log compatible family $\{m_{\delta}\}_{\delta \in \N^I}$, as in \cref{duallogcompdefin}. Notice that $\widetilde{M}_{\delta,gen}(t)=0$ if $\delta \notin (\N^I)^*$.

 From Remark \ref{genericexists} and Lemma \ref{criteriongenericityorbits}, we deduce that for any $\delta \in (\N^I)^*$ we have $$\widetilde{M}_{\delta,gen}(q)=\dfrac{q\mathbb{H}_{\delta}\left(\sqrt{q},\frac{1}{\sqrt{q}}\right)}{q-1} .$$

We can resume all the arguments above in the following Theorem:

\begin{teorema}
\label{theomultquiverstack}
For any $\beta \in (\N^I)^*$ and any $\sigma \in (\C^*)^I$, we have:
\begin{equation}
\dfrac{E(\mathcal{M}^*_{\sigma,\beta},q)}{q^{-(\beta,\beta)}}=\Coeff_{\beta}\left(\Plexp\left(\sum_{\delta \in \mathcal{H}_{\sigma,\beta}^*} \dfrac{q\mathbb{H}_{\delta}\left(\sqrt{q},\frac{1}{\sqrt{q}}\right)}{q-1}y^{\delta}\right)\right)
\end{equation}
\end{teorema}

\subsubsection{E-series for character stacks with semisimple monodromies}
\label{characterstacksparagraph}

 From Theorem \ref{theomultquiverstack} and the isomorphism of Theorem \ref{isomstacks}, we deduce the following Theorem about E-series for character stacks associated to $k$-tuples of semisimple conjugacy classes.

\begin{teorema}
\label{Epolynomialtheorem}
For any $k$-tuple $\mathcal{C}$ of semisimple conjugacy classes of $\Gl_n(\C)$, we have:
\begin{equation}
\label{Epolynomialtheorem1}
\dfrac{E(\mathcal{M}_{\mathcal{C}},q)}{q^{-(\alpha_{\mathcal{C}},\alpha_{\mathcal{C}})}}=\Coeff_{\alpha_{\mathcal{C}}}\left(\Plexp\left(\sum_{\substack{\beta \in \mathcal{H}^*_{\gamma_{\mathcal{C}},\alpha_{\mathcal{C}}}}} \dfrac{q\mathbb{H}_{\beta}\left(\sqrt{q},\frac{1}{\sqrt{q}}\right)}{q-1}y^{\beta}\right)\right).
\end{equation}

\end{teorema}

\vspace{8 pt}

\begin{oss}
Notice that Theorem \ref{Epolynomialtheorem1} implies that the E-series $E(\mathcal{M}_{\mathcal{C}},q)$ does not depend on the values on the eigenvalues $\{\gamma_{j,h}\}_{\substack{j=1,\dots,k \\ h=0,\dots,s_j}}$
but only on the subset $\mathcal{H}^*_{\gamma_{\mathcal{C}},\alpha_{\mathcal{C}}}$.

\end{oss}

\section{Mixed Poincaré series of character stacks for $\mathbb{P}^1_{\C}$ with four punctures}
\label{chaptermhpnongene}

From Theorem \ref{Epolynomialtheorem}, it seems natural to formulate  the following generalization of Conjecture \ref{conjmhs}

\begin{conjecture}
\label{conjmhsnongene}
For any $k$-tuple of semisimple conjugacy classes $\mathcal{C}$, we have:
\begin{equation}
\label{mhsconjnongene1}
\dfrac{H_c(\mathcal{M}_{\mathcal{C}},q,-t)}{(qt^2)^{-(\alpha_{\mathcal{C}},\alpha_{\mathcal{C}})}}=\Coeff_{\alpha_{\mathcal{C}}}\left(\Plexp\left(\sum_{\substack{\beta \in \mathcal{H}^*_{\gamma_{\mathcal{C}},\alpha_{\mathcal{C}}}}} \dfrac{(qt^2)\mathbb{H}_{\beta}\left(t\sqrt{q},\frac{1}{\sqrt{q}}\right)}{qt^2-1}y^{\beta}\right)\right).
\end{equation}

\end{conjecture}

\begin{oss}

As mentioned in the introduction, the $-$ sign in the term $H_c(\mathcal{M}_{\mathcal{C}},q,-t)$ of eq.(\ref{mhsconjnongene1}) is due to the  properties of the plethystic exponential $\Plexp$, see for instance \cite[Section 4.3]{davisonpreprojective}. 

In a nutshell, the plethystic exponential $\Plexp:\Q(t)[y_i]]^+_{i \in I} \to 1+\Q(t)[[y_i]]^+_{i \in I}$ can be seen as the decategorification of the symmetric power functor on the category of cohomologically graded and $\N^I$-graded vector spaces.

The $-$ sign comes then from Koszul's sign rule for the cohomological grading. We remark that  Theorem \ref{Epolynomialtheorem} is the specialization at $t=1$ of Conjecture \ref{conjmhsnongene}.
    
\end{oss}

In this chapter we will verify Conjecture \ref{conjmhsnongene} for a certain family of non-generic character stacks.

\vspace{8 pt}

Let $\Sigma=\mathbb{P}^1_{\C}$ (i.e. $g=0$), $k=4$ and $n=2$. Let $D=\{x_1,\dots,x_4\} \subseteq \mathbb{P}^1_{\C}$. For $j=1,\dots,4$, pick $\lambda_j \in \C^*$ with $\lambda_j \neq \pm 1$ and denote by $\mathcal{C}_j$ the conjugacy class of the diagonal matrix $$\begin{pmatrix}
\lambda_j &0\\
0 &\lambda_j^{-1}
\end{pmatrix} .$$

Let $\mathcal{C}$ be the $k$-tuple $\mathcal{C}=(\mathcal{C}_1,\dots,\mathcal{C}_4)$. The variety $X_{\mathcal{C}}$ is therefore $$X_{\mathcal{C}}=\{(X_1,\dots,X_4 ) \in \mathcal{C}_1 \times \cdots \times \mathcal{C}_4 \ | \ X_1X_2X_3X_4 =1\} .$$

Denote by $M_{\mathcal{C}}$ the GIT quotient  $M_{\mathcal{C}}\coloneqq X_{\mathcal{C}}/\!/\Gl_2(\C)$. Recall that the points of $M_{\mathcal{C}}$ are in bijection with the isomorphism classes of semisimple representations of $\pi(\Sigma\setminus D)$ inside $X_{\mathcal{C}}$.

\vspace{10 pt}

The study of the geometry of the character varieties $M_{\mathcal{C}}$ goes back to Fricke and Klein \cite{fricke-klein}, who gave a description of them in terms of cubic surfaces. Denote by $a_i=\lambda_i +\lambda_i^{-1}$.

The character variety $M_{\mathcal{C}}$ is isomorphic to the cubic surface defined by the following equation in $3$ variables $x,y,z$ \begin{equation}
\label{affinecubicsurface}
xyz+x^2+y^2+z^2-(a_1a_2+a_3a_4)x-(a_2a_3+a_1a_4)y-(a_1a_3+a_2a_4)z+a_1a_2a_3a_4+a_1^2+a_2^2+a_3^2+a_4^2-4=0 .\end{equation}

If $\mathcal{C}$ is generic, this description identifies $M_{\mathcal{C}}$ with a smooth (affine) Del Pezzo cubic surface (see \cite[Theorem 6.1.4]{etingof-et-al}), i.e. a smooth cubic projective surface with a triangle cut out of it. The cohomology of this kind of surfaces is well known. In particular, if $\mathcal{C}$ is generic, it holds:
$$H_c(M_{\mathcal{C}},q,t)=q^2t^4+4qt^2+t^2 .$$

If $\mathcal{C}$ is generic, we have that $M_{\mathcal{C}}=[X_{\mathcal{C}}/\PGl_2]$. From Proposition \ref{tensorprodisomgerbe}, we have $$H_c(\mathcal{M}_{\mathcal{C}},q,t)=\dfrac{q^2t^4+4qt^2+t^2}{qt^2-1} .$$

The identity above agrees with Hausel, Letellier, Rodriguez-Villegas Conjecture \ref{conjmhs}, as explained in \cite[Paragraph 1.5]{HA}.

\vspace{12 pt}

Pick now $\lambda_1,\dots,\lambda_4 \in \C^*\setminus\{1,-1\}$ with the following property. For $\epsilon_1,\dots,\epsilon_4 \in \{1,-1\}$ such that $\lambda_1^{\epsilon_1}\cdots\lambda_4^{\epsilon_4}=1$, then either $\epsilon_1=\cdots=\epsilon_4=1$ or $\epsilon_1=\cdots=\epsilon_4=-1$. Notice that in this case, the associated $k$-tuple $\mathcal{C}$ is not generic.

\vspace{8 pt}

In the following section, we will compute the mixed Poincaré series $H_c(\mathcal{M}_{\mathcal{C}},q,t)$ and verify that it respects Conjecture \ref{conjmhsnongene}.

\vspace{8 pt}

For the character stack $\mathcal{M}_{\mathcal{C}}$, the associated quiver $Q=(I,\Omega)$ is the star-shaped quiver with one central vertex and four arrows pointing inwards. We denote the central vertex by $0$ and the other vertices by $[i,1]$ for $i=1,\dots,4$. 

The dimension vector $\alpha_{\mathcal{C}}$ is the dimension vector for $Q$ defined as $(\alpha_{\mathcal{C}})_0=2$ and $(\alpha_{\mathcal{C}})_{[i,1]}=1$ for $i=1,\dots,4$. The quiver $Q$ with the dimension vector $\alpha$ is depicted below.

\begin{center}
\begin{tikzcd}
 &       &1 \arrow[d,""]  &     \\
&1 \arrow[r," "] &2 &1 \arrow[l," "]\\
         &  &1 \arrow[u,""] 
\end{tikzcd}

\end{center}

The associated parameter $\gamma_{\mathcal{C}}$ is given by $$ (\gamma_{\mathcal{C}})_0=(\lambda_1\lambda_2\lambda_3\lambda_4)^{-1}=1 $$
and, for $i=1,\dots,4$ $$(\gamma_{\mathcal{C}})_{[i,1]}=\lambda_i^{2} .$$ Denote by $\beta_1,\beta_2 \in (\N^I)^*$ the elements defined as $(\beta_1)_0=1, (\beta_1)_{[i,1]}=1$ and $(\beta_2)_1=1,(\beta_2)_{[i,1]}=0$ for $i=1,\dots,4$. Notice that it holds $\mathcal{H}^*_{\gamma_{\mathcal{C}},\alpha_{\mathcal{C}}}=\{\alpha,\beta_1,\beta_2\}$. There are equalities $$\mathbb{H}_{\beta_1}\left(t \sqrt{q},\frac{1}{\sqrt{q}}\right)=\mathbb{H}_{\beta_2}\left(t \sqrt{q},\frac{1}{\sqrt{q}}\right)=1 .$$

Conjecture \ref{conjmhsnongene} predicts then the following equality

\begin{equation}
\label{conjmhsp1}
H_c(\mathcal{M}_{\mathcal{C}},q,-t)=\dfrac{qt^2\mathbb{H}_{\alpha_{\mathcal{C}}}\left(t\sqrt{q},\frac{1}{\sqrt{q}}\right)}{qt^2-1}+\dfrac{q^2t^4\mathbb{H}_{\beta_1}\left(t \sqrt{q},\frac{1}{\sqrt{q}}\right)\mathbb{H}_{\beta_2}\left(t \sqrt{q},\frac{1}{\sqrt{q}}\right)}{(qt^2-1)^2}=
\end{equation}
\begin{equation}
\label{conjmhsp12}
=\dfrac{q^2t^4+4qt^2+t^2}{qt^2-1}+\dfrac{q^2t^4}{(qt^2-1)^2}=\dfrac{q^3t^6+4q^2t^4+qt^4-4qt^2-t^2}{(qt^2-1)^2}.
\end{equation}

Since the terms in $t$ of the RHS of eq.(\ref{conjmhsp12}) all have even degrees, in this case Conjecture \ref{conjmhsnongene} is equivalent to verify that
\begin{equation}
\label{conjmhsp111}
H_c(\mathcal{M}_{\mathcal{C}},q,t)=\dfrac{q^3t^6+4q^2t^4+qt^4-4qt^2-t^2}{(qt^2-1)^2}.
\end{equation}

\subsection{Cohomology computations}
\label{cohomologycomputations}
Denote by $\mathcal{M}'_{\mathcal{C}}$ the quotient stack $\mathcal{M}'_{\mathcal{C}}=[X_{\mathcal{C}}/\PGl_2]$. From Proposition \ref{tensorprodisomgerbe}, we have 
\begin{equation}
\label{gerbeidentity}
H_c(\mathcal{M}_{\mathcal{C}},q,t)=\dfrac{H_c(\mathcal{M}'_{\mathcal{C}},q,t)}{qt^2-1} . 
\end{equation}
We can then reduce ourselves 
to compute the  cohomology 
of the stack $\mathcal{M}'_{\mathcal{C}}$. 

\vspace{10 pt}

Inside $X_{\mathcal{C}}$ there is the open (dense) subset which we denote by $X_{\mathcal{C}}^{s} \subseteq X_{\mathcal{C}}$, given by quadruple $(X_1,X_2,X_3,X_4) \in \mathcal{C}_1 \times \cdots \times  \mathcal{C}_4$ corresponding to irreducible representations of $\pi_1(\Sigma \setminus D)$. Recall that $X_{\mathcal{C}}^s$ is smooth (see for example \cite[Proposition 5.2.8]{oblomkov-etal}).

Denote by $\mathcal{N}^s_{\mathcal{C}}$ the quotient stack $[X_{\mathcal{C}}^s/\PGl_2]$. Notice that the action of $\PGl_2$ is schematically free on $X_{\mathcal{C}}^s$
and therefore the stack $\mathcal{N}^s_{\mathcal{C}}$ is an algebraic variety.

\vspace{8 pt}
The non irreducible representations of $X_{\mathcal{C}}$ all have the same semisimplification, up to isomorphism, which corresponds to the point $m \in M_{\mathcal{C}}$, associated to the isomorphism class of the representation 
\begin{equation}
\label{semisimplerep}
m=\left(\begin{pmatrix}\lambda_1 & 0 \\ 0 & \lambda_1^{-1} \end{pmatrix}\\,\begin{pmatrix}\lambda_2 &0 \\ 0 &\lambda_2^{-1}  \\  \end{pmatrix}\\, \begin{pmatrix}\lambda_3 & 0 \\ 0 & \lambda_3^{-1} \end{pmatrix} \\, \begin{pmatrix}\lambda_4  &0\\  0 &\lambda_4^{-1} \end{pmatrix}\right) .\end{equation}

We denote by $O \subseteq X_{\mathcal{C}}$ the closed $\Gl_2(\C)$-orbit associated to $m$.  A representation $x \in X_{\mathcal{C}}$ which is neither irreducible nor semisimple, i.e. which belongs neither to $X_{\mathcal{C}}^s$ nor to $O$, can be of the following two types. Either $x$ is isomorphic to a quadruple of the form \begin{equation}
\label{matrices+}
m^{+}_{a,b,c} \coloneqq \left(\begin{pmatrix}\lambda_1 & 0 \\ 0 & \lambda_1^{-1} \end{pmatrix}\\,\begin{pmatrix}\lambda_2 &a \\ 0 &\lambda^{-1}_2  \\  \end{pmatrix}\\, \begin{pmatrix}\lambda_3 & b \\ 0 & \lambda^{-1}_3 \end{pmatrix} \\, \begin{pmatrix}\lambda_4 &c\\0  &\lambda^{-1}_4 \end{pmatrix}\right) \end{equation} with $(a,b,c) \in \C^3\setminus\{(0,0,0)\}$ and \begin{equation}
\label{conditionsmatrix}\lambda_1\lambda_2\lambda_3 c +\lambda_1\lambda_2 \mu_4b+\lambda_1\lambda_2\lambda_3 c=0 \end{equation}  or of the form $$m^{-}_{a,b,c}\coloneqq \left(\begin{pmatrix}\lambda_1 & 0 \\ 0 & \lambda^{-1}_1 \end{pmatrix}\\,\begin{pmatrix}\lambda_2 &0 \\ a &\lambda^{-1}_2  \\  \end{pmatrix}\\, \begin{pmatrix}\lambda_3 & 0 \\ b & \lambda^{-1}_3 \end{pmatrix} \\, \begin{pmatrix}\lambda_4  &0\\ c &\lambda^{-1}_4 \end{pmatrix}\right) $$ 
with $(a,b,c) \in \C^3\setminus\{(0,0,0)\}$ and \begin{equation}\label{conditionsmatrix1}\lambda_4\lambda_1^{-1}\lambda_3 a+\lambda_1^{-1}\lambda_2^{-1}\lambda_4 b+\lambda_1^{-1}\lambda_2^{-1}\lambda_3^{-1}c=0 .\end{equation}

We denote by $Z_{\mathcal{C}}^+ \subseteq X_{\mathcal{C}}$ and by $Z_{\mathcal{C}}^- \subseteq X_{\mathcal{C}}$
the locally closed subsets of representations isomorphic to elements of the form $m_{(a,b,c)}^+$ or $m_{(a,b,c)}^-$
 for some $(a,b,c) \in \C^3\setminus\{(0,0,0)\}$ respecting the conditions of eq.(\ref{conditionsmatrix}), eq.(\ref{conditionsmatrix1}) respectively.

\subsubsection{Cohomology of the character variety in the non-generic case}

As mentioned before, the variety $M_{\mathcal{C}}$ is a cubic surface defined by  equation  (\ref{affinecubicsurface}). Denote by $\overline{M}_{\mathcal{C}} \subseteq \mathbb{P}^3_{\C}$ the associated projective cubic surface. Notice that $\overline{M}_{\mathcal{C}}$ is obtained by adding to $M_{\mathcal{C}}$ the triangle at infinity $xyz=0$, which we will denote by $U \subseteq \overline{M}_{\mathcal{C}}$.

Unlike the case in which $\mathcal{C}$ is generic, for our choice of quadruples the surface $\overline{M}_{\mathcal{C}}$ is singular, with $m$ being is its only singular point. We have moreover an isomorphism $\mathcal{N}_{\mathcal{C}}^s \cong M_{\mathcal{C}}\setminus\{m\}$.

It is a well known result (see for example \cite{resolution-surfaces}) that  for such a singular  cubic surface $\overline{M}_{\mathcal{C}}$, there exists a resolution of singularities $$f:\widetilde{M_{\mathcal{C}}} \to \overline{M}_{\mathcal{C}}$$ such that $f^{-1}(m) \cong \mathbb{P}^1_{\C}$ and $f$ is an isomorphism over $\overline{M}_{\mathcal{C}}\setminus\{m\}$, i.e. $\overline{M}_{\mathcal{C}}\setminus\{m\} \cong \widetilde{M}_{\mathcal{C}}\setminus f^{-1}(\{m\})$.

Moreover, it is known that $\widetilde{M_{\mathcal{C}}}$ is the blow-up of $\mathbb{P}^2_{\C}$ at $6$ points. There is thus an equality $$H_c(\widetilde{M}_{\mathcal{C}},q,t)=q^2t^4+7qt^2+1 .$$ Using the long exact sequence in compactly supported cohomology for the open-closed decomposition $\widetilde{M}_{\mathcal{C}}=f^{-1}(\overline{M}_{\mathcal{C}}\setminus\{m\}) \bigsqcup f^{-1}(m)$, we deduce that we have $$H_c(\overline{M}_{\mathcal{C}}\setminus\{m\},q,t)=H_c(f^{-1}(\overline{M}_{\mathcal{C}}\setminus\{m\}),q,t)=q^2t^4+6qt^2$$ and so  that we have $$H_c(\overline{M}_{\mathcal{C}},q,t)=q^2t^4+6qt^2+1 .$$

It is not difficult to check that the compactly supported Poincaré polynomial of $U$ is $H_c(U,q,t)=3qt^2+t+1$. Applying the long exact sequence in compactly-supported cohomology for the open-closed decomposition $\overline{M}_{\mathcal{C}}=M_{\mathcal{C}} \bigsqcup U$ we find that we have
\begin{equation}
\label{cohoGITquoteq}
H_c(M_{\mathcal{C}},q,t)=q^2t^4+3qt^2+t^2.
\end{equation}

From eq.(\ref{cohoGITquoteq}), using the long exact sequence for the open-closed decomposition $M_{\mathcal{C}}=(M_{\mathcal{C}}\setminus\{m\})\bigsqcup \{m\}$  we deduce that we have:
\begin{equation}
\label{cohoGITquoteqsimple}
H_c(\mathcal{N}^s_{\mathcal{C}},q,t)=H_c(M_{\mathcal{C}}\setminus\{m\},q,t)=q^2t^4+3qt^2+t^2+t. 
\end{equation}

\subsubsection{Cohomology of the character stack in the non-generic case}

We introduce the following notations. Let $Y_{\mathcal{C}}=X_{\mathcal{C}}\setminus O$ and $\mathcal{N}_{\mathcal{C}}=[Y_{\mathcal{C}}/\PGl_2]$. The action of $\PGl_2$ on $Y_{\mathcal{C}}$ is set-theoretically free so that $\mathcal{N}_{\mathcal{C}}$ is at least an algebraic space. 
 
\vspace{8 pt}

Notice that there is an isomorphism $[O/\PGl_2] \cong B\mathbb{G}_m$ and an open-closed decomposition $$\mathcal{M}'_{\mathcal{C}}=\mathcal{N}_{\mathcal{C}} \bigsqcup [O/\PGl_2] .$$ Applying the long-exact sequence for compactly supported cohomology for the open-closed decomposition above and knowing that $H^*_c(B \mathbb{G}_m)$ is concentrated in strictly negative even degrees, we obtain
\begin{equation}
\label{reductioncohomology}
H_c(\mathcal{M}'_{\mathcal{C}},q,t)=H_c(\mathcal{N}_{\mathcal{C}},q,t)+H_c(B \mathbb{G}_m,q,t)=H_c(\mathcal{N}_{\mathcal{C}},q,t)+\dfrac{1}{qt^2-1}.
\end{equation}
\vspace{10 pt}

We have thus reduced ourselves to compute the mixed Poincaré series $H_c(\mathcal{N}_{\mathcal{C}},q,t)$. Let $Y_{\mathcal{C}}^+=Y_{\mathcal{C}}\setminus Z_{\mathcal{C}}^-$ and $Y_{\mathcal{C}}^-=Y_{\mathcal{C}}\setminus Z_{\mathcal{C}}^+$. Notice that, a priori, $Y_{\mathcal{C}}^+,Y^-_{\mathcal{C}}$ are only constructible subsets of $X_{\mathcal{C}}$. Therefore we don't have a good definition of the quotient stacks $[Y_{\mathcal{C}}^+/\PGl_2],[Y_{\mathcal{C}}^-/\PGl_2]$ neither of their cohomology.

\vspace{4 pt}
To solve this problem, we start by the following preliminary Lemma. 

\begin{lemma}
\label{lemmaquotstackisom}
Let $G$ be a linear algebraic group over $\C$ acting on the left on a $\C$-scheme $X$. Let $H \leq G$ be a closed subgroup. Suppose that there exists a $G$-equivariant map $p:X \to G/H$, where $G$ acts on $G/H$ by left multiplication. Denote by $X_H=p^{-1}(eH)$.

The group $H$ acts on $X_H$ and there is an isomorphism of quotient stacks
\begin{equation}
\label{quotientstackisom1}
[X/G] \cong [X_H/H].
\end{equation}

Moreover, if $X$ is an affine variety and $G,H$ are reductive, there is an isomorphism of varieties:
\begin{equation}
\label{quotientstackisom2}
X/\!/G \cong X_H/\!/H.
\end{equation}
\end{lemma}

\begin{proof}
Notice firstly that, if $X$ is affine and $G,H$ are reductive, the isomorphism (\ref{quotientstackisom2}) is implied by the isomorphism (\ref{quotientstackisom1}) as the varieties $X_H /\!/H,X/\!/G$ are good moduli spaces for the stacks $[X_H/H],[X/G]$ respectively (see \cite[Remark 4.8]{alper}).
We now prove isomorphism (\ref{quotientstackisom1}).

\vspace{8 pt}

Notice  that in general there is always a map $\alpha:[X_H/H] \to [X/G]$. We must construct an inverse $\beta:[X_H/H] \to [X/G]$.

Fix a scheme $S$ and recall that the objects of the groupoid $[X/G](S)$ are  couples $(P,\phi)$ where $P \to S$ is a principal $G$-bundle and $\phi:P \to X$ is a $G$-equivariant map and similarly for $[X_H/H](S).$ We define $\beta(P,\phi)\coloneqq(P_H,\phi_H)$ to fit in the following diagram, where both the squares are cartesian:
\begin{center}
\begin{tikzcd}
P_H \arrow[r,"\phi_H"] \arrow[d,""] &X_H \arrow[d,""] \arrow[r,""] &eH
\arrow[d,""] \\
P \arrow[r,"\phi"] &X \arrow[r,"\pi"] &G/H.
\end{tikzcd}
\end{center}

It can be checked that $P_H$ is a principal $H$-bundle over $S$ and $\phi_H$ is $H$-equivariant, so that $\beta$ actually defines a morphism $$\beta:[X/G]\to [X_H/H] .$$ The morphism $\beta$ is an inverse to $\alpha$.
\end{proof}

\vspace{12 pt}

We will apply Lemma \ref{lemmaquotstackisom} above in the case where $X=X_{\mathcal{C}}$, $G=\PGl_2$ and $H \subseteq \PGl_2$ is the maximal torus of diagonal matrices as follows. In the following, we identify $H \cong \mathbb{G}_m$, via the map $\mathbb{G}_m \to \PGl_2$, which sends  $z \in \C^*$ to the class of $\begin{pmatrix}
z &0\\
0 &1
\end{pmatrix}$.

\vspace{8 pt}

Recall that there is an isomorphism  $\mathcal{C}_1\cong G/H$. Via this latter isomorphism,  the projection on the first factor induces a $G$-equivariant morphism $$ p:X_{\mathcal{C}} \to G/H \cong \mathcal{C}_1$$ $$(X_1,X_2,X_3,X_4) \to X_1 .$$ 

Notice that $$(X_{\mathcal{C}})_H=\Biggl\{X_2 \in \mathcal{C}_2, X_3 \in \mathcal{C}_3, X_4 \in \mathcal{C}_4 \ | \ X_2X_3X_4=\begin{pmatrix}\lambda_1^{-1} &0\\
0 &\lambda_1 \end{pmatrix}\Biggr\} .$$

Denote by $(M_{\mathcal{C}})_H \coloneqq (X_{\mathcal{C}})_H/\!/H$. Lemma \ref{lemmaquotstackisom} implies that there is an isomorphism $$(M_{\mathcal{C}})_H \cong M_{\mathcal{C}} .$$ We use similar notations for $(\mathcal{N}_{\mathcal{C}})_H,(\mathcal{N}^s_{\mathcal{C}})_H$. Reapplying Lemma \ref{lemmaquotstackisom}, we see that there is an isomorphism $(\mathcal{N}^s_{\mathcal{C}})_H \cong \mathcal{N}^s_{\mathcal{C}}$. In particular,  
\begin{equation}
\label{isomorphismquotstacksimple}
H_c((\mathcal{N}_{\mathcal{C}}^s)_H,q,t)=q^2t^4+3qt^2+t^2+t.
\end{equation}

\vspace{8 pt}

Consider now the character $\theta^+:H=\mathbb{G}_m \to \mathbb{G}_m$ given by $\theta^+(z)=z$. The character $\theta$ induces a linearization of the $H$-action on the affine variety $(X_{\mathcal{C}})_H$ (see for example \cite[Section 2]{king}). Using Mumford's criterion (see \cite[Proposition 2.5]{king}), we check below that the semistable points $(X_{\mathcal{C}})_H^{ss,\theta^+}$ are given by $$(X_{\mathcal{C}})_H^{ss,\theta^+}=(Y^+_{\mathcal{C}})_H .$$ 

In particular, $(Y^+_{\mathcal{C}})_H$ is an open subset of $(X_{\mathcal{C}})_H$ and it is thus an algebraic variety. We denote by $(\mathcal{N}_{\mathcal{C}}^+)_H$ the quotient stack $(\mathcal{N}_{\mathcal{C}}^+)_H=[Y_{\mathcal{C}}^+/H]$.

Since the action of $H$ on $(Y^+_{\mathcal{C}})_H$ is free, the stack $(\mathcal{N}_{\mathcal{C}}^+)_H$ is an algebraic variety. Denote by $f^+:(\mathcal{N}_{\mathcal{C}}^+)_H \to (M_{\mathcal{C}})_H$ the canonical (proper) map.

\vspace{10 pt}

We have indeed four type of points inside $(X_{\mathcal{C}})_H$:

\begin{itemize}

\item Notice that $O \cap (X_{\mathcal{C}})_H$ is the singleton $\{m\}$, corresponding to the quadruple (\ref{semisimplerep}). The point $m$, being a $\mathbb{G}_m$ fixed point, is unstable. Indeed, considering the $1$-parameter subgroup $\lambda: \mathbb{G}_m \to \mathbb{G}_m$ given by $\lambda(z)=z^{-1}$, we have $\langle \theta^+,\lambda \rangle=-1 <0$ while it exists $\displaystyle \lim_{t \to 0}\lambda(t) \cdot m=m$.

\item The points of $(X_{\mathcal{C}}^s)_H$ are stable. Each $x \in (X_{\mathcal{C}})_H$ corresponds to an irreducible representation. For a $1$-parameter subgroup $\lambda: \mathbb{G}_m \to \mathbb{G}_m$, the limit $\displaystyle \lim_{t \to 0} \lambda(t) \cdot x$ exists if and only if $\lambda$ is trivial, i.e. $\langle \theta^+,\lambda \rangle=0$.  

\item The points of $(Z_{\mathcal{C}}^+)_H$ are semistable. Notice that $(Z_{\mathcal{C}}^+)_H$ is given by points of the form $m_{(a,b,c)}^+$ as in eq.(\ref{matrices+}), for $(a,b,c) \in \C^3\setminus\{(0,0,0)\}$ which respects eq.(\ref{conditionsmatrix}). 

For $\lambda:\mathbb{G}_m \to \mathbb{G}_m$ given by $\lambda(t)=t^n$ for $n \in \Z$ and $t \in \C^*$, we have $$\lambda(t) \cdot m_{(a,b,c)}^+=m_{(t^na,t^nb,t^nc)} .$$ In particular, we see that the limit $\displaystyle \lim_{t \to 0}\lambda(t) \cdot m_{(a,b,c)}^+$ exists (and it is given by $m$) if and only if $n \geq 0$, i.e. if and only if $\langle \theta^+,\lambda \rangle \geq 0$. 

\item By a similar reasoning, the points of $(Z_{\mathcal{C}}^-)_H$ are unstable. 

\end{itemize}

\vspace{10 pt}

 We have that $(f^+)^{-1}(m)=(Z_{\mathcal{C}}^+)_H/H$. From the description of the elements of $X_{\mathcal{C}}$ given at the beginning of \cref{cohomologycomputations}, we see that $(Z_{\mathcal{C}}^+)_H$ is isomorphic to $\C^2\setminus\{(0,0)\}$. Via this identification $\mathbb{G}_m$ acts on $\C^2\setminus\{(0,0)\}$ by scalar multiplication on both coordinates. We have therefore: $$(Z_{\mathcal{C}}^+)_H/H \cong (\C^2\setminus\{(0,0)\})/\mathbb{G}_m=\mathbb{P}^1_{\C} .$$

\vspace{10 pt}

Consider now the Leray spectral sequence for compactly supported cohomology $$E^{p,q}_2:H^p_c((M_{\mathcal{C}})_H,R^qf^+_*\Q) \Rightarrow H^{p+q}_c((\mathcal{N}_{\mathcal{C}}^+)_H,\Q) .$$

Notice that  $R^qf^+_*\Q \neq 0$ if and only if $q=0,2$. More precisely, we have $f^+_*\Q=\Q$ and $R^2f^+_*\Q=(i_m)_*\Q$, where $i_m$ is the closed embedding $$i_m:\{m\} \to (M_{\mathcal{C}})_H .$$ Recall that the differential maps of the spectral sequence go in the direction $$d_{r}^{p,q}:E^{p,q}_r \to E^{p+r,q-r+1}_r .$$

As $R^qf^+_*\Q$ is  $0$ for odd $q$, the differential $d_2^{p,q}$ is the zero map for each $p,q$ and therefore we have $E^{p,q}_3=E^{p,q}_2$ for each $p,q$. Moreover, the differentials on the third page go in the direction $d^{p,q}_3:E^{p,q}_3 \to E^{p+3,q-2}_3$ and if $q \neq 0,2$, the vector space $E^{p,q}_3$ is equal to $0$.

If $q=0$, we have $E^{3,-3}_3=\{0\}$ and so $d^{p,q}_3=0$. Lastly, if $q=2$, we have $E^{p,q}_3=\{0\}$ if $p \geq 1$ and if $p=0$, we have $E^{3,0}_3=H^3_c((M_{\mathcal{C}})_H,\Q)=\{0\}$.

We deduce therefore that the differential maps $d^{p,q}_3$ are all zero. In a similar way, it is possible to verify that $d^{p,q}_r=0$ if $r \geq 2$, for any $p,q$ and so that the spectral sequence collapses on the second page. 

\vspace{2 pt}

For each $n$, there is therefore an equality $$H^{n}_c((\mathcal{N}_{\mathcal{C}}^+)_H,\Q)=\bigoplus_{p+q=n} H^p_c((M_{\mathcal{C}})_H,R^qf^+_*\Q) .$$

From the description of the sheaves $R^qf^+_*\Q$ given above, we deduce  that we have \begin{equation}
\label{cohomologypositivepiece}
H_c((\mathcal{N}_{\mathcal{C}}^+)_H,q,t)=q^2t^4+4qt^2+t^2
\end{equation}

\vspace{12 pt}

A similar reasoning can be applied to the opposite linearization, induced by the character  $\theta^-:\mathbb{G}_m \to \mathbb{G}_m$ given by $\theta^-(z)=z^{-1}$. In this case, in a similar way we can argue that  the semistable points $(X_{\mathcal{C}})_H^{ss,\theta^-}$ are given by $(Y^-_{\mathcal{C}})_H$.

For the corresponding quotient $(\mathcal{N}_{\mathcal{C}}^-)_H$ there is therefore an equality \begin{equation}
    \label{cohomologynegativepiece}
H_c((\mathcal{N}_{\mathcal{C}}^-)_H,q,t)=q^2t^4+4qt^2+t^2 .\end{equation}

\vspace{8 pt}

Denote now by $j^+,j^-$ the open embeddings $j^+:(\mathcal{N}_{\mathcal{C}}^+)_H \to (\mathcal{N}_{\mathcal{C}})_H$ and $j^+:(\mathcal{N}_{\mathcal{C}}^-)_H \to (\mathcal{N}_{\mathcal{C}})_H$ and by $j$ the open embedding $(\mathcal{N}^s_{\mathcal{C}})_H \to (\mathcal{N}_{\mathcal{C}})_H$. Notice that there is a short exact sequence of sheaves on $(\mathcal{N}_{\mathcal{C}})_H$

\begin{center}
\begin{tikzcd}
0 \arrow[r,""] &j_!\C \arrow[r,""] &j^+_!\C \oplus j_!^-\C \arrow[r,""] &\C \arrow[r,""] &0  
\end{tikzcd}
\end{center}
and therefore an associated long exact sequence in compactly supported cohomology
\begin{center}
\begin{tikzcd}
H^{i-1}_c((\mathcal{N}_{\mathcal{C}})_H) \arrow[r,""] &H^i_c((\mathcal{N}_{\mathcal{C}}^s)_H) \arrow[r,""] &H^i_c((\mathcal{N}_{\mathcal{C}}^+)_H) \oplus H^i_c((\mathcal{N}_{\mathcal{C}}^-)_H) \arrow[r,""] &H^i_c((\mathcal{N}_{\mathcal{C}})_H)  
\end{tikzcd}
\end{center}

From Lemma \ref{lemmaquotstackisom}, we deduce that there is an  isomorphism $(\mathcal{N}_{\mathcal{C}})_H \cong \mathcal{N}_{\mathcal{C}}$. From the long exact sequence above and eq.(\ref{cohomologypositivepiece}), eq.(\ref{cohomologynegativepiece}), eq.(\ref{isomorphismquotstacksimple}), it is therefore not difficult to show that \begin{equation}
\label{eqarrogance}
H_c(\mathcal{N}_{\mathcal{C}},q,t)=q^2t^4+5qt^2+t^2+1.
\end{equation}

Plugging this result into eq.(\ref{reductioncohomology}) and using identity (\ref{gerbeidentity}), we obtain finally identity (\ref{conjmhsp12}).

\end{document}